\documentclass[a4paper,11pt,reqno]{amsart}
\usepackage{epsfig,url,paralist}
\usepackage{amssymb}
\usepackage[normalem]{ulem}
\usepackage{fullpage,dynkin-diagrams}
\usepackage{amsmath}
\usepackage{tikz}
\usepackage{fullpage}
\usepackage{hyperref,verbatim}
\usepackage[capitalize]{cleveref}
\usepackage{graphicx}
\usepackage{setspace,multirow}
\usepackage{enumitem,lineno}
\setlist{nolistsep}
\usepackage{lscape}
\usepackage{calrsfs}       
\usepackage{booktabs}     

\newtheorem{theo}{Main Result}

\numberwithin{equation}{section}

\newtheorem{theorem}{Theorem}[section]

\newtheorem{theorem*}{Main Result}
\newtheorem{lem}[theorem]{Lemma}
\newtheorem{fact}[theorem]{Fact}
\newtheorem{coro}[theorem]{Corollary}
\newtheorem{cor*}[theorem*]{Corollary}

\newtheorem{prop}[theorem]{Proposition}
\theoremstyle{definition}
\newtheorem{defn}[theorem]{Definition}

\newtheorem{remark}[theorem]{Remark}

\numberwithin{theorem}{section}
\usepackage{fancyhdr}
\def\<{\langle}
\def\>{\rangle}
\newcommand{\Res}{\mathsf{Res}}

\newcommand{\cL}{\mathcal{L}}
\newcommand{\pperp}{\perp\hspace{-0.15cm}\perp}

\newcommand{\proj}{\mathsf{proj}}

\newcommand{\SL}{\mathsf{SL}}

\newcommand{\norm}{\mathsf{n}}

\renewcommand{\gcd}{\mathsf{gcd}}
\newcommand{\PG}{\mathsf{PG}}

\newcommand{\id}{\mathsf{id}}
\newcommand{\K}{\mathbb{K}}

\newcommand{\A}{\mathbb{A}}
\newcommand{\J}{\mathbb{J}}

\newcommand{\cM}{\mathcal{M}}

\begin{document}
\title{Lines and opposition in Lie incidence geometries of exceptional type}
\author{Sira Busch}
\curraddr{Mathematisches Institut\\
Universit\"at M\"unster\\
Orl\'eans-Ring 10\\
D--48149 M\"unster\\
GERMANY}
\email{s\_busc16@uni-muenster.de}
\address{ORCID: 0009-0009-0939-6543}
\thanks{The first author is funded by the Claussen-Simon-Stiftung and by the Deutsche Forschungsgemeinschaft (DFG, German Research Foundation) under Germany's Excellence Strategy EXC 2044 --390685587, Mathematics M\"unster: Dynamics--Geometry--Structure. This work is part of the PhD project of the first author.}
 \author{Hendrik Van Maldeghem}
 \curraddr{Department of Mathematics, Computer Science and Statistics\\
Ghent University\\
Krijgslaan 299-S9\\
B--9000 Ghent\\
BELGIUM}
\email{Hendrik.VanMaldeghem@UGent.be}
\address{ORCID: 0000-0002-8022-0040} 
\maketitle
\begin{abstract}
We characterise sets of points of exceptional Lie incidence geometries, that is, the natural geometries arising from spherical buildings of exceptional types $\mathsf{F_4}$, $\mathsf{E_6}$, $\mathsf{E_7}$, $\mathsf{E_8}$ and $\mathsf{G_2}$, that form a line using the opposition relation. 
With that, we obtain a classification of so-called ``geometric lines'' in many of these geometries. 
Furthermore, our results lead to a characterisation of geometric lines in finite exceptional Lie incidence geometries as minimal blocking sets,
that is, point sets of the size of a line admitting no object opposite to all of their members, in most cases, and we classify all exceptions. 
As a further consequence, we obtain a characterisation of automorphisms of exceptional spherical buildings as certain opposition preserving maps.
\end{abstract}
\setcounter{tocdepth}{3}
\tableofcontents

\section{Introduction}
The theory of Tits-buildings plays an important role in group theory, in particular in the theory of (semisimple) algebraic groups, in finite group theory, in the theory of topological (simple) groups, and so on. 
The intricate structure of --- especially the exceptional --- spherical buildings leads to the introduction of seemingly less complex, but certainly more accessible, point-line geometries describing essentially the same object. 
These geometries are usually called \emph{Lie incidence geometries}. 
The procedure to construct such a geometry is nowadays standard: 
for a spherical building $\Delta$, say of type $\mathsf{X}_n$, pick a type $i$, consider all vertices of $\Delta$ of type $i$ and call them \emph{points}. 
Pick a chamber $C$ and delete the vertex of type $i$ to obtain a panel of cotype $i$. 
Call the set of vertices of type $i$ completing this panel again to a chamber a \emph{line}. 
Vary $C$ over all chambers. 
Then this system of points and lines is a \emph{Lie incidence geometry of type $\mathsf{X}_{n,i}$}. 
It now turns out (as follows from \cite[\S 3]{Coh:86}) that the full automorphism group of this point-line geometry coincides with the full automorphism group of $\Delta$ that preserves the type $i$. 
In \cite{Kas-Mal:13}, Kasikova and the second author investigate the interaction of (this definition of) lines and the opposition relation in $\Delta$. 
The opposition relation is something typical for spherical buildings, and it owes its existence to the strong relation with finite Coxeter groups, 
in which there is a so-called ``longest word''. 
The main result of \cite{Kas-Mal:13} characterises the lines of many Lie incidence geometries in terms of this opposition relation. 
This led to the introduction of the notion of a ``geometric line'', which is a set of points with the property that an arbitrary object of the underlying spherical building is either opposite none of its elements, or not opposite exactly one of its members. Lines are the standard examples of geometric lines. 
We observe that, in particular, a geometric line does not admit an object opposite all of its members. 
A set of points with the latter property can be viewed as a \emph{blocking set (of points)}. 
Blocking sets are a popular subject in finite geometry, both because they have many applications and because they have a great auxiliary value. 
A line is always a blocking set, and in \cite{Bus-Mal:24} the authors proved that in odd characteristic, geometric lines in finite classical Lie incidence geometries are the only minimal blocking sets. 
In characteristic~$2$ they found counterexamples. 
Their proof, however, makes very little use of the definition of a geometric line; 
the equivalence in odd characteristic just came out from the classification of minimal blocking sets, 
which uses a variety of tools from classical finite geometry. 

In the present paper, the primary aim is to classify minimal blocking sets in many finite Lie incidence geometries of exceptional type. 
However, unlike in \cite{Bus-Mal:24}, in many cases it will turn out to be beneficial to do this by showing directly the equivalence to geometric lines; 
the exceptions in characteristic~$2$ become also apparent in this approach. 
We then either appeal to the classification of geometric lines as provided in \cite{Kas-Mal:13}, 
or we classify them ourselves (if not available in \cite{Kas-Mal:13}). 
This way, we lay the foundations to study more intensively blocking sets in finite exceptional geometries.

More exactly, with the notation introduced in \cref{prelim}, we will show the following theorems.

\begin{theo}\label{A}
Let $\mathsf{X}_n\in\{\mathsf{E_6,E_7,E_8,F_4,G_2}\}$ and $1\leq i\leq n$, with $i\notin \{2,4,5\}$ if $n=7$ and $i\in \{7,8\}$ if $n=8$. 

If in a thick finite Moufang spherical building $\Delta$ of type $\mathsf{X}_n$, $n\geq 2$, the panels of cotype $\{i\}$ are $s$-thick (that is, every panel of cotype $\{i\}$ is contained in precisely $s+1$ chambers), then every set of $s+1$ vertices of type $i$ of $\Delta$ admits a common opposite vertex except precisely in the following five cases.
\begin{compactenum}[$(1)$]
\item The $s+1$ vertices form a line in the corresponding Lie incidence geometry of type $\mathsf{X}_{n,i}$.
\item $\Delta$ is a split building of type $\mathsf{F_4}$, the type $i$ corresponds to a short root in the underlying root system (so $i\in\{3,4\}$), and the $s+1$ vertices correspond to a hyperbolic line in a (thick) symp of the corresponding Lie incidence geometry of type $\mathsf{F}_{4,i}$ isomorphic to a symplectic polar space. 
\item $\Delta$ has type $\mathsf{F_4}$, has residues isomorphic to Hermitian generalised quadrangles of order $(s,\sqrt{s})$, and the $s+1$ vertices form an ovoid in a (symplectic) subquadrangle of order $(\sqrt{s},\sqrt{s})$. Here, $s$ is even and $i=2$.
\item $\Delta$ is a split building of type 
$\mathsf{G_2}$ and the $s+1$ vertices form a hyperbolic (or ideal) line in the corresponding Lie incidence geometry of type $\mathsf{G_{2,2}}$, which is a split Cayley hexagon.
\item $\Delta$ is a building of type 
$\mathsf{G_2}$ in characteristic $2$ and the $s+1$ vertices form a distance-$3$ trace in the corresponding Lie incidence geometry of type 
$\mathsf{G_{2,2}}$, which is either a split Cayley hexagon, or a twisted triality hexagon of order $(s,\sqrt[3]{s})$.
\end{compactenum}
\end{theo}

\cref{A} will follow
\begin{compactenum}[$-$] 
	\item from \cref{E61points} for types $\mathsf{E_{6,1}}$ and $\mathsf{E_{6,6}}$; 
	\item from \cref{E61lines} for types $\mathsf{E_{6,3}}$ and $\mathsf{E_{6,5}}$; 
	\item from \cref{E77points} for type $\mathsf{E_{7,7}}$; 
	\item from \cref{E77lines} for type $\mathsf{E_{7,6}}$; 
	\item from \cref{longrootlemma}, \cref{kasmal} and \cref{geomlinesexc} for types $\mathsf{E_{6,2},E_{7,1},E_{8,8},F_{4,1}}$ and $\mathsf{F_{4,4}}$; 
	\item from \cref{proplinesF4}, \cref{F44exception} and \cref{geomlinegrass} for types $\mathsf{E_{6,4},E_{7,3},E_{8,7},F_{4,2}}$ and~$\mathsf{F_{4,3}}$; 
	\item from \cref{BShex} for types $\mathsf{G_{2,1}}$ and $\mathsf{G_{2,2}}$.
\end{compactenum}

\begin{theo}\label{B}
Let $\mathsf{X}_n\in\{\mathsf{E_6,E_7,E_8,F_4,G_2}\}$ and $1\leq i\leq n$, with $i\notin \{2,4,5\}$ if $n=7$ and $i\in \{7,8\}$ if $n=8$. 

Then a geometric line of the Lie incidence geometry of type $\mathsf{X}_{n,i}$ associated to a thick Moufang spherical building $\Delta$ of type $\mathsf{X}_n$, $n\geq 2$, is one of the following.
\begin{compactenum}[$(1)$]
\item A line of the Lie incidence geometry.
\item A hyperbolic line in a symp of the Lie incidence geometry whenever this symp is a symplectic polar space (and this happens (only) in the split case for types $\mathsf{F_{4,3}}$ and $\mathsf{F_{4,4}}$). 
\item A hyperbolic (or ideal) line of a split Cayley hexagon.
\item A distance-$3$ trace of a split Cayley hexagon over a perfect field of characteristic~$2$.
\end{compactenum}
\end{theo}

Naturally, the Moufang condition in the previous theorems is only necessary for type $\mathsf{G}_2$.

\cref{B} follows from \cite[Corollary~5.6]{Kas-Mal:13} for the types $\mathsf{E_{6,1},E_{6,6}}$ and $\mathsf{E_{7,7}}$.
It will follow \begin{compactenum}[$-$] \item from \cref{GLE63} for types $\mathsf{E_{6,3}}$ and $\mathsf{E_{6,5}}$; \item from \cref{GLE76} for type $\mathsf{E_{7,6}}$; \item from \cref{kasmal} and \cref{geomlinesexc} for types $\mathsf{E_{6,2},E_{7,1},E_{8,8},F_{4,1}}$ and $\mathsf{F_{4,4}}$; \item from \cref{geomlinegrass} for types $\mathsf{E_{6,4},E_{7,3},E_{8,7},F_{4,2}}$ and $\mathsf{F_{4,3}}$; \item from \cref{geomlineshex} for types $\mathsf{G_{2,1}}$ and $\mathsf{G_{2,2}}$.
\end{compactenum}

As an application, we deduce that certain opposition preserving maps and transformations of a Lie incidence geometry are actually (bijective) collineations.
This characterises collineations using opposition.
We refer to \cref{anappl} for the exact statements. 

As a further motivation, we note that some of the results obtained in the present paper are used in \cite{Bus-Sch-Mal:24} and \cite{Bus-Mal:25} to determine projectivity groups.

We now comment on the relation with opposition of \emph{chambers} in buildings and provide some more background. Buildings are combinatorial structures introduced by Jacques Tits \cite{Tits:74} as a geometric interpretation of simple groups of Lie type. To each building is associated a \emph{Weyl group $W$}, which abstractly is a Coxeter group, furnished with a set of (standard) generators, and which describes a $W$-valued distance between chambers. When that Coxeter group is finite, then we are dealing with a  \emph{spherical building}. The existence of a \emph{longest word} in a finite Coxeter group translates to the existence of pairs of so-called \emph{opposite chambers} in a spherical building---such chambers are at maximal distance apart from each other, as opposed to 
\emph{adjacent chambers}, which are at minimal non-trivial distance from each other (their distance is a standard generator of $W$). Obviously, opposition is determined by adjacency, and adjacency determines the whole building. Remarkably, Abramenko and the second author showed in \cite{Abr-Mal:00} that opposition of chambers also determines the adjacency. In fact, they provided a very simple criterion for two chambers of a thick spherical building to be adjacent in terms of opposition: chambers $C$ and $C'$ are adjacent if, and only if, there exists a third chamber $C''$ (which is then also adjacent to both $C$ and $C'$ if the condition is satisfied) such that no chamber is opposite exactly one of $C,C',C''$. The interest in such a result came from the theory of twin buildings. In such buildings, which are not necessarily spherical, the notion of opposition between chambers of different halves of the twinnings is ingrained in the definition---and the above result showed that the opposition between chambers of the halves of a twin building determines the structure of each half. Naturally, the question whether opposition of vertices of spherical buildings also determines adjacency of vertices in such buildings arose. One might have hoped that the results for chambers would have implied similar results for vertices (of fixed type), but Kasikova and the second author showed in \cite{Kas-Mal:13} that this is not true since they found counterexamples. Not only to the analogue to the above criterion involving three chambers, but also to the analogue to the more global characterisation of panels as sets of chambers with the property that each chamber is either opposite none of that set, or not opposite exactly one of that set. Applied to vertices of certain type $i$, the hope was to characterise the lines of the so-called $i$-Grassmannian Lie incidence geometry---all the more since these lines are determined unambiguously by panels---but as mentioned there were counter examples. However, as mentioned above, this gave rise to the notion of a \emph{geometric line} and these were subsequently classified for a number of spherical buildings and types $i$---but not for all. Hence part of this paper can be seen as a complement to \cite{Kas-Mal:13}. Worth noting is that the above criterion involving the three chambers $C,C',C''$ is the basis for the notion of \emph{round-up triple}, first introduced in \cite{Kas-Mal:13}, and again heavily used in the current paper; we introduce this notion in \cref{oppprojsec}.

\textbf{The paper is structured as follows.} 
We provide preliminary information in \cref{prelim} about the geometries we will work with, and about their well-known properties. We also derive many new properties that are not available in the literature. 
Towards the end of that section, we state a corollary to a result of Tits that enables one to construct blocking sets in buildings from blocking sets in residues. Finally, we prove some general properties of round-up triples of vertices in general spherical buildings, which we will use to classify such triples in several cases in this paper.
We then start the proofs of our main theorems. 
We chose to prove \cref{A} and \cref{B} type-by-type in the same section, so that the properties of the geometry in question are fresh in the memory.
\Cref{minuscule} treats the cases $\mathsf{E_{6,1},E_{6,3},E_{7,7}}$ and $\mathsf{E_{7,6}}$ using the so-called exceptional \emph{minuscule} geometries.
These are the Lie incidence geometries of type $\mathsf{E_{6,1}}$ and $\mathsf{E_{7,7}}$. 
Then, in \Cref{hexagonic}, we prove our main theorems for $\mathsf{E_{6,2},E_{6,4},E_{7,1},E_{7,2},E_{8,8},E_{8,7}}$ and $\mathsf{F}_{4,i}$, $i\in\{1,2,3,4\}$. 
The relatively easy cases $\mathsf{E_{6,2},E_{7,1},E_{8,8}}$ and $\mathsf{F}_{4,i}$, $i\in\{1,4\}$ are proved in \cref{hexpoints}, 
whereas the other cases are treated in \cref{hexlines}. 
This is by far the longest section of the paper. 
Amongst other things, it classifies all possible mutual positions of two lines in a given hexagonic Lie incidence geometry of exceptional type. 
This certainly has other potential applications. 
Doing this enables us to reduce \cref{A} for lines of exceptional hexagonic Lie incidence geometries to the classification of geometric lines (\cref{B}), 
except in the case of geometries isomorphic to $\mathsf{F_{4,4}}(q,q^2)$. 
We treat this case separately in \cref{F44}. 
Then in \cref{geomlineslines} we prove \cref{B} uniformly for $\mathsf{E_{6,4},E_{7,2},E_{8,7}}$ and $\mathsf{F}_{4,i}$, $i\in\{2,3\}$. 
Finally, in \cref{hexagons}, we prove our main results for type $\mathsf{G_2}$, that is, for generalised hexagons. 
\cref{anappl} contains the application to certain opposition preserving maps in buildings of exceptional type alluded to above. We also mention other potential applications.


\section{Preliminaries}\label{prelim} 

\subsection{Buildings and Lie incidence geometries}

We are going to work with the buildings of exceptional type via their well-established point-line geometries, which fit perfectly in the framework of Cooperstein's theory of \emph{parapolar spaces}.
We will also adopt the corresponding terminology. 
As a result, we will not spend too much space on pure building-theoretic theory, but instead refer the interested reader to the literature, in particular to \cite{Abr-Bro:08} and \cite{Tits:74}.
We content ourselves with the following generalities.

We view buildings as thick numbered simplicial chamber complexes.
The buildings we are interested in are \emph{spherical buildings} and, as such, there is the notion of \emph{opposition} of simplices (in particular vertices and chambers), expressing that the two given simplices are at maximal distance apart.
There is also the notion of \emph{convexity} for subcomplexes, and the convex closure of two opposite chambers is a thin finite chamber complex called an apartment and isomorphic to a Coxeter complex.
These apartments play a crucial role in building theory.
By the definition of a building, every pair of simplices is contained in an apartment and so the mutual position between these two objects can be seen in this finite complex, which is also a triangulation of a sphere.
Opposite simplices in an apartment are then just antipodal on the sphere.
Also, recall that, since $\Delta$ is numbered, the vertices have \emph{types}, and a chamber consists of a set of vertices, one of each possible type.
The number of types is the \emph{rank} of the building.
A \emph{panel} of cotype $i$ is a simplex containing vertices of each type except for the type $i$.
Opposition induces a permutation on the types which is an automorphism of the Coxeter diagram.
Two simplices are \emph{joinable} if their union is again a simplex.

We now quickly outline how point-line geometries arise from (spherical) buildings, a \emph{point-line geometry} being a pair $(X,\cL)$ consisting of a set $X$ of \emph{points}, and $\cL$ a set of subsets of $X$, each member of which is called a \emph{line}.
Let $\Delta$ be a spherical building, which we will always assume to be thick and irreducible.
We consider the set $X$ of vertices of a given type, say $i$.
The set $\cL$ then consists of the subsets of $X$ whose elements each complete a given panel of cotype $i$ to a chamber (hence each panel defines a unique line of $(X,\cL)$, but different panels may define the same line).
If the type of $\Delta$ is $\mathsf{X}_n$, where $n$ is the rank and $\mathsf{X}$ one of the Coxeter types $\mathsf{A,B,C,D,E,F,G}$, then we say that $(X,\cL)$ is a \emph{Lie incidence geometry of type $\mathsf{X}_{n,i}$}.
Usually, we take as type set $\{1,2,\ldots,n\}$, where the types can be read off the corresponding Coxeter or Dynkin diagram using Bourbaki labelling \cite{Bou:68}.
If the diagram is simply laced, then $\Delta$ is completely determined by the underlying skew field $\K$ and we denote the Lie incidence geometry of type $\mathsf{X}_{n,i}$ in this case as $\mathsf{X}_{n,i}(\K)$.
In such geometries, vertices of the building have interpretations as certain subspaces (see below), and we will always call such subspaces \emph{opposite} when the vertices are opposite in the building.
We adopt the notation $v\equiv v'$ for opposite vertices (and later, subspaces); the negation is $v\not\equiv v'$, and the set of vertices (not) opposite $v$ is denoted as $v^\equiv$ ($v^{\not\equiv}$).

When $\mathsf{X}_n$ is one of $\mathsf{E_6,E_7,E_8,F_4}$, then a Lie incidence geometry of type $\mathsf{X}_{n,i}$ is always a parapolar space.
We provide a brief introduction.

First, we introduce some terminology concerning point-line geometries.
Let $\Gamma=(X,\cL)$ be a point-line geometry.
Two points $x,y\in X$ contained in a common line are called \emph{collinear} and denoted $x\perp y$.
The set of points collinear to $x$ is denoted $x^\perp$ and includes $x$ if $x$ is on some line.
Sets of points are called collinear if each point of either is collinear to each point of the other.
If each line contains exactly two (at least three) points, then $\Gamma$ is called \emph{thin} (\emph{thick}, respectively).
A \emph{subspace} of $\Gamma$ is a set of points with the property that, if two distinct collinear points belong to it, then all points of each line containing both $x$ and $y$ belong to it.
We often view a subspace as a point-line subgeometry in the obvious way.
A subspace is a \emph{(geometric) hyperplane} if it intersects each line non-trivially; it is called \emph{proper} if it does not coincide with $X$ itself, which is a trivial geometric hyperplane.
The point-line geometry $\Gamma$ is called a \emph{partial linear space} if every pair of collinear points is contained in exactly one line.
In a partial linear space, we denote the unique line containing two distinct collinear points $x$ and $y$ by $xy$, or sometimes by $\langle x,y\rangle$, for clarity.
A subspace of a point-line geometry is called \emph{singular} if every pair of its points is collinear.
Trivial examples are the empty set, each singleton, and each line.
If there exists a natural number $r$ such that every finite nested sequence of (distinct) singular subspaces (including the empty space) has size at most $r+1$, and there exists such a sequence of size $r+1$, then we say that $\Gamma$ has singular rank $r$.
A \emph{maximal} singular subspace is a singular subspace that is not properly contained in another one.

The \emph{point graph} of a point-line geometry $\Gamma$ is the graph with vertices the points of $\Gamma$, adjacent when (distinct and) collinear.
A set of points is called \emph{convex} if for each pair $\{x,y\}$ of points contained in it, all points of each shortest path between $x$ and $y$ in the point graph are also contained in it.

We assume the reader is familiar with projective spaces, which are the Lie incidence geometries $\mathsf{A}_{n,1}(\K)$, for skew fields $\K$, also sometimes denoted as $\PG(n,\K)$, or as $\PG(V)$, where $V$ is an $(n+1)$-dimensional vector space over $\K$.
One checks that $\PG(n,\K)$ has singular rank $n+1$.
As a building, it has rank $n$, which is also its projective dimension.
We extend the definition of $\PG(V)$ to infinite-dimensional vector spaces in the obvious way: the points of $\PG(V)$ are the $1$-spaces of $V$, the lines are the sets of $1$-spaces contained in given $2$-spaces.

The Lie incidence geometries of types $\mathsf{B}_{n,1}$, $n\geq 2$ and $\mathsf{D}_{n,1}$, $n\geq 3$, are \emph{polar spaces (of rank $n$)}, that is, thick point-line geometries $(X,\cL)$ of singular rank $n$ such that for each point $x\in X$, the set $x^\perp$ is a proper geometric hyperplane.
It follows that polar spaces are partial linear spaces (see \cite{Bue-Shu:74}).
Also, the non-trivial singular subspaces of any polar space of rank at least $3$ are projective spaces.
If in a polar space of rank $r$, every $(r-2)$-dimensional singular subspace is contained in exactly two (at least three) maximal singular subspaces, then we call the polar space \emph{top-thin} (\emph{thick}, respectively).
Buildings of type $\mathsf{B}_n$ correspond precisely to the thick polar spaces of rank $n$, while buildings of type $\mathsf{D}_n$, $n\geq 2$ (where we identify the type $\mathsf{D_2}$ with the reducible type $\mathsf{A_1\times A_1}$) yield top-thin polar spaces.
The singular subspaces are also the only subspaces of a polar space that are convex.
In fact, every polar space of rank at least 2 is a Lie incidence geometry of type $\mathsf{B}_{n,1}$ or $\mathsf{D}_{n,1}$, $n\geq 2$. Every polar space of rank $r\geq 3$ admits an \emph{order $(s,t)$}, that is, there is a constant $s$ such that each line contains precisely $s+1$ points and there is a constant $t$ such that each $(r-2)$-dimensional singular subspace is contained in exactly $t+1$ maximal singular subspaces. For instance, top-thin polar spaces of rank at least $3$ have order $(s,1)$. 

Polar spaces of type $\mathsf{D}_n$ have a peculiar property: their maximal singular subspaces are divided into two \emph{oriflamme classes}, where two maximal singular subspaces belong to different classes if, and only if, the parity of the dimension of their intersection coincides with the parity of $n$.
Each polar space $\mathsf{D}_{n,1}(\K)$, for $\K$ a (commutative) field and $n\geq 3$, is isomorphic to the point-line geometry naturally associated to a hyperbolic quadric in $\PG(2n-1,\K)$, that is, the null set of a quadratic form $x_{-n}x_n+x_{-n+1}x_{n-1}+\cdots + x_{-1}x_1$ in the coordinates of a point.
An oriflamme class of lines in a hyperbolic quadric in $\PG(3,\K)$ (the case $n=2$ left out in our previous sentence) will be called a \emph{regulus}. Top-thin polar spaces are also referred to as \emph{hyperbolic} polar spaces. 

Another special class of polar spaces are the \emph{symplectic polar spaces}, obtained from a symplectic polarity, or an alternating form. They have the peculiar property that they can be represented in a projective space such that all points of the projective space are the points of the symplectic polar space.  We refer to \cite{Mal:24} for more background on polar spaces. 

The Lie incidence geometries of type $\mathsf{B}_{n,n}$ are usually referred to as \emph{dual polar spaces}.

We have defined Lie incidence geometries only for vertices of (spherical) buildings; a straightforward generalisation to simplices is possible, and we will use such generalisation, but only for simplices of buildings of type $\mathsf{A}_n$, where the simplices in question are point-hyperplane pairs of the corresponding projective space.
Hence we can define the geometry $\mathsf{A}_{n,\{1,n\}}(\K)$, $\K$ any skew field, as the point-line geometry with point set the set of incident point-hyperplane pairs of $\PG(n,\K)$, where the lines are of two types: one type of lines consists of the sets of point-hyperplane pairs with common hyperplane $H$ and point ranging over a given line contained in $H$; the other type is the dual.

Let us also remark that sometimes a Lie incidence geometry $\Delta_i$ of type $\mathsf{X}_{n,i}$ can be defined using the one, say $\Delta_j$, of type $\mathsf{X}_{n,j}$ by considering the subspaces of $\Delta_j$ conforming to the vertices of type $i$ of the associated building, as points, and interpreting then the lines of $\Delta_i$ in $\Delta_j$ to define them correctly.
We give an example.
Let $\Delta$ be a building of type $\mathsf{D}_n$, $n\geq 4$, and let $\Delta_1=(X_1,\cL_1)$ be the associated polar space (a Lie incidence geometry of type $\mathsf{D}_{n,1}$).
Define $X_2$ as the set of lines of $\Delta_1$.
Now let $\Delta_2$ be the point-line geometry with point set $X_2$, and let the set $\cL_2$ of lines of $\Delta_2$ be the set of planar line pencils.
Then one checks that $(X_2,\cL_2)$ is the Lie incidence geometry of type $\mathsf{D}_{n,2}$ corresponding to $\Delta$.
We say that $\Delta_2$ is the \emph{line-Grassmannian} of $\Delta_1$ (and we use this expression also in other situations where a point-line geometry has planes, and hence planar line pencils).

All Lie incidence geometries, as we defined them, are parapolar spaces, except for the projective and polar spaces mentioned above, and for the Lie incidence geometries of buildings of rank $2$.
Let's define these objects.
Unlike polar spaces, it is not known whether all parapolar spaces of sufficiently high symplectic rank are Lie incidence geometries.

A \emph{parapolar space}, introduced by Cooperstein \cite{Coo:76,Coo:77}, is a point-line geometry $\Gamma=(X,\cL)$ satisfying the following axioms:
\begin{compactenum}[$(i)$]\item Each pair of points at distance $2$ in the point graph either is collinear to a unique point, or is contained in a convex subspace isomorphic to a polar space (such subspaces are called \emph{symplecta}, or \emph{symps} for short).
\item There exist at least two distinct symps, and each line is contained in a symp.\end{compactenum}
It follows easily that parapolar spaces are partial linear spaces.
Now we introduce some specific terminology and notation concerning parapolar spaces.
First, a pair of points ${x,y}$ at distance $2$ in the point graph, collinear to a unique point $z$, will be called \emph{special}, and we denote $z=:[x,y]$.
We also say that $x$ \emph{is special to} $y$, or that $x$ and $y$ are special, in symbols $x\Join y$.
The set of points special to a given point $x$ will be denoted as $x^{\Join}$.
A parapolar space without special pairs is called \emph{strong}.
The symp containing two given points $x,y$ at distance $2$ in the point graph and which are not special is denoted by $\xi(x,y)$.
The pair ${x,y}$ is called \emph{symplectic}, and we also say that \emph{$x$ is symplectic to $y$}, or that $x$ and $y$ are symplectic, in symbols $x\pperp y$.
The set of points symplectic to a given point $x$ will be denoted as $x^{\pperp}$.
The \emph{diameter} of $\Gamma$ is the diameter of its point graph.
We say that $\Gamma$ has \emph{symplectic rank at least $r$} if every symp has rank at least $r$.
If every symp has exactly rank $r$, then we say that $\Gamma$ has \emph{uniform rank $r$}.

Note that the line-Grassmannian of a polar space contains special pairs.
If such a polar space $\xi$ is a symp in a parapolar space, then we speak about \emph{$\xi$-special lines}, with the obvious meaning; that is, disjoint lines containing points at distance $2$ such that some point of either is collinear to all points of the other.

We can interpret residues (or links) of vertices from the theory of buildings in the corresponding Lie incidence geometries as \emph{point residuals}.
Let $\Delta=(X,\cL)$ be a parapolar space with the property that all singular subspaces are projective spaces and through each point we have at least one plane.
Then, at each point $x\in X$, we can define the point-residual $\Res_\Delta(x)$, or $\Res(x)$ if no confusion can arise, as the point-line geometry with point set the set of lines of $\Delta$ through $x$ and with as set of lines the planar line pencils of $\Delta$ at $x$ (each line of the line pencil contains $x$).
As usual, the type of the point residuals can be read off the Coxeter diagram by deleting the node corresponding to the points.

Parapolar spaces are so-called \emph{gamma spaces}, that is, point-line geometries in which each point is collinear to zero, one, or all points of a given line.
We will frequently use this property, often without reference.

We will be working with specific Lie incidence geometries, mainly of exceptional type.
For the classical types, the properties can be derived without much effort from the corresponding projective or polar space.
We now review the basic properties of the exceptional Lie incidence geometries we will be working with.
Along the way, we also prove some additional properties that we will need.

The basic properties reflect the possible mutual positions of certain elements of the geometry, usually points, singular subspaces, and symps.
In the statements of facts, we will occasionally introduce terminology and underline the introduced notions.
Subspaces or symps through a common point $x$ will occasionally be called \emph{locally opposite (at $x$)} if they correspond to opposite objects in $\Res(x)$.
We also extend this terminology to all vertices $x$.
In particular, if $\xi$ is a symp, then singular subspaces in $\xi$ that are opposite in $\xi$ as a polar space are called \emph{locally opposite at $\xi$}, or briefly \emph{$\xi$-opposite}.
This extends the terminology for lines of $\xi$ being $\xi$-special, introduced before.

\subsection{Lie incidence geometries of type $\mathsf{E_{6,1}}$} 

For each field $\K$ there exists a unique Lie incidence geometry isomorphic to $\mathsf{E_{6,1}}(\K)$.
It is a strong parapolar space of diameter $2$ and also called a \emph{minuscule geometry}.
The following properties can either be found in \cite{Tits:57}, or can be easily derived from an apartment of the corresponding building (in this case the $1$-skeleton of such an apartment, where vertices are points of $\mathsf{E_{6,1}}(\K)$ and edges are lines, is the Schl\"afli graph, see  \cite[\S10.3.4]{Bro-Coh-Neu:89}).
One can also use the chain calculus introduced in \cite{Tits:56}, also explained in \cite[\S4.5.4]{Bro-Mal:22}.

\begin{fact}\label{pointsympE61}
Let $p$ be a point and $\xi$ a symp of $\mathsf{E_{6,1}}(\K)$ with $x\notin\xi$.
Then $p^\perp\cap \xi$ is either empty (and we say that $x$ and $\xi$ are \underline{far}; they are also opposite in the corresponding building) or a maximal singular subspace of $\mathsf{E_{6,1}}(\K)$, which we call a \underline{$4'$-space} (and we say that $x$ and $\xi$ are \underline{close}).
Also, $\Res(p)$ is the Lie incidence geometry $\mathsf{D_{5,5}}(\K)$.
\end{fact}

\begin{fact}\label{sympsympE61}
Two symps intersect either in a point, or in a maximal singular subspace of either, in which case we call the subspace a \underline{$4$-space} and the symps \underline{adjacent}.
The $4$-spaces of a given symp $\xi$ constitute an oriflamme class of the symp, which is a polar space of type $\mathsf{D_{5,1}}$; the $4'$-spaces contained in $\xi$ form the other oriflamme class.
 
\end{fact}

The following lemmas can be read off the diagram, checked in an apartment, and is contained in \cite[\S3.2]{Tits:57}, but can also be proved using the above lemmas.

\begin{fact}\label{345E61}
Each $3$-space of $\mathsf{E_{6,1}}(\K)$ is the intersection of a unique $4$-space and a unique $5$-space. 
\end{fact}

\begin{fact}\label{restE61}
There are two kinds of maximal singular subspaces in $\mathsf{E_{6,1}}(\K)$.
One kind corresponds to the $4$-spaces, the other to $5$-dimensional projective spaces, called $5$-spaces, which contain $4'$-spaces.
Let $p$ be a point and $W$ a $5$-space.
Then either $p\in W$, or $p^\perp\cap W$ is a $3$-dimensional space (called a $3$-space; $p$ and $W$ are called \underline{close}), or $p$ is collinear to a unique point of $W$ (and $p$ and $W$ are called \underline{far}).
Let $W$ and $W'$ be two distinct $5$-spaces.
Then either $W\cap W'$ is a plane (then $W$ and $W'$ are called \underline{adjacent}), or $W\cap W'$ is just a point, or $W$ and $W'$ are disjoint and there exists a unique $5$-space intersecting both in a respective plane, or $W$ and $W'$ are disjoint and opposite in the building; in the latter case, every point of $W$ is far from $W'$ and every point of $W'$ is far from $W$, and collinearity defines a collineation between $W$ and $W'$.
.
\end{fact}

We can now prove some (new) lemmas. 

\begin{lem}\label{lemE6}
Let $x$ be a point and $\xi$ a symp of $\mathsf{E_{6,1}}(\K)$ with $x\notin\xi$.
Then $x$ is opposite $\xi$ if, and only if, for some point $y\in\xi$, not collinear to $x$, the symp $\xi(x,y)$ intersects $\xi$ only in $y$ if, and only if, for all points $y\in\xi$ not collinear to $x$,  the symp $\xi(x,y)$ intersects $\xi$ only in $y$.
\end{lem}

\begin{proof}
If $x$ is opposite $\xi$, then $\xi \cap \xi(x,y)$ can not be more than just the point $y$, for every $y \in \xi$, because otherwise, $\xi \cap \xi(x,y)$ is a $4$-space by \cref{sympsympE61} and $x$ is collinear to a $3$-space of that $4$-space, contradicting the fact that $x$ is opposite $\xi$.
Now suppose that $\xi \cap \xi(x,y) = \{y\}$, for some fixed $y \in \xi$.
We want to see that $x$ has to be opposite $\xi$.
Suppose there exists some $z \in \xi$ at distance $2$ from $x$, such that $\xi \cap \xi(x,z)$ is a $4$-space (cf.~\cref{pointsympE61}).
Then $y$ and $x$ have to be collinear to $3$-spaces $U_y$ and $U_x$, respectively, of that $4$-space, which will necessarily intersect in at least a plane $\pi$.
But then $\pi$ will also be contained in $\xi(x,y)$, by convexity, and hence $\xi \cap \xi(x,y)$ will be more than just the point $y$, which is a contradiction.
So for every point $z \in \xi\setminus x^\perp$, we have $\xi(x,z) = \{z\}$.
That means that $x$ is not collinear to any point in $\xi$ and thus, $x$ is opposite $\xi$.
\end{proof}

\begin{lem}\label{lemE6b}
Let $L$ be a line and $\xi$ a symp of $\mathsf{E_{6,1}}(\K)$ with $L\cap\xi=\varnothing$.
If no point of $L$ is opposite $\xi$, then $L$ is collinear to a unique plane of $\xi$.
 
\end{lem}

\begin{proof}
If no point of $L$ is opposite $\xi$, then every point of $L$ is collinear to a $4'$-space of $\xi$.
Two $4'$-spaces in a symp intersect in either a point or a plane or they coincide, because they belong to the same oriflamme class by \cref{restE61}.
If two points of $L$ were collinear to the same $4'$-space of $\xi$, then every point of that $4'$-space would be collinear to every point of $L$ and $L$ and that $4'$-space would span a $6$-space, which is impossible by \cref{restE61}.
Now, let $x$ and $y$ be two points of $L$, let $V$ be the $4'$-space that $y$ is collinear to in $\xi$ and let $x'$ be some point of $\xi \setminus V$ that $x$ is collinear to.
Then $x'$ has to be collinear to a $3$-space $U \subseteq V$.
 The symp $\xi(y,x')$ contains $U$ and $x$.
With that, $x$ has to be collinear to a plane $\pi$ of $U$.
That means $x^{\perp} \cap \xi$ and $y^{\perp} \cap \xi$ intersect in $\pi$ and since every point of $\pi$ is collinear to $x, y \in L$, and $\mathsf{E_{6,1}}(\K)$ is a gamma space, every point of $\pi$ has to be collinear to every point of $L$.
With that, $L$ is collinear to a unique plane of $\xi$.
\end{proof}

\begin{lem}\label{3spacesuffices}
Let $p$ be a point and $M$ a $3$-space of  $\mathsf{E_{6,1}}(\K)$, such that no point of $M$ is collinear to $p$.
Then the unique maximal $4$-space containing $M$ does not contain any point collinear to $p$.
\end{lem}

\begin{proof}
Suppose a point $p$ is collinear to a point $q$ of a $4$-space $C$ containing a $3$-space $M$ which does not contain any point collinear to $p$.
Put $C$ in a symp $\xi$.
If $p$ is in $\xi$, then $p$ is collinear to a $3$-space of $C$, which intersects $M$, a contradiction.
If $p$ is not in $\xi$, then $p$ is collinear to a $4'$-space $W$, which already intersects $C$ in $q$.
But the intersection must have odd codimension, hence the intersection $C\cap W$ is either a line or a $3$-space.
Both would intersect $M$, a contradiction. 
\end{proof}

\begin{lem}\label{554E61}
Let $W,W'$ be two opposite $5$-spaces, and let $U$ be a $4$-space intersecting $W$ in a $3$-space.
Then there exists a unique point $p\in U\setminus W$ close to $W'$.
Also, $p^\perp\cap W'=\{x'\in W\cap U'\mid (\exists x\in W)(x'\perp x)\}$. 
\end{lem}

\begin{proof}
Putting $W,W'$ and $U$ in a common apartment, the existence of $p$ readily follows.
\cref{3spacesuffices} shows the second assertion, and then uniqueness of $p$ also follows immediately. 
\end{proof}

\begin{lem}\label{linesEsixMaximalFourSpace} 
Let $L$ be a line, and let $b$ be a point not collinear to any point of $L$ in $\mathsf{E_{6,1}}(\K)$. 
Then $\<b, b^\perp \cap L^\perp\>$ is a maximal $4$-space.
\end{lem}

\begin{proof}
By convexity of symps.
$\<b,b^\perp\cap L^\perp\>$ is contained in each symp $\xi(b,p)$, with $p\in L$.
Fix such a symp $\xi=\xi(b,p)$, for some $p\in L$.
\cref{pointsympE61} implies that $L$ is collinear to a $4'$-space $U'$ of $\xi$.
Then $U:=\<b,b^\perp\cap U'\>$ is a $4$-space, because $U$ and $U'$ belong to different oriflamme classes.
By \cref{restE61}, $U$ is a maximal singular subspace and the assertions are proved.
\end{proof}

\subsection{Lie incidence geometries of type $\mathsf{E_{7,7}}$}  

For each field $\K$ there exists a unique Lie incidence geometry isomorphic to $\mathsf{E_{7,7}}(\K)$.
It is a strong parapolar space of diameter $3$ and also called a \emph{minuscule geometry}.
The following properties can easily be derived from an apartment of the corresponding building (in this case the $1$-skeleton of such an apartment, where vertices are points of $\mathsf{E_{7,7}}(\K)$ and edges are lines, is the Gosset graph; see \cite[\S10.3.5]{Bro-Coh-Neu:89}). 

\begin{fact}\label{preE77} Let $\Delta$ be the Lie incidence geometry $\mathsf{E_{7,7}}(\K)$.
Then the following assertions hold.
\begin{compactenum}[$(i)$]
\item $\Delta$ is strong, has uniform symplectic rank $6$, singular rank $7$ and diameter $3$.
Points at distance $3$ are opposite.
\item Symps in $\Delta$ are isomorphic to hyperbolic polar spaces $\mathsf{D_{6,1}}(\K)$. 
\item Point residuals in $\Delta$ are isomorphic to $\mathsf{E_{6,1}}(\K)$. 
\item The maximal singular subspaces of highest dimension in $\Delta$ are projective spaces of dimension $6$.
Like before, we call $5$-dimensional projective subspaces contained in $6$-spaces $5'$-spaces. 
\item Maximal $5$-spaces occur as the intersection of two symps.
On the other hand, $5'$-spaces occur as the intersection of a unique $6$-space and a unique symp.
\item For each symp $\xi$, its $5$-spaces form an oriflamme class, and its $5'$-spaces form the other oriflamme class of $\xi$.
\end{compactenum}
\end{fact}

\begin{fact}\label{point-sympE7} Let $p$ be a point and $\xi$ a symp of $\mathsf{E_{7,7}}(\K)$, with $p \notin \xi$.
Then precisely one of the following occurs.
\begin{compactenum}[$(i)$]
    \item $p$ is collinear to a $5'$-space $A$ of $\xi$, $p$ is symplectic to the points of $\xi \setminus A$, and we say that $p$ is \underline{close} to $\xi$.
    \item $p$ is collinear to a unique point $q \in \xi$, $p$ is symplectic to the points of $\xi \cap (q^\perp \setminus \{q\})$, and $p$ is opposite to the points $\xi \setminus q^\perp$.
We say that $p$ is \underline{far} from $\xi$.
 
\end{compactenum}
\end{fact}
This fact implies that, on each line $L$, there is at least one point symplectic to a given point $p$ (unique when $L$ contains at least one point opposite $p$). 

\begin{fact}\label{symp-sympE77} Let $\xi$ and $\xi'$ be two distinct symps of $\mathsf{E_{7,7}}(\K)$.
Then precisely one of the following occurs.
\begin{enumerate}[label=(\roman*)]
    \item $\xi \cap \xi'$ is a $5$-space, and we call $\xi$ and $\xi'$ \underline{adjacent}.
   
    \item $\xi \cap \xi'$ is a line $L$.
Then points $x \in \xi \setminus L$ and $x' \in \xi' \setminus L$ are never collinear.
 We call $\{\xi, \xi'\}$ \underline{symplectic}.
   
    \item $\xi \cap \xi' = \varnothing$, and there is a unique symp $\xi''$ intersecting both $\xi$ and $\xi'$ in respective $5$-spaces $A$ and $A'$, which are opposite in $\xi''$.
All points of $\xi \setminus A$ are far from $\xi'$, and each point of $A$ is close to $\xi'$.
Each line containing a point of $\xi$ and a point of $\xi'$ contains a point of $A \cup A'$.
We call $\{\xi, \xi'\}$ \underline{special}.
   
    \item $\xi \cap \xi' = \varnothing$, and every point of $\xi$ is far from $\xi'$.
In this situation, each point of $\xi'$ is also far from $\xi$, and $\xi$ and $\xi'$ are opposite. 
\end{enumerate}
\end{fact}

\begin{fact} Let $\xi_1$ and $\xi_2$ be two opposite symps of $\mathsf{E_{7,7}}(\K)$.
Let $\cL$ be the set of all lines that contain a point of $\xi_1$ and a point of $\xi_2$.
Then, for each point $p$ that is contained in a line of $\cL$, there exists a unique symp $\xi_p$ that intersects each line $L \in \cL$. 
\end{fact}

We can now show the following lemma. 

\begin{lem}\label{lemE7}
Let $x,y$ be two points of $\mathsf{E_{7,7}}(\K)$.
Then $x$ and $y$ are opposite if, and only if, they are contained in respective symps intersecting in a line $L$ such that $x$ and $y$ are collinear to unique respective distinct points of $L$.
\end{lem}

\begin{proof}
First suppose that $x$ and $y$ are opposite.
Let $\xi_x$ be an arbitrary symp containing $x$.
\cref{point-sympE7}$(ii)$ implies that $y$ is collinear to a unique point $z\in\xi_x$.
Let $\xi_y$ be an arbitrary symp through $y$ and $z$.
If $\xi_x$ and $\xi_y$ intersected in more than a line, they would intersect in a $5$-space (cf.~\cref{symp-sympE77}), and both $x$ and $y$ would have to be collinear to $4$-spaces of that $5$-space.
These would necessarily intersect, meaning that there would exist points collinear to both $x$ and $y$, contradicting the fact that $x$ and $y$ are opposite.
Hence $\xi_x\cap\xi_y$ is a line $L$.
Now, clearly $x$ and $y$ are not collinear to a common point on $L$, and the assertion follows. 

Next suppose, conversely, that $x$ and $y$ are contained in respective symps $\xi_x$ and $\xi_y$ intersecting in a line $L$ such that $x$ and $y$ are collinear to unique respective distinct points $x'$ and $y'$ of $L$.
Suppose, for a contradiction, that $x$ is collinear to a $5'$-space $U$ of $\xi_y$.
Then $y'^\perp\cap x^\perp$ contains points of $U$ that do not belong to $\xi_x$, a contradiction.
Hence $x$ is far from $\xi_y$, and the assertion follows from \cref{point-sympE7}$(ii)$.
 
\end{proof}

\subsection{Lie incidence geometries of types $\mathsf{E_{6,2},E_{7,1},E_{8,8},F_{4,1}}$ and $\mathsf{F_{4,4}}$.} 

These Lie incidence geometries are examples of \emph{hexagonic geometries}, as defined by Shult \cite[Section~13.7]{Shu:11}, inspired by his work with Kasikova \cite{Kas-Shu:02}.
We will not need the formal definition of such geometries; some defining properties will be part of the facts that we state below.
Since we are only concerned with exceptional geometries, we will restrict ourselves to these cases.
This implies, for instance, that we can assume that the parapolar space in question has uniform symplectic rank (which is at least $3$).
Other examples of hexagonic Lie incidence geometries are the line-Grassmannians of polar spaces, which have symplectic rank at least $3$ if the polar space has rank at least $4$, and uniform symplectic rank $3$ if, and only if, the polar space has rank $4$.
Many facts stated below also hold for these spaces, but can in that case easily be directly checked in the polar space. 

The following facts can again be easily checked in an apartment (for models of such, see \cite{Mal-Vic:22}), or follow from the diagram.
We will refer to the geometries in the title of this section as the \emph{exceptional hexagonic (Lie incidence) geometries}.
The ones of type $\mathsf{F_{4,1}}$ and $\mathsf{F_{4,4}}$ are also known as \emph{(thick) metasymplectic spaces}.
A detailed introduction to the latter is contained in \cite{Lam-Mal:25}.
We will always assume thickness when mentioning metasymplectic spaces. 

\begin{fact}\label{pointpointhex}
Let $x$ and $y$ be two distinct non-collinear points of an exceptional hexagonic Lie incidence geometry.
Then $x$ and $y$ are either symplectic, special, or opposite.
In the latter case, the distance between $x$ and $y$ is $3$.
The set $x^{\not\equiv}$ is always a proper geometric hyperplane. 
\end{fact}

The following fact can also be deduced from \cite[Lemma 2(v)]{Coh-Ivan:07}.

\begin{fact}\label{joinjoin}
Let $x$ and $u$ be two points of an exceptional hexagonic Lie incidence geometry.
Let $x\perp y\perp z\perp u$.
Then $x$ and $u$ are opposite if, and only if, both $\{x,z\}$ and $\{y,u\}$ are special pairs.
In particular, if $x\pperp v\perp u$ for some point $v$, then $x$ and $u$ are not opposite.
\end{fact}

\begin{fact}\label{pointsymphex}
Let $x$ be a point and $\xi$ a symp of an exceptional hexagonic Lie incidence geometry.
Then exactly one of the following occurs.
\begin{compactenum}[$(i)$]
\item $x\in\xi$;
\item $x\notin\xi$ and $x^\perp\cap\xi$ is a maximal singular subspace in $\xi$ (this cannot happen in a metasymplectic space);
\item $x\notin\xi$ and $x^\perp\cap\xi$ is a line $L$;
\item $x^\perp\cap\xi=\varnothing$ and $x^{\pperp}\cap\xi$ is a maximal singular subspace $U$ (this cannot happen in a metasymplectic space); 
\item $x^\perp\cap\xi=\varnothing$ and $x^{\pperp}\cap\xi=\xi$ (this does not occur in types $\mathsf{F_4,E_8}$);
\item $x^\perp\cap\xi=\varnothing$ and $x^{\pperp}\cap\xi$ is a unique point $y$.
\end{compactenum}
\end{fact}
Regarding the possibilities $(ii)$ and $(iv)$, we note that these are ``dual'' situations in the sense that, if in an apartment $\Sigma$ a point $x$ is collinear to exactly a maximal singular subspace $U$ of a symp $\xi$ of $\Sigma$, then the point opposite $x$ in $\Sigma$ is symplectic to each point of the unique maximal singular subspace $U'$ of $\xi$ disjoint from $U$ and belonging to $\Sigma$. Hence situation $(ii)$ and $(iv)$ are simultaneously   possible or impossible (and the latter happens for metasymplectic spaces). 

The following fact follows from the diagrams by taking point residuals. 
\begin{fact}\label{sympsymphex}
Let $x$ be a point of the parapolar space $\Delta$ isomorphic to either $\mathsf{E_{6,2}}(\K), \mathsf{E_{7,1}}(\K),\mathsf{E_{8,8}}(\K)$, or a metasymplectic space.
Then $\Res_\Delta(x)$ is isomorphic to $\mathsf{A_{5,3}}(\K),\mathsf{D_{6,6}}(\K),\mathsf{E_{7,7}}(\K)$, or a dual polar space of rank $3$, respectively.
Consequently, when two symps of $\Delta$ have a line in common, then they have (at least) a plane in common.
Two non-disjoint symps either intersect in a point, a plane, or a maximal singular subspace.
The singular rank of $\Delta$ is $5,7,8$ or $3$, respectively.
\end{fact}

In general, an $i$-dimensional singular subspace whose points do not correspond to the set of vertices of the corresponding building contained in a simplex together with another given vertex, will be called an \emph{$i'$-space}. Hence, with some standard terminology that we will not need in this paper, an $i'$ space is a singular $i$-space that is not the point-shadow of a single vertex of the building (but usually is it the point-shadow of a simplex of size $2$ for spherical Coxeter diagrams have at most one branching node).
An $i'$-space usually arises as the intersection of a maximal singular subspace with a symp. However, if all $i$-dimensional singular subspaces are $i'$-spaces, then we call them again $i$-spaces as there is no confusion possible. 

We can now be more specific in \cref{pointsymphex}.

\begin{lem}\label{generalspecial}
Let $\Delta$ be an exceptional hexagonic Lie incidence geometry.
Let $x$ be a point of $\Delta$ and $\xi$ a symp of $\Delta$ not containing $x$.
Then the following hold. 
\begin{compactenum}[$(i)$]
\item If $x^\perp\cap\xi$ is a maximal singular subspace $U$ in $\xi$, then each point of $\xi\setminus U$ is symplectic to $x$;
\item If $x^\perp\cap\xi$ is a line $L$, then each point $y$ of $\xi\setminus L$ collinear to a unique point of $L$ is special to $x$ (the other points of $\xi\setminus L$ are symplectic to $x$);
\item If $x^{\pperp}\cap\xi$ is a maximal singular subspace $U$, then each point of $\xi\setminus U$ is special to $x$;
\item If $x^{\pperp}\cap\xi$ is a unique point $y$, then each point of $\xi$ not collinear to $y$ is opposite $x$ (consequently, each other point of $\xi\setminus\{y\}$ is special to $x$).
\end{compactenum}
\end{lem}

\begin{proof}\begin{compactenum}[$(i)$]\item The maximal singular subspace $U$ of $\xi$ has dimension at least $2$, hence for each $y\in\xi\setminus U$, the set $y^\perp\cap U$ has at least three elements, implying, by the definition of parapolar spaces, that $x$ and $y$ are symplectic.
\item Suppose $x$ and $y$ were symplectic.
Then the symps $\xi$ and  $\xi(x,y)$ would have a line in common, hence, by \cref{sympsymphex}, they would share a plane $\alpha$, which has to contain  $L$ as $x^\perp\cap\alpha$ is a line.
Since also $y\in\alpha$, $y$ is collinear to all points of $L$, which contradicts the assumptions. 
\item This follows immediately from \cref{joinjoin}. 
\item No point of $\xi$ is collinear to $x$, as otherwise it follows from the other cases that $y$ is not unique.
Hence all points of $\xi$ collinear to $y$ are special to $x$.
Let $z$ be such a point.
Then, by the previous possibilities, $z^\perp\cap \xi(x,y)$ is a line $K$ and $u:=[x,z]\in K$.
Suppose, for a contradiction, that $u^\perp\cap\xi$ is a maximal singular subspace $U$ of $\xi$.
 Then, by \cref{pointsymphex}$(ii)$ and \cref{sympsymphex}, $\dim U\geq 3$ and any symp $\xi(u,w)$, with $w\in\xi\setminus u^\perp$, shares a maximal singular subspace $W$ with $\xi$.
It follows from \cref{pointsymphex} that $x^\perp\cap\xi(u,w)$ is at least a line $M$.
But then each point of $M^\perp\cap W$, which is at least $1$-dimensional, is symplectic to $x$, a contradiction.
So, $u^\perp\cap\xi$ is a line $N$, and it follows from the previous possibility that $u\Join v$, for every point $v\in\xi\setminus y^\perp$.
Now \cref{joinjoin} proves the assertion.
 
\end{compactenum}
\end{proof}

\begin{lem}\label{pentagon}
Let $x$ be a point of an exceptional hexagonic Lie incidence geometry, and suppose $x$ is special to $y_1$ and $y_2$, with $y_1\perp y_2$.
Then also the points $z_1=[x,y_1]$ and $z_2=[x,y_2]$ are collinear. 
\end{lem}

\begin{proof}
By \cref{joinjoin}, $z_1$ and $y_2$ are symplectic.
The point $x$ is collinear to $z_1$, and hence to a line $L$ of $\xi(z_1, y_2)$.
Then $y_2$ has to be collinear to a point of $L$, but this point can only be $z_2$, since $x$ and $y_2$ are special.
Note that it is possible that $z_1 = z_2$.
\end{proof}

\begin{lem}\label{pointspecialline}
Let $x$ and $L$ be a point and line, respectively, of an exceptional hexagonic Lie incidence geometry, and suppose that $x$ is special to at least two points $y_1$ and $y_2$ of $L$.
Then exactly one of the following occurs. 
\begin{compactenum}[$(i)$] \item We have $L\subseteq x^{\Join}$ and there exists a line $M$ consisting of the points collinear to $x$ and some point of $L$. Moreover, the mapping $L\to M:y\mapsto [x,y]$ is bijective.
\item There is a unique point $y\in L$ not special to $x$, and $L\setminus\{y\}\subseteq x^{\Join}$. Also, $x^\perp\cap L^\perp=\{z\}$, for some point $z$, and we have either $x\perp y=z$, or $x\pperp y$ and $y\neq z$.
\end{compactenum}
\end{lem}

\begin{proof}
For $i=1,2$, set $z_i:=[x,y_i]$. Suppose firstly that $z_1\neq z_2$. \cref{pentagon} yields $z_1\perp z_2$; set $M=z_1z_2$ and note $x\perp M$. Hence $z_1\pperp y_2$ and the symp $\xi(z_1, y_2)=:\xi$ contains $y_1$ and $M$, hence also $L$, but not $x$.  Now, for every point $b$ on $L$, the point $[x,b]$ is the unique point on $M$ collinear to $b$. Now $(i)$ follows.

Suppose secondly $z_1=z_2=:z$. Then $z\perp L$ and if $y\in L\cap x^{\Join}$, then $[x,y]=z$. If $z\in L$, then, denoting by $\xi$ any symp containing $L$, we have, in view of \cref{generalspecial}, that $x^\perp\cap\xi$ is a line not collinear to $L$ and so the first possibility of $(ii)$ occurs.

Hence we may assume $z\notin L$. Then $L$ and $z$ are contained in a unique plane $\pi$. This time, let $\xi$ be any symp containing $\pi$. As in the previous paragraph, $x^\perp\cap\xi$ is a line $K$ not collinear to $\pi$. Hence $K^\perp\cap \pi$ is a line intersecting $L$ in a unique point $y$. \cref{generalspecial}$(ii)$ yields $y\pperp x$ and the second case of $(ii)$ occurs. 
\end{proof}

Similarly, we can say something about a point being symplectic to all points of a line. 

\begin{lem}\label{pointsymplecticline}
Let $x$ and $L$ be a point and line, respectively, of an exceptional hexagonic Lie incidence geometry, and suppose that $x$ is symplectic to each point of $L$.
Then there exists a maximal singular subspace $W$ not contained in a symp, containing $L$, such that $x^\perp \cap W$ is complementary to $L$ in $W$, and each symp containing $x$ and a point of $L$ contains $x^\perp\cap W$. 
\end{lem}

\begin{proof}
Pick $x_1,x_2\in L$ and set $\xi_1:=\xi(x,x_1)$.
Then $x_2^\perp\cap\xi_1$ is either a line or a maximal singular subspace (cf.~\cref{pointsymphex}).
If it were a line $M$, then, by \cref{generalspecial}, $x$ would be special to $x_2$, as $x_1\in M$ and $x_1$ is not collinear to $x$.
Hence $x_2^\perp\cap\xi_1$ is a maximal singular subspace $U$.
Defining $W$ as the singular subspace generated by $U$ and $x_2$, the assertions follow. 
\end{proof}

\begin{lem}\label{opppointinsymps}
Let $\xi_1$ and $\xi_2$ be two non-disjoint symps of an exceptional hexagonic Lie incidence geometry, and let $x_i\in\xi_i$, $i=1,2$, be two points.
Then $x_1\equiv x_2$ if, and only if, $\xi_1\cap\xi_2$ is a point $z$, the symps $\xi_1$ and $\xi_2$ are locally opposite at $z$, and $\{x_i,z\}$ is a symplectic pair, $i=1,2$. 
\end{lem}

\begin{proof}
This follows from \cref{joinjoin} and \cref{generalspecial}, taking into account that $\xi_1$ and $\xi_2$ are locally opposite at an intersection point $z$ if, and only if, each point $z_i\in\xi_i\cap z^\perp \setminus\{z\}$ is collinear to a unique line of $\xi_j$, $\{i,j\}=\{1,2\}$ (which can easily be seen in $\Res(x)$; for instance, for $\Res(x)\cong\mathsf{E_{7,7}}(\K)$, this is \cref{symp-sympE77}$(iv)$). 
\end{proof}

\begin{lem}\label{perp22perp}
Let $\xi_1$ and $\xi_2$ be two opposite symps of an exceptional hexagonic Lie incidence geometry, and let $L_1\subseteq\xi_1$ be a line.
Then the set of points of $\xi_2$ symplectic to some point of $L_1$ is a line $L_2$ of $\xi_2$.
All symps having a point on $L_1$ and a point on $L_2$ share a unique common point $x$, which is collinear to both $L_1$ and $L_2$. 
\end{lem}

\begin{proof}
First, we note that, from general building-theoretic considerations, every point in $\xi_1$ has an opposite in $\xi_2$; hence, if $x_1\in L_1$ and $x_2\in x_1^{\pperp}\cap\xi_2$, then \cref{opppointinsymps} implies that $\xi_2$ and $\xi(x_1,x_2)$ are locally opposite at $x_2$ and, by \cref{pointsymphex} and \cref{generalspecial}, every point of $\xi_2$ not collinear to $x_2$ is opposite $x_1$.
By \cref{joinjoin}, such a point $y_2$ is not symplectic to any point of $L_1$. 
We conclude that ``being symplectic'' preserves collinearity in both directions (interchanging the roles of $\xi_1$ and $\xi_2$), and hence is an isomorphism between $\xi_1$ and $\xi_2$.
Let $L_2\subseteq\xi_2$ correspond to $L_1$ under that isomorphism.

Let $x_1'\in L_1\setminus\{x_1\}$.
Then there is a unique point $[x_1',x_2]=:x\in \xi(x_1,x_2)$ collinear to $x_2$ and $x_1'$.
Clearly, $x\perp L_1$.
Standard arguments switching roles of points on $L_1$ and $L_2$ imply that $x$ is independent of $x_1$ and $x_2$, and so $x$ is collinear to each point of $L_1\cup L_2$.
It is now easy to see that $x$ is contained in each symp containing a point of $L_1$ and a point of $L_2$.
Uniqueness of $x$ follows from the fact that $x=[x_1',x_2]$. 
\end{proof}

\begin{lem}\label{generalnotspecial}
Let $\Delta$ be an exceptional hexagonic Lie incidence geometry, and let $\{x_1, x_2\}$ be a symplectic pair of points in $\Delta$.
Let $x_2'$ be another point in $\Delta$ symplectic to $x_1$ and collinear to $x_2$.
Let $x_1'$ be a point such that $x_1 \perp x_1' \perp x_2'$.
Then $x_1'$ and $x_2$ cannot be special.
\end{lem}

\begin{proof}
The point $x_1'$ is collinear to (at least) a line $L$ of $\xi(x_1, x_2)$, which contains $x_1$ and a point $y$ collinear to $x_2$.
If $x_1'$ were special to $x_2$, then $y=[x_1',x_2]=x_2'\pperp x_1$, contradicting $y\perp x_1$.
\end{proof}

\cref{generalspecial} implies that, whenever two points $x,y$ of an exceptional hexagonic Lie incidence geometry $\Delta$ are opposite, then every symp through $x$ contains a unique point symplectic to $y$.
This way, one obtains all points $x^{\pperp}\cap y^{\pperp}$, and we call such a set an \emph{equator}.
This defines a subspace of $\Delta$ which we call the \emph{equator geometry (with poles $x$ and $y$)} and denote as $E(x,y)$.
We view these as point-line geometries as soon as they contain lines.
The latter is the case in the simply laced case (types $\mathsf{A}_n,\mathsf{D}_n,\mathsf{E_6,E_7}$ and $\mathsf{E_8}$).
In all those cases, these equator geometries can be defined in exactly the same way for the corresponding hexagonic geometries (here the Lie incidence geometries of types $\mathsf{A}_{n,\{1,n\}}$, $\mathsf{D}_{n,2}$, and the exceptional ones not of type $\mathsf{F_4}$), and we have the following sequences (where $\Delta\to\Delta'$ means that $\Delta'$ is an equator geometry of $\Delta$), which can be deduced from \cite{Sch-Mal:23}:
\[\begin{cases} \mathsf{E_{8,8}}(\K)\to\mathsf{E_{7,1}}(\K)\to\mathsf{D_{6,2}}(\K),\\ 
\mathsf{E_{6,2}}(\K)\to\mathsf{A_{5,\{1,5\}}}(\K)\to\mathsf{A_{3,\{1,3\}}}(\K).\end{cases}
\]
\label{eqgeom}

In the case of type $\mathsf{F_4}$, equator geometries as defined here have no lines.
We shall give an alternative definition for that case in the next paragraph such that also in type $\mathsf{F_4}$, each equator admits an interesting geometric structure.

\subsection{Metasymplectic spaces} 

The previous paragraph includes the metasymplectic spaces, that is, the Lie incidence geometries of types $\mathsf{F_{4,1}}$ and $\mathsf{F_{4,4}}$.
We now introduce some notation making apparent the differences between $\mathsf{F_{4,1}}$ and $\mathsf{F_{4,4}}$, based on the Dynkin diagram of type $\mathsf{F_4}$ rather than the Coxeter diagram.
Everything in the paragraph can be found in \cite{Lam-Mal:25} and is, of course, based on the fundamental work of Tits in \cite{Tits:74}. 

Given a field $\K$, a quadratic algebra $\A$ over $\K$ is an algebra that admits a bilinear form $b:\A\to\K$ such that for every $x\in\A$ we have $x^2-(b(1,x)+b(x,1))x+b(x,x)=0$.
The element $b(x,x)$ is called the \emph{norm} of $x$ and briefly denoted as $\norm(x)$.
We assume that $\A$ is \emph{alternative}, that is, $\A$ satisfies the alternative laws $(ab)b=ab^2$ and $a(ab)=a^2b$, and that $\A$ is a \emph{unital division} algebra, that is, it has an identity and every element has an inverse. This force the norm $\norm$ to be anisotropic, that is, $\norm(x)=0$ if, and only if, $x=0$. 
We can associate a polar space of rank $r\geq 2$ with every quadratic alternative unital division algebra as follows.
Let $V$ be the vector space isomorphic to the direct sum of $\A$ and $2r$ copies of $\K$.
Then define the quadratic form \[\beta:V\to\K:(x_{-r},x_{-r+1},\ldots,x_{-1},x_0,x_1,\ldots,x_r)\mapsto x_{-1}x_1+x_{-2}x_2+\cdots+x_{-r}x_r-\norm(x_0),\] where $x_0\in \A$ and $x_i\in\K$, for all $i\in\{-r,-r+1,\ldots,-1,1,2,\ldots,r\}$.
Then the null set of $\beta$ defines a quadric of Witt index $r$, whose natural point-line geometry is a polar space of rank $r$, which we denote by $\mathsf{B}_{r,1}(\K,\A)$. 

From the classification of buildings of type $\mathsf{F_4}$ in \cite{Tits:74}, we know that such a building is uniquely determined by a field $\K$ and a quadratic alternative unital division algebra $\A$ over $\K$.
We denote that building by $\mathsf{F_4}(\K,\A)$, where we usually substitute $\K$ and $\A$ with their sizes if they are finite.
Now, we assign the type function to the diagram in such a way that the symps of $\mathsf{F_{4,1}}(\K,\A)$ are isomorphic to the polar space $\mathsf{B_{3,1}}(\K,\A)$.
The building $\mathsf{F_4}(\K,\K)$ will sometimes be referred to as ``split''.
It is also characterised by the fact that the residues of vertices of type $1$ correspond to \emph{symplectic} polar spaces, that is, polar spaces defined by a non-degenerate alternating bilinear form, or equivalently, a null polarity.
If we define a \emph{hyperbolic line} of a polar space as the set  $(x^\perp\cap y^\perp)^\perp=\{x,y\}^{\perp\perp}$, for two non-collinear points $x$ and $y$, then, in a symplectic polar space, a hyperbolic line is an ordinary line of the ambient projective space which is not a line of the polar space.

We now define equator geometries.
Let $p,q$ be two opposite points of $\mathsf{F}_{4,1}(\K,\A)$, $i\in\{1,4\}$.
The \emph{equator} $E(p,q)$ is the set of points that are symplectic simultaneously to $p$ and $q$.
The intersection of $E(p,q)$ with the union of the symps through a given plane containing either $p$ or $q$ is, by definition, a line of the \emph{equator geometry}.
One checks that, replacing ``plane'' with ``maximal singular subspace contained in a symp'', this definition applied to the simply laced case provides the same equator geometries as defined earlier (this is proved explicitly in many cases in \cite{Sch-Mal:23}).

We are going to briefly need the extended equator geometry, but only for the split case $\A=\K$.
Let $E(p,q)$ be an equator of $\mathsf{F_{4,4}}(\K,\K)$ and define $\widehat{E}(p,q)$ as the set of points symplectic to at least two opposite points of $E(p,q)$.
Endow $\widehat{E}(p,q)$ with all lines of each equator geometry included in it.
Then we obtain the \emph{extended equator geometry}, also denoted by $\widehat{E}(p,q)$.
This time, $p,q\in\widehat{E}(p,q)$.
It is always isomorphic to $\mathsf{B_{4,1}}(\K,\K)$, see \cite{Lam-Mal:25}.
If $\K$ is perfect of characteristic~$2$, then notice that $\mathsf{B_{4,1}}(\K,\K)$ is a symplectic polar space.
\label{extEG}

\subsection{Generalised hexagons}\label{secgenhex}

Our results include buildings of type $\mathsf{G_2}$; the associated point-line geometries are better known as generalised hexagons, introduced by Tits \cite{Tits:59}.
For $n\geq  3$, a \emph{generalised $n$-gon}, or \emph{generalised polygon} if we do not want to specify $n$, is a point-line geometry $\Gamma=(X,\cL)$ such that the (bipartite) graph on $X\cup\cL$, with $x\in X$ adjacent to $L\in\cL$ if $x\in L$, has diameter $n$ and girth $2n$ (we call this graph the \emph{incidence graph of $\Gamma$}).
We also assume that every line has at least three points and every point is contained in at least three lines (thickness of the associated building).
Generalised $3$-gons are the same things as projective planes, and generalised $4$-gons, also known as generalised quadrangles, are polar spaces of rank $2$.
For more general background and results on generalised $n$-gons, see \cite{Mal:98}.
Finite generalised quadrangles are studied in detail in \cite{Pay-Tha:84}.
We recall the following definitions.
The \emph{order} of a generalised $n$-gon is the pair $(s,t)$ such that each line contains precisely $s+1$ points and each point is contained in precisely $t+1$ lines.
A \emph{spread} in a generalised quadrangle $\Gamma=(X,\cL)$ is a partition of $X$ into members of $\cL$.
If $\Gamma$ has order $(s,t)$, then a spread contains exactly $1+st$ lines.
A \emph{subpolygon} $\Gamma'$ of a generalised polygon $\Gamma$ is a generalised polygon whose points and lines are point and lines, respectively, of $\Gamma$, and elements of $\Gamma'$ have the same mutual distance in the incidence graph of $\Gamma'$ as in that of $\Gamma$.

Here, we are particularly interested in generalised hexagons, and more specifically in those that satisfy the \emph{Moufang condition}, as these are the counterparts of type $\mathsf{G_2}$ of the spherical buildings of rank at least $3$ (since these automatically satisfy such condition) and are the natural geometries for the simple algebraic groups of that type.
Recall from \cite{Tits-Wei:02} that a Moufang hexagon is determined by a field $\K$ and a quadratic Jordan division algebra $\J$ over $\K$; using the Bourbaki labelling of nodes for Dynkin diagrams, we define $\mathsf{G_{2,2}}(\K,\J)$ as the hexagon where the point rows are parametrised by $\J$ and the line pencils by $\K$.
There is perhaps ambiguity when $\J=\K$, but we take that ambiguity away by mentioning that, in this case, $\mathsf{G_{2,2}}(\K,\K)$ is the hexagon arising from a triality of type I$_{\id}$, as defined by Tits \cite{Tits:59}.
Such a hexagon shall be called the \emph{split Cayley hexagon} (as in \cite{Mal:98}), and the corresponding building is referred to as being \emph{split}.
The \emph{triality hexagon} is $\mathsf{G_{2,2}}(\K,\J)$, where $\J$ is a cubic Galois field extension of $\K$.
We will only need this hexagon in the finite case.
The finite triality hexagon has order $(q^3,q)$, for some prime power $q$.
In the finite case, we replace the field and the Jordan algebra with their sizes in the notation, just like we did for finite metasymplectic spaces.

Let $\Gamma=(X,\cL)$ be a generalised hexagon.
We use the notation and terminology of parapolar spaces to express the mutual position of two points, with the obvious meaning.
In particular, points are special if they are not collinear but they are collinear to a unique common point.
With that, a \emph{hyperbolic line} $H$ in $\Gamma$ is a set of mutual special points collinear to a common given point $p$ such that, for each point $q$ opposite $p$, we have $q^{\Join}\cap p^\perp=H$ as soon as $|q^{\Join}\cap H|\geq 2$.
If we call a point of $\Gamma$ \emph{close to} a line if it is collinear to a unique point of that line, then a \emph{distance-$3$ trace} in $\Gamma$ is the set of points close to two given opposite lines $L$ and $M$.
We denote it by $[L,M]_3$.
A distance-$3$ trace $[L,M]_3$ is called \emph{regular} if $[N,M]_3=[L,M]_3$ whenever $N$ is opposite $M$ and $|[N,M]_3\cap[L,M]_3|\geq 2$.
There are reasons to call a distance-$3$ trace in a split Cayley hexagon over a perfect field in characteristic~2 an \emph{imaginary line}, as we shall explain in the proof of \cref{geomlineshex}.
The terminology between parentheses in \cref{A} and \cref{B} above, namely \emph{ideal line}, is Ronan's terminology \cite{Ron:80}.

\subsection{Opposition and projections} \label{oppprojsec}

In this final paragraph of the preliminaries, we invoke some general theory about opposition, define projections, and note down some consequences which are rather interesting in our context because they unify some arguments across all types.

Let $F$ and $F'$ be opposite simplices in a spherical building $\Omega$.
For us, $F$ and $F'$ will almost always be vertices, but we choose to state and define things slightly more generally.
Then, by \cite[look up]{Tits:74}, for every chamber $C\supseteq F$ there exists a unique chamber $C'\supseteq F'$ at nearest (gallery) distance from $C$.
By \cite[look up]{Tits:74}, the map $C\mapsto C'$ induces an isomorphism from $\Res_\Omega(F)$ to $\Res_\Omega(F')$.
This means that for every vertex $v$ joinable to $F$, there exists a unique vertex $v'$ joinable to $F'$ and closest to $v$; we denote $v'=\proj^F_{F'}(v)$ and we call $v'$ the \emph{projection of $v$ onto $F'$}.
If $F$ is obvious or not important, we sometimes write $\proj_{F'}$ instead of $\proj_{F'}^F$. 

For example, the map in the proof of \cref{perp22perp} between the two opposite symps $\xi_1$ and $\xi_2$ of an exceptional hexagonic Lie incidence geometry $\Delta$, given on the points by ``being symplectic'', coincides with the pair of projections $\proj_{\xi'}^\xi$ and $\proj_\xi^{\xi'}$.
As a second example, the projection $\proj^x_y$ from a point $x$ to an opposite point $y$ in an exceptional hexagonic Lie incidence geometry maps a line $L$ through $x$ to the unique line through $y$ containing a point collinear to some point of $L$. 

The following notion will be very convenient to classify geometric lines.
Let $\Omega$ be any spherical building and let $v_1,v_2,v_3$ be three vertices of the same type.
Then, using the terminology of \cite{Kas-Mal:13}, we call $\{v_1,v_2,v_3\}$ a \emph{round-up triple} if no vertex of $\Omega$ is opposite exactly one of $v_1,v_2,v_3$.
Then, by definition, round-up triples in Lie incidence geometries correspond to round-up triples in the corresponding building.
We will make extensive use of round-up triples when classifying geometric lines in this paper. 

We have the following connection between global and local opposition. 

\begin{prop}[Proposition~3.29 of \cite{Tits:74}] \label{Tits}
Let $F$ and $F'$ be opposite simplices of a spherical building $\Omega$.
Let $v$ be a vertex of $\Omega$ adjacent to each vertex of $F$, and let $i$ be the type of $v$.
Then the type $i'$ of the vertex $\proj^{F}_{F'}(v)$ is the opposite in $\Res(F')$ of the opposite type of $i$ in $\Omega$.
Also, vertices $v\sim F$ and $v'\sim F'$ are opposite in $\Omega$ if, and only if, $v'$ is opposite $\proj^{F}_{F'}(v)$ in $\Res_\Omega(F')$.
  
\end{prop}

\begin{coro}\label{corolTits}
Let $F$ be a simplex of a spherical building $\Omega$.
Then a collection $T$ of vertices in $\Res(F)$ admits an opposite in $\Res(F)$ (viewed as a spherical building on its own) if and only if it admits an opposite vertex in $\Omega$.
Also, a set of vertices in $\Res(F)$ is a geometric line of $\Res(F)$ if, and only if, it is a geometric line of $\Omega$.
A triplet of vertices is a round-up triple in $\Res(F)$ if, and only if, it is a round-up triple in $\Omega$.
\end{coro}

\begin{proof}
If $v$ is a vertex in $\Res(F)$ opposite each vertex of $T$ in $\Res(F)$, then, for any simplex $F'$ opposite $F$, the vertex $\proj^{F}_{F'}(v)$ is opposite every member of $T$ by \cref{Tits}.
Conversely, let $v'$ be a vertex of $\Omega$ opposite each member of $T$.
Select $t\in T$.
Applying \cref{Tits}, we find a simplex $F'$ in $\proj^t_{v'}(F)$ opposite $F$.
Then, again by \cref{Tits}, $\proj^{F'}_F(v')$ is, in $\Res(F)$, opposite every member of $T$.
The second and third assertions now also follow. 
\end{proof}

If we want to use \cref{corolTits}, then we have to know something about blocking sets, geometric lines, or round-up triples in residues; if these are classical, then from \cite{Bos-Bur:66}, \cite{Bus-Mal:24}, and \cite{Kas-Mal:13}, we infer:

\begin{prop}\label{typeAD}\begin{compactenum}[$(1)$]\item
Let $T$ be a set of $q+1$ vertices of type $j$ in either $\mathsf{A_{n}}(q)$, $1\leq j\leq n$, $n\geq 2$, or $\mathsf{D_{n}}(q)$, $1\leq j\leq n$, $n\geq 4$, such that no vertex is opposite all of them.
Then $T$ is a line in the corresponding Lie incidence geometry $\mathsf{A_{n,j}}(q)$, or $\mathsf{D}_{n,j}(\K)$, respectively.
\item Every geometric line of $\mathsf{A_{n,j}}(\K)$, $1\leq j\leq n$, $n\geq 2$, for $\K$ an arbitrary skew field, is an ordinary line.
Hence, every round-up triple of points in such a geometry is contained in an ordinary line.
The same holds in $\mathsf{D_{n,j}}(\K)$, $1\leq j\leq n$, $n\geq 4$, for $\K$ an arbitrary field.
\end{compactenum}
\end{prop}

\begin{prop}\label{typeB}
\begin{compactenum}[$(1)$]
\item
Let $L$ be a geometric line of a Lie incidence geometry $\Gamma$ of type $\mathsf{B}_{n,j}$, $1\leq j\leq n$.
Then $L$ is either a line of $\Gamma$, or $j\neq n$, $\Gamma$ is the $j$-Grassmannian of a symplectic polar space $\Delta$, there exists a singular subspace $U$ of dimension $j-2$ of $\Delta$, and $L$ corresponds to a hyperbolic line of $\Res_\Delta(U)$, or $j=n$, all symps of $\Gamma$ are symplectic polar spaces of rank~$2$, and $L$ is a hyperbolic line in a symp.
   
\item
Let $T$ be a round-up triple of points of a Lie incidence geometry $\Gamma$ of type $\mathsf{B}_{n,j}$, $1\leq j\leq n$, and let $\Delta$ be the associated polar space.
Then $T$ is either contained in a line of $\Gamma$, or $j\neq n$, there exists a singular subspace $U$ of $\Delta$ of dimension $j-2$, and $T$ is contained in a hyperbolic line of $\Res_\Delta(U)$, or $j=n$, and $L$ is contained in a hyperbolic line of a symp of $\Gamma$.
\item
Let $T$ be a set of at most $q+1$ vertices of type $j$ in a finite (thick) building $\Delta$ of type $\mathsf{B}_n$, $1\leq j\leq n$, $n\geq 3$, such that each panel of cotype $j$ is contained in exactly $q+1$ chambers.
Suppose no vertex of $\Delta$ is opposite all members of $T$.
Then $T$ is either a geometric line in the corresponding Lie incidence geometry of type $\mathsf{B}_{n,j}$, or $q$ is a power of $2$ with even exponent and one of the following occurs.

\begin{compactenum}[(a)] \item $j=n-1$, $\Delta$ has order $(q,\sqrt{q})$ (viewed as a polar space) and $T$ is an ovoid in an ideal (symplectic) subquadrangle of order $(\sqrt{q},\sqrt{q})$ of a $\mathsf{B}_2$ residue, viewed as generalised quadrangle of order $(q,\sqrt{q})$.
\item $j=n$, $\Delta$ has order $(\sqrt{q},q)$ (viewed as a polar space) and $T$ is a spread in a full (dually symplectic) subquadrangle of order $(\sqrt{q},\sqrt{q})$ of a $\mathsf{B}_2$ residue, viewed as generalised quadrangle of order $(\sqrt{q},q)$.
\end{compactenum}
In particular, if $|T|\leq q$, then the set $T$ always admits a vertex opposite all its members.
\end{compactenum}
\end{prop}

Also, let us quote the following property, which we will use regularly.
\begin{prop}[Proposition~8.5 of \cite{Bus-Sch-Mal:24}]\label{3.30}
If every panel of a spherical building is contained in at least $s+1$ chambers, with $s\in\mathbb{N}$, then every set of $s$ chambers admits an opposite chamber.
In particular, every set of $s$ vertices admits an opposite vertex. 
\end{prop}

Noting that in any apartment each vertex has a unique opposite, the following assertion is immediate by considering an apartment through $S_1$ and $S_2$.
\begin{lem}\label{3.30b}
If $S_1$ and $S_2$ are two distinct simplices of a spherical building, then there exists at least one simplex opposite $S_1$, but not opposite $S_2$. 
\end{lem}

One more property we will use frequently is the following.

\begin{prop}\label{RUPtoGL}
Suppose every round-up triple of points in a Lie incidence geometry $\Delta$ is contained in a line.
Then each geometric line of $\Delta$ is an ordinary line. 
\end{prop}

\begin{proof}
Let $L$ be a geometric line of $\Delta$.
By \cref{3.30}, $L$ contains at least three elements.
Pick $x_1,x_2\in L$.
Clearly, every triple of elements of $L$ containing $x_1,x_2$ is a round-up triple and hence contained in a line $M$, which coincides with the unique line through $x_1$ and $x_2$.
Hence $L\subseteq M$.
If $L\neq M$, then each point opposite $x_1$ and not opposite some point of $M\setminus L$ would be opposite all members of $L$, a contradiction.
The proposition is proved.
 
\end{proof}

In the present paper, we will have to classify round-up triples in many geometries.
The following general properties will come in handy.

\begin{lem}\label{apartment}
Let $\{v_1,v_2,v_3\}$ be a round-up triple of vertices of some common type in a spherical building $\Delta$.
If $v_1$ and $v_2$ are joinable to a common vertex $v$, then also $v_3$ is joinable to $v$.
Hence, in this case, $\{v_1,v_2,v_3\}$ is a round-up triple in the residue of $v$. 
\end{lem}
\begin{proof}
Consider an apartment $\Sigma$ of $\Delta$ containing $v$ and $v_3$, and let $w$ be the unique vertex of $\Sigma$ opposite $v_3$.
By the definition of round-up triple, we may assume that $w$ is opposite $v_1$.
Let $\Sigma'$ be an apartment containing $w$ and the simplex $\{v,v_1\}$.
Let $\varphi:\Sigma'\to\Sigma$ be an isomorphism of complexes fixing both $w$ and $v$ (as required to exist by the definition of a building).
Then $\varphi(v_1)=v_3$ as opposition is preserved.
Hence $v$ is joinable to $v_3$, as $\varphi$ preserves the simplicial structure.
\end{proof}

\begin{lem}\label{apartment2}
Let $\{v_1,v_2,v_3\}$ be a round-up triple of vertices of some common type in a spherical building $\Delta$.
Suppose there are simplices $\{v,w_1\}$ and $\{v,w_2\}$ such that $w_i$ is joinable to $v_i$, $i=1,2$.
Suppose also that $v_1$ and $v_2$ are not joinable to a common vertex of the same type as $w_1$ or $w_2$.
Then there exists a vertex $w_3$ joinable to both $v$ and $v_3$, and the type of $w_3$ can be chosen equal to the type of either $w_1$ or $w_2$. 
\end{lem}
\begin{proof}
Consider an apartment $\Sigma$ of $\Delta$ containing $\{v,w_1\}$ and $v_3$, and let $w$ be the unique vertex of $\Sigma$ opposite $v_3$.
Assume, for a contradiction, that $v_1$ is opposite $w$.
Let $\Sigma^*$ be an apartment containing $\{v_1,w_1\}$ and $w$, and let $\varphi^*:\Sigma^*\to\Sigma$ be an isomorphism fixing $w$ and $w_1$.
Then $\varphi^*(v_1)=v_3$ by uniqueness of opposites in apartments.
Hence $w_1$ is joinable to $v_3$, and \cref{apartment} implies that it is also joinable to $v_2$, contradicting our assumptions.
We conclude that $w$ is opposite $v_2$.
Let $\Sigma'$ be an apartment containing $\{v_2,w_2\}$ and $w$, and let $\Sigma''$ be an apartment containing $\{v,w_2\}$ and $w$. 

Let $\varphi'':\Sigma''\to\Sigma$ be an isomorphism fixing $w$ and $v$, and denote $\varphi''(w_2)$ briefly as $w_3$.
Note that, by the definition of buildings, we may assume that $\varphi''$ is type-preserving, implying that $w_3$ has the same type as $w_2$.
Also, $w_3$ is joinable to $v$, as $w_2$ is.
Let $\varphi':\Sigma'\to\Sigma''$ be an isomorphism fixing $w_2$ and $w$.
Then $\varphi=\varphi''\circ \varphi':\Sigma'\to\Sigma$ is an isomorphism fixing $w$, hence mapping $v_2$ to $v_3$.
But $\varphi(w_2)=w_3$.
Hence $w_3$ is joinable to both $v_3$ and $v$.
Interchanging the roles of $v_1$ and $v_2$, we can also choose the type of $w_3$ to be the same as that of $w_1$. 

This completes the proof of the lemma. 
\end{proof}

Setting $v_2=w_2$ in the previous lemma, noting that the only vertex of the same type as $v_2$ joinable to $v_3$ is $v_3$ itself, and then also noting the symmetry of the situation, we obtain the following consequence.

\begin{coro}\label{apartment3}
Let $\{v_1,v_2,v_3\}$ be a round-up triple of vertices of some common type in a spherical building $\Delta$.
Suppose there exists a simplex $\{w_1,w_2\}$ such that $w_i$ is joinable to $v_i$, $i=1,2$.
Suppose also that $v_1$ and $v_2$ are not joinable to a common vertex of the same type as $w_1$ or $w_2$.
Then $v_3$ is joinable to both $w_1$ and $w_2$. 
\end{coro}

Finally, we note that we called certain objects ``far'' from each other.
It was always the case that these objects referred to vertices of the corresponding building, and either they were opposite (for example, a point and a symp in $\mathsf{E_{6,1}}(\K)$), or one vertex was joinable to a vertex opposite the other.
We will extend now this notion of \emph{far} to all such situations.
Hence two objects in a Lie incidence geometry will be called \emph{far (from each other)} if they correspond to vertices in the building and one vertex is joinable to a vertex opposite the other. 


From now on all buildings and geometries are finite, and hence $s$ will always denote a natural number.

\section{Points and lines in the exceptional minuscule geometries}\label{minuscule}

\subsection{Points of Lie incidence geometries of type $\mathsf{E_{6,1}}$}\label{E61}

\subsubsection{Blocking sets}

\begin{prop}\label{E61points}
If every line of a parapolar space $\Gamma$ of type $\mathsf{E_{6,1}}$ contains exactly $s+1$ points, then there exists a symp opposite each point of an arbitrary set $S$ of $s+1$ (distinct) points, except if these points are contained in a single line. 
\end{prop}

\begin{proof}
Let $S$ be a set of $s+1$ distinct points $p_{0}, \dots, p_{s}$ in $\Gamma$ that are not all on a common line.
Assume first that all points of $S$ are mutually collinear.
Then $p_{0}, \dots, p_{s}$ are either contained in a common $4$-space or in a common $5$-space.
Since both are residues in the corresponding building, \cref{corolTits} and \cref{typeAD} imply that $S$ is a line, proving the assertion. 

So, we may assume that at least two points of $S$ have distance $2$.
Let $p_{0}$ and $p_{1}$ be two non-collinear points of $S$.
Consider the symp $\xi_{0}=\xi(p_{0}, p_{1})$.
For each point of $S$ close to $\xi_0$, we choose an arbitrary point in $\xi_0$ collinear with it.
Let $S'$ be the set of points thus obtained, complemented with the points of $S$ contained in $\xi_0$.
By the main result of \cite{Bus-Mal:24}, in particular Lemma 3.5 therein, we find a point $b$ in $\xi_0$ non-collinear to each member of $S'$, and hence non-collinear to each member of $S$.
Let $\Xi$ be the set of symps $\xi(b,p_i)$, for $i\in\{0,1,\ldots,s\}$.
There are at most $s$ such different symps, as $\xi(b,p_0)=\xi(b,p_1)$.
Applying \cref{3.30} in $\Res(b)$, we find a symp $\xi$ through $b$ intersecting each member of $\Xi$ in exactly $b$.
Then, by \cref{lemE6}, the symp $\xi$ is far from all points of $S$. 

The proof is complete.
\end{proof}

\subsubsection{Geometric lines} 

Corollary 5.6 of \cite{Kas-Mal:13} states that geometric lines of Lie incidence geometries of type $\mathsf{E_{6,1}}$ are the same things as lines. 

\subsection{Lines of Lie incidence geometries of type $\mathsf{E_{6,1}}$}\label{E63}

We now look at the lines of Lie incidence geometries of type $\mathsf{E_{6,1}}$. 

\subsubsection{Blocking sets}

\begin{prop} \label{E61lines}
Let $\Delta$ be Lie incidence geometry of type $\mathsf{E_{6,1}}$ such that every line has $s+1$ points. 
Let $T=\{L_{0}, \dots, L_{s}\}$ be a set of $s+1$ different lines in $\Delta$ such that they do not form a line pencil in a plane. 
Then there exists a $4$-space opposite all lines $L_{0}, \dots, L_{s}$ in $\Delta$.
\end{prop}

\begin{proof}
We may put $L_{0}, \dots, L_{s}$ in respective symps not containing a common $4$-space.
Then (the dual of) \cref{E61points} yields 
a point $b$ opposite all of those symps.
Then \cref{linesEsixMaximalFourSpace} yields $3$-spaces $U_i=b^\perp\cap L_i^\perp$. 

In $\Res(b)\cong\mathsf{D_{5,5}}(\K)$, the $\<b,U_i\>$ are $3$-spaces. With $\Res(b)$
viewed as $\mathsf{D_{5,1}}(\K)$, they become lines.
If they do not form a planar line pencil, then \cref{typeAD} yields a line opposite all of them, which translates into a $4$-space $W$ through $b$ locally opposite each $\<p,U_i\>$.
\cref{Tits} implies that $W$ is opposite  each member of $T$.
So, we may suppose that the $\<b,U_i\>$ form a planar line pencil in $\Res(b)$, viewed as a polar space.
Since points and planes of that polar space correspond to symps and lines of $\Res(b)$, viewed as $\mathsf{D_{5,5}}(\K)$, we may assume that the $4$-spaces $\<b,U_i\>$ form the set of $4$-spaces through a given plane $\beta$ of $\Delta$ contained in a given symp $\xi$.

All $U_{i}$ intersect $\beta$ in lines $N_{i}$ not containing $b$.
Set $W_i=\langle U_{i}, L_{i} \rangle$.

We can find a $5$-space opposite all $W_i$
unless all $W_i$ intersect in a common plane $\alpha$.

\textbf{Case 1:} Suppose first that all $W_i$ do not
intersect in a common plane and denote by $B$ a $5$-space opposite all of them.
The last assertion of \cref{restE61} yields  
lines $L_{i}'$ in $B$ such that each point of $L_{i}$ is collinear to a unique point of $L_{i}'$,
but no point of $B - \bigcup_{i=0}^s L_{i}'$ is collinear to any point of any $L_{i}$. 

If the $L_{i}'$ do not form a planar line pencil in $B$,
then we can find a $3$-space $M$ in $B$ opposite each $L_{i}'$.
By \cref{345E61}, that $3$-space is contained in a unique maximal $4$-space that
we will denote by $C$, and, by \cref{3spacesuffices}, $C$ is opposite all $L_{i}$.
So suppose the $L_{i}'$ form a planar line pencil in $B$.
We will denote the corresponding plane as $\gamma$ and the intersection $L_0'\cap L_1'\cap\cdots L_s'$ as $d$.

\cref{554E61} yields a point $q_0\in\<b,U_0\>\setminus U_0$ collinear to a $3$-space of $B$ disjoint from $L'_0$, and hence intersecting $\gamma$ in a unique point $q_0'$.
Without loss of generality, we may assume that $q_0'\in L_1'$.
Hence there is a unique point $q_1\in L_1$ collinear to $q_0'$.
Since $q_0\notin U_1$, the symp $\xi(q_0,q_1)$ is well-defined and contains both $q_1'$ and $N_1$.
But then $N_1$, which is contained $\<L_1,U_1\>$, contains a point collinear to $q_1'$.
Since $b$ is collinear to all points of $N_1$ and to no point of $L_1$, these lines are disjoint.
This contradicts the last assertion of \cref{restE61}. Hence this case does not occur, the $L_i'$ do not form a planar line pencil and the assertion is proved in Case~1 . 

\textbf{Case 2:} Now suppose that all $W_i$ pairwise intersect in a common plane $\alpha$.
Let $i\in\{0,1,\ldots,s\}$.
Since $U_i\subseteq\xi$, the $5$-space $W_i$ intersects $\xi$ in a $4'$-space $V_i$.
Since, by \cref{pointsympE61}, $W_i$ is determined by $\xi$ and any point of $W_i\setminus V_i$, the plane $\alpha$ is contained in $\xi$.
Hence the $V_i$ form the set of $4'$-spaces through $\alpha$.
Since $U_i=b^\perp\cap V_i$, each line $N_i$ coincides with $L:=\alpha\cap b^\perp\subseteq U_i$. 

By the choice of $b$, the line $L_i$ is not contained in $\xi$, and a fortiori neither in $\alpha$.

Let $A_{0}$ be a $5$-space containing $L_{0}$ that does not contain the plane $\alpha$. Referring forward to \cref{longrootlemma}, in combination with the know result \cref{kasmal} (noting that $W_1$ is not opposite $W_2$), 
 we can find a $5$-space $B$ opposite all $5$-spaces $A_{0}$, $W_1, \dots, W_s$.

Again, since $\langle b, U_{i} \rangle$ is a $4$-space and $B$ a $5$-space,
we can find a point $q_1$ in $\langle b, U_{1} \rangle$ that is collinear to a $3$-space of $B$.
Since $U_1\subseteq W_1\equiv B$, we deduce $q_{1}\notin U_{1}$.  Let $L_0',L_1',\ldots,L_s'$ be defined as in Case~1. Then, as in Case~1, we may assume they form a planar line pencil (otheriwse we get the conclusion) and we can define $\gamma$ and $d$ as in Case~1. 
Since every $3$-space of $B$ intersects $\gamma$, $q_{1}$ is collinear to a point $q_{1}'$ in $\gamma$
that is collinear to some point $q_k$ of $L_{k}$, for some $k\in\{0,2,3\ldots,s\}$.

As in Case~1, $k\neq 0$ leads again to a contradiction.
Hence, considering $\xi(q_0,q_1)$, we see that $q'_1$ is collinear to some line $q_0q$, with $q\in L$. 
We now make some observations and define planes $\epsilon$ and $\epsilon'$.

\begin{itemize}
\item Since $W_0$ and $A_{0}$ have the line $L_{0}$ in common,
$A_{0} \cap W_0$ is a plane and hence they are adjacent.
Since $\delta:A_0\cap W_0$ is in $A_{0}$ and $A_{0}$ is opposite $B$, every point of $\delta$ is far from~$B$.

\item Since $W_0$ and $B$ cannot be opposite, but $W_0$ is adjacent to the opposite $A_0$ to $B$, \cref{joinjoin} yields a $5$-space $W$ intersecting both $W_0$ and $B$ in planes denoted by $\epsilon$ and $\epsilon'$, respectively.

\item Every point of $\epsilon$ is close to $B$, and every point of $\alpha$ is far from $B$.
Therefore, the planes $\alpha$ and $\epsilon$ cannot intersect.
\end{itemize}

As $L_0\subseteq A_0$, no point of $L_{0}$ is in $\epsilon$.
But since each point of $L_{0}$ has to be collinear to a point of $\epsilon'$,
and every point of $L_{0}$ is collinear to a unique point of $B$ on $L_{0}'$,
it follows that $L_{0}'$ is contained in $\epsilon'$.
In particular $d\in\epsilon'$.
It follows that $d^\perp\cap W_0$ is a $3$-space containing $\epsilon$, hence intersecting $\alpha$ in a unique point $e$.
Since $d\in L_i'$, $i=1,2,\ldots,s$, it is collinear to a unique point of $W_i$, and that point is on $L_i$.
hence $e\in L_1\cap L_2\cap\cdots\cap L_s$.
We conclude $\{e\}=\xi\cap L_i$.
Switching the roles of $L_0$ and $L_1$, we obtain $L_0\cap\xi=L_2\cap\xi=\{e\}$. 

Now $T$ belongs to $\Res_\Delta(e)$ and the assertion follows from \cref{corolTits}.
  
\end{proof}

\subsubsection{Geometric lines}

Now we classify geometric lines in Lie incidence geometries of type $\mathsf{E_{6,3}}$.
In order to do so, we classify round-up triples of lines in $\mathsf{E_{6,1}}(\K)$, for an arbitrary field $\K$.

\begin{lem}\label{E6RuT0} 
Let $\{L_1,L_2,L_3\}$ be a round-up triple of lines in the exceptional Lie incidence geometry $\mathsf{E_{6,1}}(\K)$ such that $L_1$ and $L_2$ have at least one point in common.
Then exactly one of the following holds.
\begin{compactenum}[$(i)$]
\item $L_1=L_2=L_3$;
\item $L_1,L_2,L_3$ are three lines in a common planar line pencil.
\end{compactenum}
\end{lem}

\begin{proof}
If $L_1=L_2\neq L_3$, then a $4$-space opposite $L_1$ but not opposite $L_3$ (which exists by \cref{3.30b}) violates the defining property of a round-up triple.
Hence, if $L_1=L_2$, then $(i)$ holds. 

Now assume $L_1\cap L_2=\{x\}$.
By \cref{apartment}, $x\in L_3$.
Then \cref{corolTits} and \cref{typeAD} imply that  
$L_1,L_2,L_3$ are contained in a plane, and $(ii)$ follows. 
\end{proof}

\begin{lem}\label{E6RuT1}
Let $\{L_1,L_2,L_3\}$ be a round-up triple of pairwise disjoint lines in the exceptional Lie incidence geometry $\mathsf{E_{6,1}}(\K)$.
Let $y_1$ be an arbitrary point of $L_1$.
If $y_1$ is symplectic to some point of $L_2$, then it is collinear to some point of $L_3$.
\end{lem}

\begin{proof}
Suppose $y_i\in L_i$, $i=1,2$, with $y_1\pperp y_2$.
We claim that there exists a point $u\in y_1^\perp\cap y_2^\perp$ not collinear to any point of $L_3$.
Indeed, suppose each point of $y_1^\perp\cap y_2^\perp$ is collinear to some point of $L_3$.
If that point of $L_3$ is unique, then, since $\{y_1,y_2\}^{\perp\perp}=\{y_1,y_2\}$ in each hyperbolic quadric, either $y_1$ or $y_2$ belongs to $L_3$, a contradiction.
If $L_3$ were contained in $\xi(y_1,y_2)$, then $y_1$ would be collinear to a point of $L_3$, and the assertion would follow.
If $L_3$ intersected $\xi(y_1,y_2)$ in a point $z$, then for each point $z_3\in L_3$ we would have $z_3^\perp\cap\xi(y_1,y_2)\subseteq z^\perp\cap\xi(y_1,y_2)$, and this only contains $y_1^\perp\cap y_2^\perp$ if $z\in\{y_1,y_2\}$, a contradiction again.
Hence $L_3$ is disjoint from $\xi(y_1,y_2)$.
Now \cref{lemE6b} implies that $L_3$ is collinear to a plane $\pi$ of $\xi(y_1,y_2)$, and all points of $\xi(y_1,y_2)$ collinear to some point of $L_3$ are collinear to all points of $\pi$, a contradiction since $y_1^\perp\cap y_2^\perp$ is only collinear to $\{y_1,y_2\}$.
The claim is proved. 

Now, by \cref{Tits}, a $4$-space $U$ through $u$ locally opposite the projection of $L_3$ onto $u$ is opposite $L_3$, but not opposite either $L_1$ or $L_2$, as $u$ is collinear to both $y_1\in L_1$ and $y_2\in L_2$.
This final contradiction proves the lemma.
\end{proof}

\begin{lem}\label{E6RuT2}
Let $\{L_1,L_2,L_3\}$ be a round-up triple of lines in the exceptional Lie incidence geometry $\mathsf{E_{6,1}}(\K)$.
Then exactly one of the following holds.
\begin{compactenum}[$(i)$]
\item $L_1=L_2=L_3$;
\item $L_1,L_2,L_3$ are three lines in a common planar line pencil.
\end{compactenum}
\end{lem}

\begin{proof} By \cref{E6RuT0}, the statement is true if, for some $i,j\in\{1,2,3\}$, $i\neq j$, the lines $L_i$ and $L_j$ have a point in common.
So we may assume that $L_1,L_2,L_3$ are pairwise disjoint. 

By \cref{E6RuT1}, we may assume that some point $x_1\in L_1$ is collinear to some point $x_2\in L_2$. 
Set $M:=x_1x_2$.
Since $L_1$ and $L_2$ are disjoint, \cref{apartment2} implies that $M$ also intersects $L_3$.
So $M\cap L_3$ is also a point $x_3$.
Now select $y_1\in L_1\setminus\{x_1\}$.
By \cref{E6RuT1}, we may again assume that $y_1$ is collinear to some point $y_2$ of $L_2$.
Then, as in the previous paragraph, the line $L_3$ has a point $y_3$ in common with $y_1y_2$.
Clearly $y_3\neq x_3$, as otherwise $L_1,L_2$ are contained in the plane generated by the lines $M$ and $y_1y_2$, contradicting our assumption that $L_1\cap L_2=\varnothing$.
Similarly, $y_2\neq x_2$.
If all of $x_1,x_2,x_3$ are collinear to all of $y_1,y_2,y_3$, then all of $L_1,L_2,L_3$ are contained in a singular $3$-space, which is a residue of a simplex of type $\{2,5,6\}$ in the corresponding building.
Then \cref{corolTits} and \cref{typeAD} lead to the contradiction that $L_1,L_2,L_3$ are coplanar and hence not disjoint.
So we may assume that $x_1$ is not collinear to $y_2$.
But then all of $L_1,L_2,L_3$ are contained in the symp $\xi(x_1,y_2)$, which is again a residue, and so \cref{corolTits} and \cref{typeAD} again lead to a contradiction.

The lemma is proved.
\end{proof}

\begin{prop}\label{GLE63}
Every geometric line of $\mathsf{E_{6,3}}(\K)$ is an ordinary line. 
\end{prop}

\begin{proof}
This follows directly from \cref{RUPtoGL} and \cref{E6RuT2}. 
\end{proof}

\subsection{Points of Lie incidence geometries of type $\mathsf{E_{7,7}}$}\label{E77}

\subsubsection{Blocking sets}

\begin{prop}\label{E77points}
If every line of a parapolar space $\Gamma$ of type $\mathsf{E_{7,7}}$ contains exactly $s+1$ points, then there exists a point at distance $3$ from each point of an arbitrary set $S$ of $s+1$ (distinct) points, except if these points are contained in a single line. 
\end{prop}

\begin{proof}
Suppose $S$ is not a line.
Set $S=\{p_0,\ldots,p_s\}$.
We distinguish several cases. 
\begin{compactenum}[$(i)$]
\item \emph{Suppose all points in $S$ are pairwise collinear.} Then $S$ is contained in a (maximal) singular subspace.
As this corresponds to a residue in the corresponding building, \cref{corolTits} and \cref{typeAD} lead to the conclusion. 

\item \emph{Suppose some pair of points from $S$ is symplectic.} Suppose $p_0$ and $p_1$ are symplectic and let $\xi$ be the symp containing them.
Let the set $T$ of points of $\Delta$ consist of the points of $S$ in $\xi$, the set of points of $\xi$ collinear to a point of $T$ far from $\xi$, and, for each point $p_i$ of $S$ close to $\xi$, an arbitrary point of $\xi$ collinear to $p_i$.
Then $T$ is a set of at most $s+1$ points not forming a line (as $p_0$ and $p_1$ belong to $T$ and are not collinear).
Hence, by \cref{typeAD} (or more precisely, \cite[Lemma~3.5]{Bus-Mal:24}), we find a point $b\in \xi$ not collinear to any member of $T$.
Hence, for each point $p_i\in S$ not far from $\xi$, there is a unique symp $\xi_i$ containing $b$ and $p_i$.
Since $\xi=\xi_0=\xi_1$, there are at most $s$ such symps.
\cref{3.30} yields a line $L$ through $b$ locally opposite each such symp, which means that each point of $L\setminus\{b\}$ is far from each such symp.
This implies that each point of $L\setminus\{b\}$ is opposite each point of $S$ that is not far from $\xi$.
But since $b$ is opposite every point of $S$ that is far from $\xi$ by construction, the line $L$ contains points opposite each member of $S$.
Hence, each point $p_i$, $i=0,1,\ldots,s$, has a unique projection $p_i'$ onto $L$.
Since $b=p_0'=p_1'$, there is at least one point of $L$ opposite each point of $S$. 

\item \emph{Some pair of points from $S$ is opposite and there are no symplectic pairs in $S$.} It follows that collinearity is an equivalence relation in $S$.
Let $C\subseteq S$ be a corresponding equivalence class, which we may choose such that  $|S\setminus C|\geq 2$.
Let $\xi$ be a symp either containing $C$ (if the latter is contained in a singular $5$-space or a singular $5'$-space; we assume $p_0\in C$), or containing at least one point, say again $p_0$, of $C$ and intersecting the $6$-space spanned by $C$ in a $5'$-space (if $C$ generates a $6$-space; \cref{preE77} allows for this).
Either way, let $W$ be the singular subspace of $\xi$ obtained by intersecting $\xi$ with the singular subspace generated by $C$.
Let $S'\neq\varnothing$ be the set of points of $\xi$ collinear to a point of $S\setminus C$.

Select $q\in S'$ arbitrarily and suppose $q\perp p_1\in S\setminus C$.
As $|S\setminus C|\geq 2$, we can select a line $M$ in $\xi$ through $q$ disjoint from $W$ and not containing $S'$.
With that, $M$ contains at least one point $b\perp q$ such that $b\notin S'$ and $b$ is not collinear to a member of $C$.
It follows that $b$ is not collinear to any point of $S$.
So we can consider the set $\Xi$ of symps containing $b$ and some point of $S$.
Since $p_0$ and $p_1$ are opposite, the symps $\xi=\xi(b,p_0)$ and $\xi(b,p_1)$ only intersect in the line $bq$.
Hence, in the residue of $b$, the set $\Xi$ does not correspond to the set of symps containing a given maximal singular subspace.
Hence, the dual of \cref{E61points} yields a line $L$ through $b$ locally opposite each member of $\Xi$.
As in $(ii)$ above, the line $L$ is far from each member of $S$, but $p_0$ and $p_1$ project to the same point of $L$, implying that there is at least one point on $L$ opposite every member of $S$.
\qedhere
\end{compactenum}
\end{proof}

\subsubsection{Geometric lines}

Corollary 5.6 of \cite{Kas-Mal:13} states that a geometric line of a Lie incidence geometry of type $\mathsf{E_{7,7}}$ is the same thing as a line. 

\subsection{Lines of Lie incidence geometries of type $\mathsf{E_{7,7}}$}\label{E76}

We now look at the lines of Lie incidence geometries of type $\mathsf{E_{7,7}}$.

\subsubsection{Blocking sets}

\begin{prop}\label{E77lines}
If every line of a parapolar space $\Gamma$ of type $\mathsf{E_{7,7}}$ contains exactly $s+1$ points, then there exists a line opposite each member of an arbitrary set $T$ of $s+1$ (distinct) lines, except if these lines form a planar line pencil. 
\end{prop}

\begin{proof}
Since symps are vertices of type $1$ in the corresponding building, \cref{kasmal} and \cref{geomlinesexc}, proved independently, allow us to consider symps $\widehat{\xi}_0, \widehat{\xi}_1, \dots, \widehat{\xi}_s$, such that $L_i \in \widehat{\xi}_i$ and such that there exists a symp $\xi$ opposite every $\widehat{\xi}_i$, for $i \in \{0,1,\dots,s\}$.
Then, in view of \cref{symp-sympE77}, every point of any $\widehat{\xi}_i$ is collinear to a unique point of $\xi$.
The points of $\xi$, which are in bijection with the points of $L_i \in \widehat{\xi}_i$, form a line again that we will denote by $L_i'$.
Suppose the lines $L_i'$ do not form a planar line pencil.
Then, by Theorem~A of \cite{Bus-Mal:24}, or more in particular Lemma~3.15 therein (see also \cref{typeAD}), we can find a line $M$ in $\xi$ that is locally opposite all $L_i'$ at $\xi$ (viewed as a residue).
With \cref{Tits}, it follows that $M$ is opposite all $L_i$.

Now suppose the $L_i'$ do form a planar line pencil in $\xi$ and denote the plane by $\pi$.
Let $p$ be the point in which all $L_i'$ intersect and let $q$ be a point in $\xi \setminus \pi$, not collinear to $p$.
Then $\proj_{\pi} (q)$ (in $\xi$) is a line $K$, which intersects each $L_i'$ in some point $p_i'$.
Set $p_i=\proj_{L_i}p_i'$.
Since $p_i$ and $q$ are both collinear to $p_i'$, they are symplectic and we denote the symp spanned by $p_i$ and $q$ by $\xi_i$.
If the $\xi_i$ do not all intersect in a common $5$-space, then we can find a line $M$ through $q$ locally opposite all $\xi_i$ and, by \cref{Tits}, $M$ is opposite all $L_i$.

So we may assume from now on that all $\xi_i$ intersect in a common $5$-space $Q$.
Note that they are pairwise distinct, as every $\xi_i$ intersects $\xi$ in exactly the line $qp_i'$.
Let $q_0$ be a point on $qp_0'$ distinct from $q$ and $p_0'$.
The point $q_0$ is in the plane $\alpha:=\langle q, K \rangle$ and with that $\proj_{\pi} (q_0) = \proj_{\pi} (q)=K$.
Hence, $q_0$ is symplectic to each $p_i$ as well.
We will denote the symps spanned by $p_i$ and $q_0$ by $\xi_i^0$.
Since $q_0$ is on $qp_0'$, $\xi_0^0$ coincides with $\xi_0$.
Similarly to before, we may assume that all $\xi_i^0$ intersect in a common $5$-space $Q^0$.
For each $j \in \{1, 2, \dots, s\}$, we have that $\xi_j \neq \xi_j^0$ and $\xi_j^0$ intersects $\xi$ in exactly the line $q_0p_j'$.

The $5$-spaces $Q$ and $Q^0$ are distinct, but both contained in $\xi_0$.
Since symps in $\Delta$ are polar spaces of type $\mathsf{D_{6,1}}$, two $5$-spaces in $\xi_0$ intersect in even codimension and hence either in nothing, a line, a $3$-space, or they coincide. 

Note that $p_j$, for $j \in \{1, 2, \dots, s\}$, cannot be contained in $\xi_0$, since otherwise $\xi_0 = \xi(q,p_j)=\xi_j$, a contradiction.
We have that $\proj_{\xi_0}(p_j)$ (in $\Delta$) is a $5'$-space $A_j$.
We know that $\proj_{Q} (p_j)$ is a $4$-space $U_j \subseteq A_j$ and $\proj_{Q^0} (p_j)$ is a $4$-space $U_j^0 \subseteq A_j$.
Since $U_j\cap U_j^0\subseteq Q\cap Q^0$, it follows that $Q\cap Q^0=U_j\cap U_j^0=:V$ is a $3$-space independent of $j\in\{1,2,\ldots,s\}$.

Let $q_1$ be a point of $qp_1'\setminus\{q,p_1'\}$.
Without loss of generality, assume $p_2'\in q_0q_1$.
Then the symps $\xi(q_1,p_1)$ and $\xi(q_1,p_2)$ contain both $V$, and it follows again that all symps $\xi(q_1,p_i)$ contain $V$.
Since they also contain $q_1$, we find $q_1\perp V$.
Since $q_1$ was essentially arbitrary, we conclude that $V$ is the intersection of all symps defined by $p_i$ and some point of $\alpha\setminus K$.
Moreover, $\alpha\perp V$ and $W:=\<\alpha,V\>$ is a (maximal) $6$-space.
Since both $p_0$ and $V$ are in the intersection of $\xi(q,p_0)$ and $\xi(q_1,p_0)$, we also deduce $p_0\perp V$. 

We observe that $A_j = \proj_{\xi_0}(p_j) = \<\proj_{Q} (p_j) , \proj_{Q^0} (p_j)\> = \<U_j , U_j^0\>$ and that all $A_j$ contain $V$.
We define $U_0$ and $U_0^0$ as the $4$-spaces $\proj_Q(p_0)$ and $\proj_{Q^0}(p_0)$, respectively.
By the foregoing, $V=U_0^0\cap U_0$. 

The point $q$ is contained in $Q$ and the point $q_0$ is contained in $Q^0$.
In a $5$-space, there can only be $s+1$ different $4$-spaces which contain a given $3$-space.
Since none of the $U_i$ ($U_i^0$ respectively) can contain $q$ ($q_0$ respectively), at least two of the $U_i$ ($U_i^0$ respectively) have to coincide.
If $U_m = U_n$ for some $m, n \in \{0,1,\dots,s\}$, then it follows that $U_m^0 = U_n^0$, since otherwise $A_m$ and $A_n$ would intersect in a $4$-space.
We conclude that two $A_i$, for $i \in \{0,1,\dots,s\}$, have to coincide.
Let $m, n \in \{0,1,\dots,s\}$ be such that $A_m = A_n$.
Since, by \cref{preE77}, $6$-spaces cannot intersect in $5'$-spaces, $\langle p_m, A_m \rangle$ and $\langle p_n, A_m \rangle$ have to be equal.
Hence $p_m$ and $p_n$ are collinear.

We now claim that all $p_i$ are contained in a common symp.
Indeed, by \cref{preE77}, there is a unique symp $\zeta$ containing the $5'$-space $\<V,K\>$.
Now $\xi_0$ and $\zeta$ have the $4$-space $\<p_0',V\>$ in common, and hence they intersect in the $5$-space $\<p_0,p_0',V\>$, since this is the unique $5$-space of $\xi_0$ containing $\<p_0',V\>$.
We get $p_0\in\zeta$.
Similarly, $p_i\in\zeta$, for all $i\in\{1,2,\ldots,s\}$, and the claim is proved. 

We now vary $K$ over all lines of $\pi$ not containing $p$.
We obtain $s^2$ distinct pairs of collinear points, where each pair of points is contained in two different lines $L_i,L_j$, $i\neq j$, and no point of such pair is collinear to $p$.
Since we only have $\binom{s+1}{2}=\frac12s(s+1)<s^2$ pairs of lines, there must exist $n,m\in\{0,1,\ldots,s\}$, $n\neq m$, and distinct point pairs $\{a_n,a_m\}$ and $\{b_n,b_m\}$, with $a_n,b_n\in L_n$ and $a_m,b_m\in L_m$, such that $a_n\perp a_m$ and $b_n\perp b_m$.
There are two cases. 

\emph{Case $(a)$: $a_n\neq b_n$ and $a_m\neq b_m$.}\\
In this case, it is easy to see that $L_n$ and $L_m$ are contained in a common symp $\xi^*$.
We claim that all points of $L_0\cup L_1\cup\cdots\cup L_s$ are mutually collinear.
Indeed, suppose not, then there exists some symplectic pair $\{a,b\}$ of points on that union, contained in a unique symp $\xi^{**}$. Consider the set $\Lambda=\{\widehat{\xi}_0, \widehat{\xi}_1, \dots, \widehat{\xi}_s,\xi^*,\xi^{**}\}\setminus\{\widehat{\xi}_{n},\widehat{\xi}_m\}$. By possibly re-choosing $\widehat{\xi}_i$, for $i\in\{0,1,\ldots,s\}\setminus\{m,n\}$, we may assume that not all members of $\Lambda$ contain the same maximal singular subspace. Moreover, $|\Lambda|\leq s+1$, and so, as before, there exists a  symp $\xi'$ opposite each member of $\Lambda$. 
But the projection of the union of the $L_i$ onto $\xi'$ contains a symplectic pair, namely the projection of $\{a,b\}$ from $\xi^{**}$ onto $\xi'$, hence cannot be a planar line pencil.
As before, this leads to a line opposite all of the $L_i$.
The claim is proved.
Hence, $T$ is contained in a singular subspace, and \cref{corolTits} and \cref{typeAD} lead to the assertion. 

\emph{Case $(b)$: without loss of generality $a_n=b_n$ and $a_m\neq b_m$.}\\
In this case, $a_n$ is collinear to the line $L_m$.
Let $c_n\in L_n$ be different from $a_n$ and not collinear to $p$.
If $c_n\perp a_m$, then we are back in Case $(a)$.
So we may assume that $c_n$ is not collinear to $a_m$.
Then, interchanging the role of $K$ with that of the line of $\pi$ containing the respective points collinear to $c_n$ and $a_m$, we see that there exists a unique symp $\zeta_n$ containing $c_n$, $a_m$, and a point $c_i$ of each $L_i$, $i\in\{0,1,\ldots,s\}\setminus\{n,m\}$.
It also contains $L_n$.
We may now re-choose $\widehat{\xi_n}$ as $\zeta_n$.
As $c_n$ and $a_m$ are not collinear, this again leads, as in Case $(a)$ above, to a line opposite each $L_i$, and the proposition is proved. 
\end{proof}

\subsubsection{Geometric lines}

Now we classify geometric lines in Lie incidence geometries of type $\mathsf{E_{7,6}}$.
This will follow from the classification of round-up triples of lines.

\begin{lem}\label{E7RuT0} 
Let $\{L_1,L_2,L_3\}$ be a round-up triple of lines in the exceptional Lie incidence geometry $\Gamma$ of type $\mathsf{E_{7,7}}$, such that $L_1$ and $L_2$ intersect.
Then exactly one of the following holds.
\begin{compactenum}[$(i)$]
\item $L_1=L_2=L_3$;
\item $L_1,L_2,L_3$ are three lines in a common planar line pencil.
 
\end{compactenum}
\end{lem}

\begin{proof}
Clearly, if $L_1=L_2$, then also $L_3=L_1$, since otherwise there exists a line opposite $L_3$ and not opposite $L_1$.
So we may assume $L_1\cap L_2=\{x\}$.
By \cref{apartment}, we see that $x\in L_3$, and \cref{corolTits}, in combination with \cref{E6RuT2}, implies $(ii)$. 
\end{proof}

\begin{lem}\label{E7RuT1} 
Let $\{L_1,L_2,L_3\}$ be a round-up triple of pairwise disjoint lines in an exceptional Lie incidence geometry of type $\mathsf{E_{7,7}}$.
Then no point of $L_2$ is collinear to any point of $L_1$. 
\end{lem}

\begin{proof}
Let, for a contradiction, $M$ be a line joining a point $x_1\in L_1$ to a point $x_2\in L_2$.
Note that $L_1\neq M\neq L_2$.
Since $L_1$ and $L_2$ are disjoint, we can apply \cref{apartment2} and find 
that $M$ intersects $L_3$.
Set $x_i:=M\cap L_i$, $i=1,2,3$. 

Assume, for a contradiction, that $x_1$ is symplectic to some point $y_3\in L_3$.
Set $\xi:=\xi(x_1,y_3)$.
Noting that $M\subseteq \xi$, we want to apply \cref{apartment3}. Since $L_1$ and $L_2$ are disjoint, we only have to check that $L_1$ and $L_2$ are not contained in a common symp. If that were the case, then also $L_3$ would be contained in that symp and $\{L_1,L_2,L_3\}$ would be, by \cref{corolTits}, a round-up triple of lines in polar space of type $\mathsf{D}_6$, implying by \cref{typeAD} that $L_1$ and $L_2$ are not disjoint after all, a contradiction. Hence \cref{apartment3} yields $x_1\in L_2\subseteq \xi$, a contradiction again. 
Hence, $x_1$ is collinear to each point of $L_3$.
But then again, every line through $x_1$ intersecting $L_3$ meets $L_2$, and so $L_2,L_3$ are contained in a common plane, hence intersecting, contradicting our assumptions.
The lemma is proved. 
\end{proof}

\begin{prop}\label{E7RuT3}
Let $\{L_1,L_2,L_3\}$ be a round-up triple of lines in the exceptional Lie incidence geometry $\Gamma$ of type $\mathsf{E_{7,7}}$.
Then exactly one of the following holds.
\begin{compactenum}[$(i)$]
\item $L_1=L_2=L_3$;
\item $L_1,L_2,L_3$ are three lines in a common planar line pencil.
 
\end{compactenum}
\end{prop}

\begin{proof}
In view of \cref{E7RuT0} and \cref{E7RuT1}, it suffices to show that no round-up triple $\{L_1,L_2,L_3\}$ exists for which no point of $L_i$ coincides with or is collinear to any point of $L_j$, $i,j\in\{1,2,3\}$, $i\neq j$.
So suppose, for a contradiction, such a triple does exist.
Select $x_1\in L_1$.
Then there exists $x_2\in L_2$ symplectic to $x_1$.
Set $\xi:=\xi(x_1,x_2)$.
Suppose, for a contradiction, that some point $x_3\in L_3$ is opposite some point $y_{12}\in x_1^\perp\cap x_2^\perp$.
Then we can find a line through $y_{12}$ opposite $L_3$, but that line is certainly not opposite either $L_1$ or $L_2$, as it contains a point $y_{12}$ collinear to points of $L_1$ and $L_2$.
Hence no point of $L_3$ is opposite any point of $x_1^\perp\cap x_2^\perp$.
If some point $x_3\in L_3$ were far from $\xi$, this would imply that the unique point $x_3'$ of $\xi$ collinear to $x_3$ is collinear to all of $x_1^\perp\cap x_2^\perp$, forcing $x_3'\in\{x_1,x_2\}$, contradicting our assumption that no point of $L_3$ is collinear or equal to any point of $L_1\cup L_2$.
Hence all points of $L_3$ are close to $\xi$. 
Note also that, by \cref{apartment2} and the fact that $L_1$ and $L_2$ are not contained in a common symp (as we assume no point of $L_1$ is collinear to any point of $L_2$), some point $x_3\in L_3$ belongs to $\xi$.
By interchanging the roles of $L_3$ and $L_i$, $i=1,2$, we see that each point of $L_1\cup L_2$ is close to $\xi$.
Since $\xi$ is hyperbolic, there exists a point $y\in L_2^\perp\cap x_3^\perp\setminus x_1^\perp$.
Let $M$ be a line through $y$ locally opposite $\xi$ and select $z\in M\setminus\{y\}$.
However, if $y$ is not collinear to all points of $L_3$, then we (re)choose $M$ locally not opposite the symp through $y$ and $L_3$.
In any case, $z$ is not opposite any point of $L_2\cup L_3$.
Since it is opposite $x_1$, we find a line $K$ through $z$ opposite $L_1$.
But $K$ is not opposite either $L_2$ or $L_3$ by the properties of $z$, a contradiction.

The lemma is proved.
\end{proof}

We conclude:

\begin{prop}\label{GLE76}
Every geometric line of $\mathsf{E_{7,6}}(\K)$ is an ordinary line. 
\end{prop}

\begin{proof}
This follows directly from \cref{RUPtoGL} and \cref{E7RuT3}. 
\end{proof}

\section{Points and lines of hexagonic Lie incidence geometries}\label{hexagonic}

\subsection{Points of hexagonic Lie incidence geometries}\label{hexpoints}

Here we prove Main Results A and B for the points  of the exceptional hexagonic geometries. 

\subsubsection{Blocking sets.
Reduction to geometric lines}

\begin{prop}\label{longrootlemma}
Let $\Gamma$ be an exceptional hexagonic geometry with $s+1$ points per line.
Then a given set $T$ of $s+1$ points of $\Gamma$ admits an opposite point if, and only if, $T$ is not a geometric line of $\Gamma$. 
\end{prop}

\begin{proof}
Clearly, if $T$ is a geometric line, then $T$ does not admit any point opposite all its points.
Now suppose $T$ does not admit any point opposite all of its points.
We show that $T$ is a geometric line.
Suppose, for a contradiction, that $T$ is not a geometric line.
Then there exists a point $x$ not opposite at least two points of $T$, but opposite at least one point of $T$, and we shall call each such point a \emph{spoilsport}.
Suppose $x$ is not opposite $r\geq 2$ points of $T$, with $r\leq s$, and let $S$ be that set of points.
We adopt the following notation.
For each point $p\in S$ not equal or symplectic to $x$, we denote the line through $x$ closest to $p$ by $L_{x,p}$.
If $p\in S$ is symplectic to $x$, then we denote by $\cL_{x,p}$ the set of lines of $\xi(x,p)$ through $x$.

Note that each point $z\neq x$ on any line $K$ through $x$ locally opposite some member of $\cL_{x,p}$ is special to $p$. 

\begin{compactenum}[$(i)$]
\item \emph{Suppose $x\in T$.} For each point $p\in S$ symplectic to $x$, we choose an arbitrary line $L_{x,p}\in\cL_{x,p}$.
Then, by 
\cref{typeB} for $\mathsf{F_4}$, and by \cref{3.30} for the other cases,
we find a line $L\ni x$ locally opposite all of $L_{x,p}$, for $p$ ranging through $S\setminus\{x\}$.
Since $T\setminus S$ contains at most $s-1$ elements, there is a point $x'$ on $L$ opposite all members of $T\setminus S$.
If $S$ contains at least one point collinear or symplectic to $x$, then $x'\notin T$ is a spoilsport.
If $S\setminus\{x\}$ only contains points special to $x$, then some point of $L$ at distance $2$ of at least one member of $T\setminus S$ is a spoilsport not contained in $T$.
Hence we may assume from now on that $x\notin T$. 
\item \emph{Suppose $x\notin T$ and $S$ contains at least one point collinear or symplectic to $x$.}  Again, we choose an arbitrary line $L_{x,p}\in\cL_{x,p}$ 
for each $p\in S$ symplectic to $x$, 
and we find a line $L$ through $x$ locally opposite all $L_{x,p}$, $p\in S$.
There are two possibilities.
First assume that $S$ contains at least one point special to $x$.
Then we select a point $x'\in L$ at distance $2$ from at least one member of $T\setminus S$, and we see that $x'$ is a spoilsport not collinear and not symplectic to any point of $T$.
Secondly, assume that $S$ does not contain any point special to $x$.
Then we select a point $x''\in L\setminus\{x\}$ distinct from the at most $(s+1)-r\leq s-1$ points at distance $2$ from some member of $T\setminus S$.
Then $x''$ is opposite every member of $T\setminus S$ and special to each member of $S$, and hence $x''$ is a spoilsport.
So, in both cases we constructed a spoilsport not collinear and not symplectic to any point of $T$.
So from now we may assume that $x$ is special to each point of $S$. 
\item \emph{Suppose $x$ is special to each point of $S$.} Then we can find a line $L$ locally opposite each $L_{x,p}$, with $p\in S$, and a point $y\in L\setminus\{x\}$ opposite each member of $T\setminus S$.
The point $y$ is opposite each member of $T$, a contradiction.
\end{compactenum}  
We conclude that $T$ is a geometric line. 
\end{proof}

\cref{longrootlemma} reduces the classification of point sets $T$ in a finite exceptional hexagonic geometry, where $T$ has the size of a line and does not admit a point opposite each of its members, to the classification of geometric lines in such geometries.
This is the goal of the next theorem.
It completes the partial classification given in \cite{Kas-Mal:13}, which we now briefly repeat.

\begin{prop}[Theorem 6.5 in \cite{Kas-Mal:13}]\label{kasmal}
Let $L$ be a geometric line in an exceptional hexagonic geometry $\Gamma$.
Then exactly one of the following cases occurs. 
\begin{compactenum}[$(1)$]
\item $L$ is an ordinary line of $\Gamma$;
\item $L$ is a hyperbolic line in a symplecton of $\Gamma$ isomorphic to a symplectic polar space (and this only occurs in the hexagonic geometries of type $\mathsf{F_{4,4}}$ that arise from a split building of type $\mathsf{F_4}$);
\item $L$ consists of mutually opposite points. 
\end{compactenum}
\end{prop}

In view of \cref{kasmal}, it remains to classify geometric lines in exceptional hexagonic geometries consisting of mutually opposite points. 

\subsubsection{Classification of geometric lines}

The following lemma will be very efficient for such classification.

\begin{lem}\label{reducingtoequator}
Each geometric line $L$ containing opposite points of any (exceptional) hexagonic geometry $\Gamma$ is a geometric line of any equator geometry of $\Gamma$ containing at least two points of $L$.
\end{lem}

\begin{proof}
Let $x,y$ be two points of the geometric line $L$, consisting of mutually opposite points of $\Gamma$.
We claim that no point of $L$ is special to any point of $E(x,y)$, the equator geometry with poles $x$ and $y$.
Indeed, suppose $z\in L$ is special to $u\in E(x,y)$.
Extend the unique path $z\perp [u,z]\perp u$ to a path $z\perp[u,z]\perp u\perp v$, with $v\Join[u,z]$.
Then $v$ is opposite $z\in L$, but since $x\pperp u\perp v$, \cref{joinjoin} implies that $v$ is not opposite $x$.
Similarly, $v$ is not opposite $y$, a contradiction to $L$ being a geometric line.
The claim is proved.

Since we assume that all points of $L\setminus\{x,y\}$ are opposite $x$ and $y$, no such point belongs to $E(x,y)$, and, by \cref{joinjoin}, no such point is collinear to any point of $E(x,y)$. Finally, by the defintion of geometric line, no point of $L$ is opposite any point of $E(x,y)$ as both $x$ and $y$ are not oppoiste any point of $E(x,y)$.

Hence we have shown that each point of $L$ is symplectic to each point of $E(x,y)$.
Taking two opposite points $a,b\in E(x,y)$, this implies $L\subseteq E(a,b)$.
Since opposition in $E(x,y)$ as a Lie incidence geometry coincides with the opposition inherited from $\Gamma$, the assertion follows. 
\end{proof}

Counterexamples to the converse of \cref{reducingtoequator} will be given in type $\mathsf{F_4}$ (see the proof of the next theorem). 

We can now prove the announced classification.

\begin{theorem}\label{geomlinesexc}
Let $L$ be a geometric line in an exceptional hexagonic geometry $\Gamma$.
Then $L$ does not consist of mutually opposite points. 
\end{theorem}

\begin{proof}
By \cref{reducingtoequator}, the non-existence of geometric lines consisting of mutually opposite points in hexagonic geometries of types $\mathsf{E_6}$, $\mathsf{E_7}$ and $\mathsf{E_8}$ follows from the non-existence of such geometric lines in the Lie incidence geometries of types $\mathsf{D_{6,2}}$ and $\mathsf{A_{5,\{1,5\}}}$. 

The former case is taken care of by \cref{typeAD}.
In the latter case, by taking again equator geometries, see the last paragraphs of \cref{eqgeom}, we reduce the question to the case $\mathsf{A_{3,\{1,3\}}}$.
Then by \cref{reducingtoequator}, we see 
that $L$ consists of incident point-plane pairs in $\PG(3,\K)$ with the point ranging over a given line $K$ and the plane determined by the point and a given line $K'$ skew to $K$.
Consider a point $x\notin K\cup K'$ and a plane $\alpha$ through $x$ intersecting both $K$ and $K'$ in respective unique points, say $y$ and $y'$, respectively.
Then $\{x,\alpha\}$ is not opposite $\{\<K',x\>\cap K,\<K',x\>\}$ and not opposite $\{y,\<y,K'\>\}$, but opposite every other member of $L$, a contradiction. 

This shows that no geometric line of $\Gamma$ consists of opposite points, for $\Gamma$ of types~$\mathsf{E_{6,2},E_{7,1}}$ or~$\mathsf{E_{8,8}}$. 

Hence we may suppose that $\Gamma$ has type $\mathsf{F_4}$.
In that case, a geometric line $T$ consisting of mutually opposite points is a geometric line of the polar space $\Delta$ of rank $3$ corresponding to a point residual of $\Gamma$.
It follows from \cite[Lemma~4.8]{Kas-Mal:13} that  $\Delta$ is a symplectic polar space (hence $\Gamma$ is split---but possibly over a non-perfect field) and $T$ is a hyperbolic line.
By the obvious transitivity of the automorphism group on the set of hyperbolic lines, we may assume that every hyperbolic line of each equator geometry is a geometric line.
Now let $x,y$ be two opposite points of $\Gamma$ and let $E(x,y)$ be the corresponding equator geometry.
Let $\xi$ be a symp corresponding to a line $L$ of $E(x,y)$.
Assume first that the underlying field is not perfect of characteristic $2$.
Then, again by \cite[Lemma~4.8]{Kas-Mal:13}, there exists a point $a\in\xi$ either collinear to at least two points of $L$ but not all, or not collinear to any point of $L$.
Also, more precisely, since $\Delta$ is split, $\xi$ is 
a polar space corresponding to a quadric in $\PG(6,\K)$, $L$ is the intersection of the perps of two opposite lines, and so, $a$ is collinear to either $0$ or $2$ points of $L$.
Let $b$ be a point of $\Gamma$ far from $\xi$ and symplectic to $a$.
Then $b^{\not\equiv}$ intersects every hyperbolic line of $E(x,y)$ in one or all its points (as these are all assumed to be geometric lines), $L$ in $0$ or $2$ points, and, by the above argument for $a$ applied to other appropriate points, intersecting every line in either $0$, $1$, $2$ or all points.
We view $E(x,y)$ in its natural embedding in $\PG(5,\K)$.
Let $\pi$ be a non-singular plane of $\PG(5,\K)$ (with respect to the underlying non-degenerate alternating form) containing $L$.
We can choose $\pi$ such that the unique point $x_L\in L$ for which $\pi\subseteq p_L^\perp$ is opposite $b$.
Let $\ell\neq x_L$ be another point on $L$ opposite $b$ in $\Gamma$ (which exists  by our assumptions above and the fact that there are at least four points on a line---indeed, $\mathbb{F}_2$ is perfect of characteristic $2$, and so $|\K|\geq 3$) and let $h_1,h_2,h_3$ be three distinct hyperbolic lines of $E(x,y)$ in $\pi$ through $\ell$.
Then there are unique points $a_1\in h_1$, $a_2\in h_2$ and $a_3\in h_3$ not opposite $\ell$.
Suppose first that $a_1,a_2,a_3$ lie on the same line $L'$ of $\pi$.
Then the whole line $L'$ belongs to $b^{\not\equiv}$, and hence at least one point $L\cap L'$ of $L$ does.
But then two points of $L$ do, and connecting that second point, which does not coincide with $x_L$, with all points of $L'$ leads to $\pi\subseteq b^{\not\equiv}$, a contradiction.
Hence we may assume that $a_1,a_2,a_3$ are not contained in the same line of $\pi$.
Then it is easy to see that joining with hyperbolic lines yields all points of $\pi$, except for $x_L$.
Since $|\K|>2$, this is again a contradiction to $|b^{\not\equiv}\cap L|\in\{0,2\}$. 

Now assume that $\K$ is perfect of characteristic $2$.
We may embed every equator geometry in an extended equator geometry $\widehat{E}$, which is then isomorphic to a symplectic polar space of rank $4$ (see \cref{extEG}).
By \cite[Corollary~5.38]{Sch-Sas-Mal:18}, there exists a point $b$ of $\Gamma$ such that $H:=b^{\not\equiv}\cap\widehat{E}$ is a (hyperbolic) polar subspace of type $\mathsf{D_{4,1}}$.
Note that we still may assume that every hyperbolic line is a geometric line.
Then $H$ is also a geometric hyperplane of the ambient projective space $\PG(7,\K)$ of $\widehat{E}$, and hence coincides with $p^\perp$ for some point $p\in\widehat{E}$ (where the perp $\perp$ is now with respect to the symplectic polar space).
This is clearly a contradiction.

The theorem is proved. 
\end{proof}

\subsection{Lines of exceptional hexagonic Lie incidence geometries}\label{hexlines}

\subsubsection{Two lemmas in the residues} 

We begin with two results in the point residuals of hexagonic geometries.
The first one summarises earlier findings. 
\begin{lem}\label{class2}
Let $\Delta$ be either a Lie incidence geometry of type $\mathsf{A_{5,3}}$, $\mathsf{D_{6,6}}$ or $\mathsf{E_{7,7}}$, or a dual polar space of rank $3$.
Suppose each line has exactly $s+1$ points.
Suppose also that $\Delta$ is not isomorphic to $\mathsf{B_{3,3}}(\sqrt{s},s)$.
Then a set of at most $s+1$ points of $\Delta$ admits no common opposite point if, and only if, the points form a geometric line.
In particular, if there are at most $s$ points, or if there exists a point opposite at least one point of the set, and not opposite at least two points of the set, then the set admits an opposite point. 
\end{lem}

\begin{proof}
This follows from Main Results~A and~B of \cite{Bus-Mal:24} for types $\mathsf{A_{5,3}}$, $\mathsf{D_{6,6}}$, and for dual polar spaces, and from \cref{E77points} and \cite[Corollary~5.6]{Kas-Mal:13} for type $\mathsf{E_{7,7}}$.
\end{proof}

\begin{lem}\label{ADEdistances}
Let $\Delta$ be either a Lie incidence geometry of type $\mathsf{A_{5,3}}$, $\mathsf{D_{6,6}}$ or $\mathsf{E_{7,7}}$, or a dual polar space of rank $3$.
Suppose each line contains precisely $s+1$ points and let $p$ be a point.
Let $Q:=\{q_1,\ldots,q_\ell\}$ be a set of points not containing $p$, $\ell\le s$.
Then there exists a point $p'$ opposite each member of $Q$, and not opposite $p$. 
\end{lem}

\begin{proof}
We first construct a symp $\xi$ through $p$ far from each member of $Q$.
If $q\in Q$ is opposite $p$, each symp through $p$ qualifies.
If $q\in Q$ is symplectic to $p$, let $K_q$ be an arbitrary line through $p$ in the symp containing $p$ and $q$; if $q\perp p$ then let $K_q$ be the line containing $p$ and $q$.
Since $\ell\leq s$, we infer from \cref{3.30} and \cref{typeB} that there exists a symp $\xi$ through $p$ locally opposite all $K_q$, $q\in Q$.
Then $\xi$ is far from each member of $Q$.
Now the same references yield a point $p'$ in $\xi$ (locally) opposite in $\xi$ each intersection $\xi\cap q^\perp$, for $q\in Q$.
The point $p'$ is opposite each member of $Q$ and not opposite $p$.
\end{proof}

\subsubsection{Description of mutual positions}

We will describe the mutual position $\delta(L,M)$ of two lines $L$ and $M$ of an exceptional hexagonic Lie incidence geometry with four parameters $(a,b,c,d)$, chosen in the set $\{=,\perp,\pperp,\Join,\equiv\}$, where $a, b, c, d$ are defined as follows.
If there is a unique point on $L$ closest to $M$, we call it $x$; otherwise, $x$ is an arbitrary point on $L$.
Similarly, if there is a point on $M$ nearest to $L$, call it $y$; if not, but there is a point on $M$ nearest to $x$, call this $y$.
If not, then $y$ is an arbitrary point of $M$.
Let $x'$ be any point on $L$ distinct from $x$.
If there is a point on $M$ nearest to $x'$, and it is different from $y$, call it $y'$; otherwise, $y'$ is any point on the second line distinct from $y$.
Then $a$ is the relation between $x$ and $y$, while $b$ is the relation between $x$ and $y'$.
Also, $c$ is the relation between $x'$ and $y$, whereas $d$ is the relation between $x'$ and $y'$.
It will turn out that such a 4-tuple unambiguously determines the mutual position.

\begin{remark}
Of course the symbols $=,\perp,\pperp,\Join$ and $\equiv$ describe the possible relations of pairs of points of a hexagonic geometry, and hence, without going into the group theoretical details, also the relations between the root groups of a long root subgroup geometry (see \cite{Tim:01}). Indeed, the symbol $\perp$ means that the two root groups have trivial commutator and their product is a union of root groups. The symbol $\pperp$ means that the two root groups have trivial commutator, but their product is not the union of root groups. The symbol $\Join$ means that the commutator of the two root groups is a root group itself. Finally, the symbol $\equiv$ means that the two root groups generate an $\SL_2$.  In the finite case, this defines an association scheme $\Pi$. This also induces an association scheme $\Lambda$ on the set of lines. However, the authors are not aware of a (general) method to derive the parameters, and in particular the number of classes, or the ``distance''-distribution diagram, of $\Lambda$ from $\Pi$. This seems highly non-trivial, as we will see that even the number of classes depends on the type of the geometry---the answer is different for type $\mathsf{F}_4$ and a fortiori also for type $\mathsf{G}_2$.   Our methods will be entirely geometric.
\end{remark}

In a shorthand alternative notation, we write $0$ for $=$, $1$ for $\perp$, $\frac32$ for $\pperp$, $2$ for $\Join$ and $3$ for $\equiv$.
The \emph{inverse} of $\delta(L,M)$ is $\delta(M,L)$.
The dual of $\delta(L,M)$ is $\delta(M,K)$, for the line $K$ opposite $L$ in any apartment containing $L$ and $M$.
On the level of point distances, $0$ is dual to $3$, $1$ to $2$, and $\frac32$ is self-dual. We sometimes call $\delta(L,M)$ the \emph{distance} between $L$ and $M$.

There are basically four classes of positions $\delta(L,M)$, if one takes into account the homogeneity with respect to the points of the lines $L$ and $M$.
The classes are the following (where we use the notation of the previous paragraphs).

\textbf{Class I --- Completely homogeneous}

In this class, every point of either line has the same distance to each point of the other line.
All positions here are equal to their inverse.
The cases are:

\begin{itemize}
\item[$(1111)$] $(\perp,\perp,\perp,\perp)$: the two lines $L,M$ span a singular $3$-space.
\item[$(\frac32 \frac32 \frac32 \frac32)$] $(\pperp,\pperp,\pperp,\pperp)$: each point of each line is symplectic to each point of the other line.
\end{itemize}

The dual of $(1111)$ is the following:
\begin{itemize}
\item[$(2222)$] $(\Join,\Join,\Join,\Join)$: each point of each line is special to each point of the other line.
\end{itemize}

\textbf{Class II --- Projection homogeneous}

In this class, each point of each line has a unique projection onto the other line; the distances between corresponding points are constant, and the other distances as well.
All distances are their own inverse.
The cases are:

\begin{itemize}

\item[$(0110)$] $(=,\perp,\perp,=)$: the lines are equal, that is, $L=M$.

\item[$(1 \frac32 \frac32 1)$] $(\perp,\pperp,\pperp,\perp)$: the lines are opposite in a symp. 

\item[$(\frac32 2 2 \frac32)$] $(\pperp,\Join,\Join,\pperp)$: the lines are each other's projection from a symp to an opposite symp. 

\item[$(2332)$] $(\Join,\equiv,\equiv,\Join)$: the lines are opposite, that is, $L\equiv M$.
\end{itemize}

\textbf{Class III --- Symmetric non-homogeneous}

In this class, both lines contain a unique projection point with respect to the other line.
Moreover, all the positions are their own inverse again.
The cases are:

\begin{itemize}
\item[$(0111)$] $(=,\perp,\perp,\perp)$: the lines are coplanar.
\item[$(011 \frac32)$] $(=,\perp,\perp,\pperp)$: the lines meet and determine a unique symp.
\item[$(0112)$] $(=,\perp,\perp,\Join)$: the lines meet and are locally opposite.
\item[$(111\frac32)$] $(\perp,\perp,\perp,\pperp)$: the lines are special in (the line-Grassmannian of) a symp.
\item[$(1 \frac32 \frac32 \frac32)$] $(\perp,\pperp,\pperp,\pperp)$: the projection point $x$ of $M$ onto $L$ is contained in a symp $\xi$ with $M$, while $L$ is locally close to $\xi$ at $x$; the same holds with the roles of $L$ and $M$ interchanged.
\item[$(1 \frac32 \frac32 2)$] $(\perp,\pperp,\pperp,\Join)$: the projection point $x$ on $L$ is contained in a (unique) symp $\xi$ with $M$, while $L$ is locally far from $\xi$ at $x$; the unique line in $\xi$ collinear to $L$ is also collinear to the line connecting the two projection points $x$ and $y$; the same holds with the roles of $L$ and $M$ interchanged.
\item[$(\frac32 \frac32 \frac32 2)$] $(\pperp,\pperp,\pperp,\Join)$: the projection points $x$ and $y$ are symplectic, the lines $L$ and $M$ are locally close to the symp $\xi(x,y)$, and the projections of the lines $L,M$ onto $\xi(x,y)$ (which are maximal singular subspaces of $\xi(x,y)$) intersect in a unique point.
\item[$(\frac32 2 2 2)$] $(\pperp,\Join,\Join,\Join)$: the points $x$ and $y$ are symplectic, the lines $L$ and $M$ are locally far from $\xi(x,y)$, and the projections of the lines $L,M$ onto $\xi(x,y)$, which are lines themselves, are $\xi(x,y)$-special.
\item[$(12 2 3)$] $(\perp,\Join,\Join,\equiv)$: the lines $L,M$ are each other's projection from a point $x'$ or $y'$ to an opposite point $y'$ or $x'$, respectively. 
\item[$(\frac32 2 2 3)$] $(\pperp,\Join,\Join,\equiv)$: the lines are locally far from the symp $\xi$ determined by the projection points $x$ and $y$ (which are symplectic), and the projections of the lines onto $\xi$ are $\xi$-opposite lines. 
\item[$(2 2 2 3)$] $(\Join,\Join,\Join,\equiv)$: There is a pair of opposite points $x'\in L$ and $y'\in M$, and the projection of $L$ (or $M$) onto $y'$ (or $x'$) is locally symplectic to $M$ (or $L$) at $y'$ (or $x'$, respectively); equivalently, $L$ and $M$ lie in opposite symps and the projection of $L$ (or $M$) onto the symp through $M$ (or $L$)  is special to $M$ (or $L$, respectively) in (the line Grassmannian of) that symp. 
\end{itemize}

\textbf{Class IV --- Asymmetric positions}

Up to now, in all $4$-tuples, the second and third entry coincided.
This is going to change now.
The --- final --- class that we consider in this paragraph contains the asymmetric mutual positions, that is, those that do not coincide with their inverse.
There are four cases with two projection points.
These cases are:

\begin{itemize}
\item[$(1 \frac32 1 2)$] $(\perp,\pperp,\perp,\Join)$: the line $L$ is coplanar with $y$; the line $M$ and the plane $\<L,y\>$ are locally far at $y$.
\item[$(1 1 \frac32 2)$] $(\perp,\perp,\pperp,\Join)$: the line $M$ is coplanar with $x$; the 
line $L$ and the plane $\<M,x\>$ are locally far at $x$.
\item[$(1 \frac32 2 2)$] $(\perp,\pperp,\Join,\Join)$: the line $M$ and the point $x$ are in a unique symp $\xi$; the line $L$ is locally far from $\xi$ at $x$ and locally opposite the line $xy$ at $x$ (and $x$ and $y$ are collinear).
\item[$(1 2 \frac32 2)$] $(\perp,\Join,\pperp,\Join)$: the line $L$ and the point $y$ are in a unique symp $\xi$; the line $M$ is locally far from $\xi$ at $y$ and locally opposite the line $xy$ at $y$ (and $x$ and $y$ are collinear).
\end{itemize}

Finally, there are four cases where only one line has a projection point.
Hence, there will be only two distinct distances around, and the corresponding $4$-tuples have the shape $(a,a,b,b)$ or $(a,b,a,b)$ (where these are each other's inverse).
In a Lie incidence geometry of type $\mathsf{A_{5,3}}$, $\mathsf{D_{6,6}}$ or $\mathsf{E_{7,7}}$, or a dual polar space of rank $3$, we call a point and a line \emph{almost far} if every point of the line is symplectic to the point (however, this does not exist in dual polar spaces). 

\begin{itemize}
\item[$(1 \frac32 1 \frac32)$] $(\perp,\pperp,\perp,\pperp)$: the line $L$ is coplanar with $y$ (which is the projection point of $L$ on $M$); the line $M$ and the plane $\<L,y\>$ are locally almost far at $y$.
\item[$(1 1 \frac32 \frac32)$] $(\perp,\perp,\pperp,\pperp)$: the line $M$ is coplanar with $x$ (which is the projection point of $M$ on $L$); the line $L$ and the plane $\<M,x\>$ are locally almost far at $x$.
\item[$(\frac32 \frac32 2 2)$] $(\pperp,\pperp,\Join,\Join)$: the line $M$ is contained in a symp $\xi$ close to the projection point $x$ of $M$ onto $L$, and $x$ is collinear to a line of $\xi$ that is $\xi$-opposite $M$.
\item[$(\frac32 2 \frac32 2)$] $(\pperp,\Join,\pperp,\Join)$:  the line $L$ is contained in a symp $\xi$ close to the projection point $y$ of $L$ onto $M$, and $y$ is collinear to a line of $\xi$ that is $\xi$-opposite $L$
\end{itemize}
 
We note that dual distances are obtained from each other by interchanging and dualising the first and fourth entry, and dualising the second and third entry, except for the ``projection homogeneous'' cases, where one has to interchange the first two entries, as well as the last two, and take inverses. For instance, the duals of  $(1\frac3222)$  and $(\frac32 \frac32  2 2)$ are $(1 \frac32  12)$ and
$(1 \frac321\frac32)$, respectively, and the dual of $(1 \frac32 \frac32 1)$ is $(\frac32 2 2 \frac32)$.

We now have the following result. 

\begin{lem}\label{mutualpositions}
Let $L$ and $M$ be two arbitrary lines of an exceptional hexagonic Lie incidence geometry $\Delta$ of uniform symplectic rank $r$.
Then $\delta(L,M)$ is one of the $4$-tuples enumerated above in \emph{\textbf{Class I}} up to \emph{\textbf{Class IV}}.
All cases occur, except for metasymplectic spaces, where the following positions cannot occur: $(1111)$, $(\frac32 \frac32 \frac32 \frac32)$, $(2222)$, $(1 \frac32 \frac32 \frac32)$, $(\frac32 \frac32 \frac32 2)$, $(1 \frac32 1 \frac32)$, $(1 1 \frac32 \frac32)$, $(\frac32 \frac32 2 2)$ and $(\frac32 2 \frac32 2)$.
 
\end{lem}
\begin{proof}
We note that existence of a given mutual position is equivalent to the existence of its dual.

\textbf{Part I.} It is convenient to first consider the case where all points of $L$ have the same distance to all points of $M$.
Then clearly we have one of the three completely homogeneous cases.
The existence of $(1111)$ is easy: consider two lines in a common singular subspace of dimension at least $3$.
Conversely, clearly, if $\delta(L,M)=(1111)$, then $L$ and $M$ span a singular subspace of dimension~3.
Since these do not exist in metasymplectic spaces, this mutual distance occurs if and only if $\Delta$ is not a metasymplectic space.
By duality, the same holds for $(2222)$. 

Let $\delta(L,M)=(\frac32\frac32\frac32\frac32)$.
Pick $x,x'\in L$ and $y,y'\in M$, $x\neq x'$, $y\neq y'$.
If $y'$ were collinear to only a line of $\xi(x,y)$, then, by \cref{generalspecial}$(ii)$, $x$ and $y'$ would be special, a contradiction (and it follows that $\Delta$ does not have type $\mathsf{F_4}$).
It follows that $x^\perp\cap M^\perp$ is a singular subspace $U$ of dimension $r-2$.
We claim that $x'^\perp\cap U$ is $(r-5)$-dimensional.
To fix the ideas, suppose $\Delta$ has type $\mathsf{E_{8,8}}$.
Since $x'^\perp\cap\xi(x,y)$ is a $6'$-space and $\<x,U\>$ is a $6$-space, we have that $x'^\perp\cap \<x,U\>$ contains a line $K$.
In $\Res_\Delta(K)$, we have a point $z'$ (corresponding to $x'$) close to each symp through a given $4$-space $W$ (corresponding to $\<x,U\>$).
If $z'^\perp\cap W=\varnothing$, then there exist non-collinear points $a$ and $b$ in different symps through $W$, both collinear to $z'$.
The symp $\xi(a,b)$ contains $z'$ and a plane $\alpha$ of $W$; hence $z$ is collinear to a line of $W$ after all, a contradiction.
So $z'^\perp\cap W\neq\varnothing$ and, by parity, it is a line.
The claim follows.
It is now easy to see that, \begin{compactenum}[$(i)$]
\item in case of $\mathsf{E_{8,8}}$, $L$ and $M$ arise from opposite lines in the residue of a plane (which is a parapolar space of type $\mathsf{D_{5,5}}$);
\item in case of $\mathsf{E_{7,1}}$, $L$ and $M$ arise from opposite lines in a given convex subspace isomorphic to a parapolar space of type $\mathsf{D_{5,5}}$ in a given point residual;
\item in case of $\mathsf{E_{6,2}}$, $L$ and $M$ are opposite lines in a given convex subspace that is a parapolar space of type $\mathsf{D_{5,5}}$.
\end{compactenum}
The last claim might be the least straightforward, as in this case $L^\perp\cap M^\perp=\varnothing$.
So let us prove this case as an example (the other cases are then easier because the parapolar spaces in the respective residues are simpler; they have types $\mathsf{D_{5,5}}$ and $\mathsf{D_{6,6}}$, respectively).
Translated to type $\mathsf{E_{6,1}}$, we have to prove that, if each $5$-space through a given plane $\alpha$ intersects each $5$-space through another given plane $\beta$ in exactly a point, then either $\alpha$ meets $\beta$ in a point at which they are locally opposite, or $\alpha$ and $\beta$ are contained in a symp in which they are (locally) opposite.
So suppose $\alpha$ and $\beta$ do not intersect.
Then one checks that, if $U_1,U_2$ are two distinct $5$-spaces through $\alpha$ and $W_1,W_2$ two distinct $5$-spaces through $\beta$, the points 
$p_{ij}=U_i\cap W_j$, $i,j\in\{1,2\}$, form a quadrangle.
That quadrangle is contained in a unique symp $\xi$ that contains both $\alpha$ and $\beta$.
If the latter are not $\xi$-opposite, then, 
arguing in the polar space $\mathsf{D_{5,1}}(\K)$, we find $4'$-spaces through them that intersect in a plane; hence this yields adjacent $5$-spaces through them, a contradiction.

This concludes the completely homogeneous case.
We obtain all members of Class I.

\textbf{Part II.} Next we consider the case where each point of $L\cup M$ has a unique nearest point on the other line.
It is easy to deduce that the distances between such nearest points are always the same.
This distance can be equal, collinear, symplectic or special, in which case the other pairs are collinear, symplectic, special or opposite, respectively (use \cref{joinjoin} for instance).
Then it is easy to see that the lines are equal, opposite in a symp, the projection of each other from opposite symps, or opposite, respectively.
Hence we obtain precisely all cases of Class II. 

\textbf{Part III.} Having done the more homogeneous cases separately, we can proceed to consider the smallest distance that can occur between points of $L$ and $M$.
In order to do so, we let $p\in L$ and $q\in M$ be points at minimal distance. 

\textbf{Case 1: $p=q$.} Here the lines $L$ and $M$ meet in $p=q$, and we clearly have only the three possibilities $(0111)$, $(011\frac32)$ and $(0112)$ of Class III.
Existence is trivial in these cases.

\textbf{Case 2: $p\perp q$.} We set $K$ equal to the line through $p$ and $q$.
We now consider the different possible mutual positions of $L$ and $K$, and of $K$ and $M$.
First suppose that $K$ and $M$ are coplanar; say they span the plane $\alpha$.
Then $\alpha$ and $L$ are contained in a common symp if and only if one of the following two possibilities occurs: \begin{compactenum}[$(1)$] \item $L\perp M$; then we have case $(1111)$, \item $|L^\perp\cap M|=1$; then case $(111\frac32)$ occurs.
\end{compactenum} 
So we may assume that no point of $L\setminus\{p\}$ is collinear to any point of $\alpha\setminus\{p\}$.
Let $\xi$ be a symp containing $\alpha$.
There are again two possibilities.
\begin{compactenum}[$(1)$] \item $L^\perp\cap\xi$ is a line $N$.
Then $N$ is not contained in $\alpha$, so that there is a unique point $q'\in M$ collinear to all points of $N$.
Then $q'\pperp p'$, for all $p'\in L\setminus \{p\}$, and $q''\Join p'$, for all $q''\in M\setminus\{q'\}$ and all $p'\in L\setminus\{p\}$.
We get $(11\frac322)$.
\item $L^\perp\cap\xi$ is a maximal singular subspace $U$.
Then $U\cap M=\varnothing$, and each point of $M$ is symplectic to each point of $L\setminus\{p\}$.
We obtain $(11\frac32\frac32)$.
 
\end{compactenum}
If $L$ and $K$ are coplanar, then we obtain the inverse distances $(1111)$, $(111\frac32)$, $(1\frac3212)$ and $(1\frac321\frac32)$.
Hence we may assume that $K$ is not coplanar with either $L$ or $M$.
Pick $p'\in L\setminus\{p\}$ and $q'\in M\setminus\{q\}$.
If both pairs $\{p',q\}$ and $\{p,q'\}$ are special, then we have $(1223)$.
So we may assume $\{p,q'\}$ is symplectic.
Again, there are some possibilities. 
\begin{compactenum}[$(1)$]
\item $L^\perp\cap\xi(p,q')$ is a line $N$ not $\xi(p,q')$-opposite $M$.
Then our assumptions imply that $q\perp N$, and so we obtain $(1\frac32\frac322)$.
\item $L^\perp\cap\xi(p,q')$ is a line $N$ which is $\xi(p,q')$-opposite $M$.
Clearly, this gives rise to $(1\frac3222)$.
\item $L^\perp\cap\xi(p,q')$ is a maximal singular subspace $U$.
Then we clearly have $(1\frac32\frac32\frac32)$.
\end{compactenum}

The case where $L$ and $K$ are contained in a common symp gives additionally rise to the inverse $(12\frac322)$ of $(1\frac3222)$.
This takes care of all distances beginning with ``collinear''. 

\textbf{Case 3: $p\pperp q$.} Let $\xi$ be the symp containing both $p$ and $q$.
Note that neither $L$ nor $M$ is contained in $\xi$, as otherwise we are back in the previous case. 
There are a few possibilities.
\begin{compactenum}[$(1)$]
\item Both $L^\perp\cap\xi$ and $M^\perp\cap\xi$ are maximal singular subspaces, and they intersect in a subspace of dimension at least $1$.
Then we are in the homogeneous case $(\frac32\frac32\frac32\frac32)$, which we already discussed in detail in Part I. 
\item Both $L^\perp\cap\xi$ and $M^\perp\cap\xi$ are maximal singular subspaces, and they intersect in exactly a point.
By Part I, we are not in case $(\frac32\frac32\frac32\frac32)$; hence we have distance $(\frac32\frac32\frac322)$. 
\item Both $L^\perp\cap\xi$ and $M^\perp\cap\xi$ are maximal singular subspaces, and they are disjoint.
Select $p'\in L\setminus\{p\}$ and $q'\in M\setminus\{q\}$.
We claim that $p'\pperp q'$.
Indeed, the only alternative is $p'\Join q'$.
If so, let $p'\perp r\perp q'$.
Then, by considering symps through points of $U:=L^\perp\cap\xi$ and $q'$, we see that $r\perp U$; likewise, $r$ is collinear to each point of $M^\perp\cap\xi$, a contradiction.
Hence this leads to $(\frac32\frac32\frac32\frac32)$, which we discussed in Part I. 
\item Suppose $L^\perp\cap\xi$ is a maximal singular subspace $U$ of $\xi$ and $M^\perp\cap\xi$ is a line $K$.
If $K\cap U$ is a point $x$, then $\<x,L\>$ and $\<x,M\>$ are planes intersecting in $x$, and it follows from \cref{pointspecialline} that each point of $L$ is symplectic to either a unique point of $M$, or all points of $M$.
This now clearly leads to $(\frac32\frac32\frac322)$ again (and replacing $p$ with the unique point on $L$ symplectic to all points of $M$, we are back to (2)).
Hence we may assume that $K\cap U=\varnothing$.
If some point $p'$ of $L\setminus\{p\}$ were symplectic to some point of $M\setminus\{q\}$, then we could replace $p$ by $p'$ and are back to situation (2).
If no point of $L\setminus\{p\}$ is symplectic to any point of $M\setminus\{q\}$, then we have distance $(\frac322\frac322)$.

\item Similarly, $L^\perp\cap\xi$ a line and $M^\perp\cap\xi$ a maximal singular subspace lead to $(\frac32\frac32\frac322)$ or $(\frac32\frac3222)$. 

\item Finally, suppose $L^\perp\cap\xi$ is a line $K$ and $M^\perp\cap\xi$ is a line $N$.
If $K$ and $N$ intersect, then we are back to a previous case already handled, namely $(\frac3222\frac32)$. 
If $K$ and $N$ are $\xi$-special, then we claim that we have distance $(\frac32222)$.
Indeed, the alternative would be that some point $q'\in M\setminus\{q\}$ is symplectic to some point $p'\in L\setminus\{p\}$.
This would imply that $p'$ is locally close to the symp determined by $q'$ and $N^\perp\cap K$.
But this implies that $q$ is symplectic to $p'$, which contradicts the fact that $q$ is not collinear to all points of $K$.
If $K$ and $N$ are $\xi$-opposite, then, using similar arguments, we have distance $(\frac32223)$. 
\end{compactenum}
This takes care of the case $p\pperp q$.
 
\textbf{Case 4: $p\Join q$.} Since we may assume we are not in the ``Completely homogeneous'' case, there are opposite pairs of points on $L\cup M$.
Since we may also assume that we are not in the ``Projection homogeneous'' case, we may assume that no point of $L\setminus\{p\}$ is special to any point of $M\setminus\{q\}$.
But then every point of $L$ is special to $q$ and every point of $M$ is special to $p$, whereas each point of $L\setminus\{p\}$ is opposite each point of $M\setminus\{q\}$.
This is distance $(2223)$.

One checks that the cases involving a singular subspace of dimension at least $3$ (this includes each case where some point is collinear to a maximal singular subspace of a symp) are precisely the positions $(1111)$, $(\frac32 \frac32 \frac32 \frac32)$, $(2222)$, $(1 \frac32 \frac32 \frac32)$, $(\frac32 \frac32 \frac32 2)$, $(1 \frac32 1 \frac32)$, $(1 1 \frac32 \frac32)$, $(\frac32 \frac32 2 2)$ and $(\frac32 2 \frac32 2)$.

The lemma is completely proved. 
\end{proof}

For any ordered pair of lines $(L,M)$, we call each point of $L$ distinct from the projection point, if it exists, a \emph{free point (for $(L,M)$)}; hence if there is no projection point, every point is free. 

We define the following order $0<1<\frac32<2<3$.
We have the following lemma, where we denote the set of mutual positions of two lines, enumerated above in \textbf{Class I} up to \textbf{Class IV},  as $\mathcal{D}$. 
\begin{lem}\label{combingprep2}
For every $D \in\mathcal{D}$, there exists a unique $D'\in\mathcal{D}$ such that
\begin{compactenum} \item[$(*)$] for every pair of lines $(L, M )$ with $\delta(L, M ) = D$ and every free point $x \in L$, there exists a line $K$ through $x$ not locally opposite $L$ at $x$ and such that
$\delta(L',M) = D'$ for every line $L' \ni x$ locally opposite $K$ at $x$.
\end{compactenum}
The function that maps every $D \in \mathcal{D}$ onto the unique distance $D' \in \mathcal{D}$ that satisfies property $(*)$ is as described by the first column of arrows in  \emph{\cref{table}}.
Applying this assertion to $(L',M)$ again and again, we eventually arrive at an opposite pair.
In \emph{\cref{table}}, we list the consecutive distances when we apply this algorithm.
The penultimate column lists the local mutual position of $L$ and $K$ (with respect to the first arrow in the row), and the last column mentions when $K$ is not unique.
\end{lem}

\begin{table}[ht]
\begin{tabular}{cccccccccc}\toprule
$\{1\}$ & $(0110)$ & $\to$ & $(0112)$ & $\to$ & $(1223)$ & $\to$ & $(2332)$ & equal & \\
$\{2\}$ & $(0111)$ & $\to$ & $(11 \frac32 2)$ &  $\to$ & $(\frac32 223)$ &  $\to$ & $(2332)$ &eq or coll& not unique\\
$\{3\}$ & $(011\frac32)$ & $\to$ & $(1 \frac32 22)$ &  $\to$ & $(2223)$ &  $\to$ & $(2332)$ &equal& \\
$\{4\}$ & $(111\frac32)$ & $\to$ & $(1 \frac32 22)$ &  $\to$ & $(2223)$ &  $\to$ & $(2332)$ &collinear&\\
$\{5\}$ & $(1 \frac32 \frac32 1)$ & $\to$ & $(1 \frac32 22)$ &  $\to$ & $(2223)$ &  $\to$ & $(2332)$ &symplectic&\\ 
$\{6\}$ &  $(0112)$ & $\to$ & $(1223)$ & $\to$ & $(2332)$&&&equal& \\
$\{7\}$ & $(1 \frac32 12)$ & $\to$ & $(1223)$ & $\to$ & $(2332)$ &&&collinear&\\
$\{8\}$ & $(11 \frac32 2)$ & $\to$ & $(\frac32 223)$ &  $\to$ & $(2332)$ &&&equal& \\
$\{9\}$ & $(1 \frac32 \frac32 2)$ & $\to$ & $(\frac32 223)$ &  $\to$ & $(2332)$ &&&collinear&\\
$\{10\}$ & $(12 \frac32 2)$ & $\to$ & $(\frac32 223)$ &  $\to$ & $(2332)$ &&&symplectic&\\
$\{11\}$ & $(\frac32 22 \frac32)$ & $\to$ & $(\frac32 223)$ &  $\to$ & $(2332)$ &&&collinear&\\
$\{12\}$ & $(1223)$ & $\to$ & $(2332)$ &&&&&equal&\\
$\{13\}$ & $(1 \frac32 22)$ &  $\to$ & $(2223)$ &  $\to$ & $(2332)$  &&&eq or coll&not unique\\
$\{14\}$ & $(\frac32 222)$ & $\to$ &  $(2223)$ &  $\to$ & $(2332)$ &&&coll or sympl&not unique\\
$\{15\}$ & $(\frac32 223)$ &  $\to$ & $(2332)$ &&&&&collinear&\\
$\{16\}$ & $(2223)$ &  $\to$ & $(2332)$ &&&&&symplectic&\\
$\{17\}$ & $(2332)$ & $\to$ & $(2332)$ &&&&&special& \\
$\{18\}$ & $(1111)$ & $\to$ & $(11 \frac32 2)$ & $\to$ &  $(\frac32 223)$ &  $\to$ & $(2332)$ & collinear & not unique \\
$\{19\}$ & $(\frac32 \frac32 \frac32 \frac32)$ & $\to$ & $(\frac32 \frac32 22)$ & $\to$ & $(2223)$ & $\to$ &  $(2332)$ & collinear & not unique\\
$\{20\}$ & $(2222)$ & $\to$ & $(2223)$ & $\to$ &  $(2332)$ &&& symplectic & not unique\\
$\{21\}$ & $(1 \frac32 1 \frac32)$ & $\to$ & $(1 \frac32 22)$ &  $\to$ & $(2223)$ &  $\to$ & $(2332)$ & collinear & \\
$\{22\}$ & $(11 \frac32 \frac32)$ & $\to$ & $(\frac32 \frac32 22)$ & $\to$ & $(2223)$ & $\to$ &  $(2332)$ & equal &\\
$\{23\}$ & $(1 \frac32 \frac32 \frac32)$ & $\to$ & $(\frac32 \frac32 22)$ & $\to$ & $(2223)$ & $\to$ &  $(2332)$ & collinear &not unique\\
$\{24\}$ & $(\frac32 \frac32 22)$ & $\to$ & $(2223)$ & $\to$ &  $(2332)$ & &&collinear &not unique\\
$\{25\}$ & $(\frac32 2 \frac32 2)$ & $\to$ & $(\frac32 223)$ &  $\to$ & $(2332)$ &&&symplectic&\\
$\{26\}$ & $(\frac32 \frac32 \frac32 2)$ & $\to$ & $(\frac32 223)$ &  $\to$ & $(2332)$ &&&collinear&\\ \bottomrule \addlinespace
\end{tabular}
\caption{Combing distances between lines\label{table}}
\end{table} 

\begin{proof} We have to treat the 26 cases one by one.
However, some cases are immediate or at least easy, and we skip those.
It concerns many of the cases where $K=L$, namely $\{1\}$, $\{3\}$, $\{6\}$, $\{12\}$.
Other cases are easy once one knows $K$, and we only give that information below.
Basically, $K$ can always be thought of as a kind of projection of $M$ onto $x$.
In cases $\{12\}$, $\{15\}$, $\{16\}$ and $\{17\}$, the point $x$ is opposite some point of $M$, and then $K$ is really that projection, and so these cases are straightforward and we skip them, too.
More tricky cases are treated in full detail.
It concerns in particular the cases that cannot occur in type $\mathsf{F_4}$.
  
\begin{compactenum}
\item[$\{2\}$] $(=,\perp,\perp,\perp)\to(\perp,\perp,\pperp,\Join)$.\\
 The line $K$ is any line through $x$ in the plane spanned by $L$ and $M$. 
 
 \item[$\{4\}$] $(\perp,\perp,\perp,\pperp)\to(\perp,\pperp,\Join,\Join)$.\\
 The line $K$ is the line joining $x$ with the unique point of $M$ collinear to each point of $L$. 
 
 \item[$\{5\}$] $(\perp,\pperp,\pperp,\perp)\to(\perp,\pperp,\Join,\Join)$.\\ 
 Here, $K$ is the line joining $x$ with the unique point $y\in M$ collinear to $x$.
Then let $N$ be a line through $x$ locally opposite $K$ and pick $z\in N\setminus\{x\}$.
Let $\xi$ be the symp containing $L$ and $M$ and set $z^\perp\cap\xi=M'$. 
Then $M'$ and $M$ are $\xi$-opposite, since if they were not, either $x$ would be collinear to $M$ (which contradicts the fact that $L$ and $M$ are $\xi$-opposite), or $y$ would be collinear to all points of $M'$ (which contradicts the fact that $z\Join y$ by the choice of $N$ locally opposite $K$). 
 
 \item[$\{7\}$]  $(\perp, \pperp, \perp,\Join)\to(\perp,\Join,\Join,\equiv)$.\\
 Let $y\in M$ be collinear to each point of $L$.
Then $K=xy$, and the rest follows from \cref{joinjoin}.
 
 \item[$\{8\}$] $(\perp,\perp,\pperp,\Join)\to(\pperp,\Join,\Join,\equiv)$.\\
Here $K=L$.
Set $p\in L$ the unique point collinear to each point of $M$ and $q\in M$ the unique point of $M$ symplectic to $x$.
By \cref{joinjoin}, every other point of $M$ is opposite each point of $L'\setminus\{x\}$.
It follows that the distance between $L'$ to $M$ is $(\frac32223)$.
 
 \item[$\{9\}$] $(\perp,\pperp,\pperp,\Join)\to(\pperp,\Join,\Join,\equiv)$.\\
 Let $p\in L$ and $q\in M$ be collinear.
Let $\xi$ be the symp through $M$ and $p$ and let $z$ be the unique point of $\xi$ collinear to $M$ and collinear to $x$.
Set $K=xz$.
Then, precisely like in the previous case $\{8\}$, we conclude that the distance between $L'$ to $M$ is $(\frac32223)$.
 
 \item[$\{10\}$] $(\perp,\Join,\pperp,\Join)\to(\pperp,\Join,\Join,\equiv)$.\\
 Let $\xi$ be the symp containing $L$ and a unique point $q\in M$.
Then there is a unique point $z\in\xi$ collinear to $M$ and $x$.
Putting $K=xz$, the rest of the argument is the same as for the previous cases $\{8\}$ and $\{9\}$. 
 
 \item[$\{11\}$] $(\pperp, \Join, \Join, \pperp) \rightarrow (\pperp, \Join, \Join,\equiv)$.\\
\cref{perp22perp} yields a point $z$ collinear to each point of $L\cup M$.
Then $K=xz$, and the same arguments as in the three previous cases imply that the distance between $L'$ to $M$ is $(\frac32223)$.

 \item[$\{13\}$] $(\perp, \pperp, \Join, \Join) \rightarrow (\Join, \Join, \Join,\equiv)$.\\
Let $p\in L$ be the point of $L$ contained in a common symp $\xi$ with $M$.
Then $x\neq p$ is collinear to a line $N$ of $\xi$.
This line $N$ is $\xi$-opposite $M$.
Then $K$ is any line through $x$ in the plane spanned by $x$ and $N$.
Set $K\cap N=\{p'\}$ and $p'^\perp\cap M=\{q'\}$.
Then the position $(2223)$ between $L'$ and $M$ follows from the facts that $L'$ is locally opposite $K$ at $x$; $K$ is locally opposite $p'q'$ at $p'$, and $p'q'$ locally symplectic to $M$ at $q'$.
 
 \item[$\{14\}$] $(\pperp, \Join,\Join, \Join) \rightarrow (\Join, \Join, \Join,\equiv)$.\\
Let $\xi$ be the symp through the unique points $p\in L$ and $q\in M$.
Set $N:=x^\perp\cap\xi$, and let $z$ be the point on $N$ collinear to $q$.
\cref{pointspecialline} yields a line $N'\ni z$ consisting of all points collinear to a point of $M$ and $x$.
Then $K$ is any line through $x$ in the plane $\<x,N'\>$.
Similar arguments as in the previous case $\{13\}$ show that the mutual position of $L'$ and $M$ is $(2223)$. 

 \item[$\{18\}$] $(\perp, \perp, \perp, \perp) \rightarrow (\perp, \perp, \pperp,\Join)$.\\
Let $K$ be any line through $x$ intersecting $M$.
Then, by \cref{pointspecialline}, there is a unique point on $M$ symplectic to the points of $L'\setminus\{x\}$.
This implies that the mutual position between $L'$ and $M$ is $(11\frac322)$. 
 
\item[$\{19\}$]  $(\pperp, \pperp, \pperp, \pperp) \rightarrow (\pperp, \pperp, \Join, \Join)$.\\
Pick two points $y_1,y_2$ on $M$.
By \cref{generalspecial}, 
the point $y_2$ is collinear to a maximal singular subspace $U$ of $\xi(x,y_1)$.
Then $x$ is collinear to a hyperplane $W$ of $U$, and $K$ is any line joining $x$ with a point $z$ of $W$.
Let $u\in L'\setminus\{x\}$ be arbitrary.
Then $u$ and $z$ are special.

Suppose $u$ were symplectic to some point of $M$, and we may without loss of generality assume that point is $y_1$.
Then $u$ would have to be collinear to a line $N$ of $\xi(x, y_1)$ including $x$ and $y_1$.
But $x$ is not collinear to any point of $M$.
It follows that $u$ is special to every point of $M$.
With that, $L'$ contains exactly one point, which is $x$, symplectic to all points of $L_2$; and otherwise, all remaining points of $L'$ are special to all points of $L_2$, since $u \in M$ was arbitrary.

\item[$\{20\}$] $(\Join, \Join, \Join, \Join) \rightarrow (\Join, \Join, \Join, \equiv)$.\\
\cref{pointspecialline} yields a line $N$ consisting of all points collinear to $x$ and some point of $M$.
Then $K$ is any line through $x$ in the plane $\<x,N\>$.
Note that $K$ is locally symplectic to $L$ at $x$.
Let $u\in L'\setminus\{x\}$.
Then there is a unique point on $N$ symplectic to $u$; all other points of $N$ are special to $u$.
\cref{joinjoin} implies that $u$ is opposite all but exactly one point of $M$.
Since $x$ is special to all points of $M$, the mutual distance between $L'$ and $M$ is $(2223)$. 

\item[$\{21\}$] $(\perp, \pperp, \perp, \pperp) \rightarrow (\perp, \pperp, \Join, \Join)$.\\
Let $q$ be the unique point of $M$ collinear to all points of $L$.
Then $K$ is the line $qx$.
 
Let $q'$ be any point of $M\setminus\{q\}$.
Then $M$ is contained in the symp $\xi(q', x)$.
Let $u$ be an arbitrary point of $L'$ not equal to $x$.
Then $u$ and $q$ are special, and consequently $u$ can only be collinear to a line of the symp $\xi(q', x)$ through $x$, 
which implies that $u$ and $q'$ are special.
In summary, $q$ is collinear to $x$ and special to every other point of $L'$, and every point in $M \setminus \{ q \}$ is symplectic to $x$ and special to every point in $L' \setminus \{ x_1 \}$.

\item[$\{22\}$] $(\perp, \perp, \pperp, \pperp) \rightarrow (\pperp, \pperp, \Join, \Join)$.\\
Let $p$ be the point of $L$ collinear to every point of $M$.
Then $x\neq p$.
Here, $K=L$. 
Let $u$ be some arbitrary point of $L'$ not equal to $x$.
Then $u$ and $p$ are special.
The point $x$ is symplectic to every point of $M$.
Considering the symp $\xi$ containing $x$ and an arbitrary point $y \in M$, we see that $u^{\perp} \cap \xi$ is a line $N$ through $x$ (indeed, $p\in\xi$ and $u\Join p$).
Now, $y$ is only collinear to a unique point of $L$, since $y$ is not collinear to $x$.
It follows, with \cref{generalspecial}, that $u$ is special to $y$.
We obtain $(\frac32\frac3222)$. 

\item[$\{23\}$] $(\perp, \pperp, \pperp, \pperp) \rightarrow (\pperp, \pperp, \Join, \Join)$.\\
Let $\xi$ be the symp through $M$ and the unique point $p$ of $L$ collinear to some point $q$ of $M$.
It is easy to see that $x^\perp\cap\xi$ is a maximal singular subspace $U$ of $\xi$.
Then $K$ is any line through $x$ and a point of $M^\perp\cap U$.
Considering the respective symps through $x$ and the points of $M$, we see that the points of $M$ and $L'\setminus\{x\}$ are special.
Hence we obtain $\left(\frac32\frac3222\right)$.

\item[$\{24\}$] $(\pperp, \pperp, \Join, \Join) \rightarrow (\Join, \Join, \Join, \equiv)$.\\
Let $p\in L$ be symplectic to all points of $M$.
\cref{pointsymplecticline} yields a singular subspace $U=p^\perp\cap M^\perp$ contained in each symplecton containing $p$ and a point of $M$, and such that $\<p,U\>$ is a maximal singular subspace of each such symp.
Combining this with \cref{pointspecialline}, we find a line $N\subseteq U$ consisting of all points collinear to $x$ and some point of $M$.
Then $K$ is an arbitrary line through $x$ in the plane $\<x,N\>$.
Taking into account that $N$ and $M$ are contained in a common symp $\xi$ in which they are $\xi$-opposite, we arrive at $(2223)$ for the distance between $L'$ to $M$. 

\item[$\{25\}$] $(\pperp, \Join, \pperp, \Join) \rightarrow (\pperp, \Join, \Join, \equiv)$.\\
Here, there is a unique point $q\in M$ symplectic to each point of $L$, in particular to $x$.
The line $M$ is collinear to a unique line $N$ of $\xi(q,x)$, and $x$ is collinear to a unique point $z$ of $N$.
Then $K$ is the line $xz$, and it is immediate that $L'$ and $M$ are at mutual distance $(\frac32223)$.

\item[$\{26\}$]  $(\pperp, \pperp, \pperp, \Join) \rightarrow (\pperp, \Join, \Join, \equiv)$.\\
Let $p \in L$ be the point symplectic to all points of $M$, and let $q \in M$ be the point symplectic to all points of $L$.
Then $M$ is collinear to a unique line $N$ of $\xi(q,x)$, and $x$ is collinear to a unique point $z$ of $N$.
We define $K=xz$. 
Then \cref{joinjoin} implies that each point of $L'\setminus\{x\}$ is opposite each point of $M\setminus\{q\}$.
Since $x\pperp q$ and $x$ is special to all points of $M\setminus\{q\}$, we conclude that the mutual distance of $L'$ and $M$ is given by $(\frac32223)$.
\end{compactenum}
This completes the proof of the lemma. 
\end{proof}

The length of the sequence in the previous lemma is called the \emph{level} of the corresponding line $M$ (with respect to $L$), except that when $M$ is opposite $L$, we say it has level $0$.  

\subsubsection{Algorithms and end of the proof}

Let $\Delta=(X,\cL)$ be a finite exceptional hexagonic Lie incidence geometry whose lines carry exactly $s+1$ points.  We introduce two algorithms, that we will call \emph{combing algorithms}.
They require that certain conditions are met, and we will also introduce these.
Naturally, we will only run them when all conditions are satisfied.
They are defined as follows. 

\begin{defn}[The combing algorithms] 
Let $L_0,L_1,\ldots,L_s\in\cL$ be $s+1$ lines of $\Delta$ and let $L$ be another arbitrary line.
Suppose 
\begin{compactenum} \item[(COMB1)] There exists a point $x\in L$ which is free for every pair $(L,L_i)$, $i=0,1,\ldots,s$.
\end{compactenum}
Condition (COMB1) just means that the set of projection points on $L$ with respect to the lines $L_i$, $i=0,1,\ldots,s$, does not cover $L$.

For each $L_i$, $i\in\{0,1,\ldots,s\}$, and each point $x$ on $L$ that is free with respect to each line $L_j$, $j\in\{0,1,\ldots,s\}$, we define a line $M_i$ as follows.
If $L_i$ is opposite $L$, then $M_i$ is the unique line through $x$ containing a point collinear to $L_i$.
The pair $\{L,M_i\}$ has distance $(0112)$, or is, in other words, locally opposite at $x$.
If $L_i$ is not opposite $L$, then we set $M_i$ equal to the line $K$, as (perhaps not uniquely) defined in \cref{combingprep2}.
If the line $K$ is not uniquely defined there, then we arbitrarily choose one (and one may think of taking the closest to $L$, if this exists). 
We set $\cM=\{M_i\mid i\in\{0,1,\ldots,s\}\}$. Suppose now

\begin{compactenum}[(COMB2)] \item There exists a line $L'$ through $x$ locally opposite each member of $\cM$
at $x$. 
\end{compactenum}

Then the line $L'$ is, by definition, the outcome of the first combing algorithm. Suppose on the other hand

\begin{compactenum}[(COMB3)] \item There exists a line $M\in\cM$ locally opposite $L$ at $x$ and there exists a line $L''$ through $x$ locally opposite each member of $\cM\setminus\{M\}$ (where we view $\cM$ as a set and not as a multiset) and not locally opposite $M$. 
\end{compactenum}

Then line $L''$ is, by definition, the outcome of the second combing algorithm (\emph{combing back at $M$}).

\end{defn}
We observe:

\begin{lem}\label{combingprep4}
Using the above notation, and under
hypothesis {\rm(COMB2)}, if $0$ is the level of $(L, L_i)$, then $0$ is also the level of $(L',L_i)$;
in contrast, under hypohesis {\rm(COMB3)}, if $0$ is the level of $(L,L_i)$, then the level of
$(L'', L_i)$ is at most $1$ (and equal to $0$ whenever $M_i\neq  M$)
\end{lem}

We will always use \cref{class2} to be able to perform the first combing algorithm, whereas \cref{ADEdistances} will allow us to perform the second combing algorithm.
This is roughly the content of the proof of the next result. 

\begin{prop}\label{proplinesF4}
Every set $T=\{L_0,L_1,\ldots,L_s\}$ of $s+1$ lines in a metasymplectic space, not isomorphic to $\mathsf{F_{4,4}}(\sqrt{s},s)$, or in an exceptional long root subgroup geometry of type $\mathsf{E_6}$, $\mathsf{E_7}$ or $\mathsf{E_8}$, where every line has exactly $s+1$ points, such that every other line is not opposite at least one member of $T$, is a geometric line in the line-Grassmannian geometry, that is, has the property that every other line is either not opposite a unique member of $T$, or opposite no member of $T$. 
\end{prop}

\begin{proof}
Obviously, a geometric line has the stated property.
So assume now that $T$ is not a geometric line, but every other line is not opposite at least one member of $T$.
The only way in which we can violate the defining property of a geometric line is to assume the existence of a line $L$ not opposite at least $2$ members of $T$ and opposite at least one member of $T$.
We prove that this leads to a contradiction.
The rough idea is to apply the combing algorithms to $L$ and $T$ until we find a line opposite all members of $T$.
Since our proof will be inductive in some sense, it is important that after each application of the combing algorithm, the new line $L$ satisfies the same assumption, that is, the new line $L$ is not opposite at least two members of $T$ and opposite at least one member of $T$, or the proof ends and $L$ is opposite each member of $T$.
(We say in these cases that the new line $L$ is \emph{legal}.) This little condition implies that we cannot blindly run the combing algorithms, but we have to choose the right one.
The way we do this goes as follows. 

We start by noting that a line $M_i\in\cM$ is locally opposite $L$ if and only if $L_i\equiv L$.
Hence, since at least two lines of $T$ are not opposite $L$, and at least one line of $T$ is opposite $L$, \cref{class2} implies that (COMB2) is satisfied.
Also, since opposite lines do not define projection points, and there is at least one line in $T$ opposite $L$, (COMB1) is satisfied.
Moreover, \cref{ADEdistances} allows us to run the second combing algorithm since $T$ contains at least one line opposite $L$. 

Now we combine the two combing algorithms in one overarching algorithm that proves the theorem.
That algorithm goes as follows. 

\begin{compactenum}[$\bullet$]  
\item[$\bullet$] If $T$ contains at least two members at level at least $2$, then we apply the first combing algorithm.
Note that elements of $T$ opposite $L$ remain opposite $L'$, and elements of $T$ at level $k\geq 1$ with respect to $L$ are at level $k-1$ with respect to $L'$.
Hence $L'$ is legal, and the maximum level decreases. 

\item If $T$ contains exactly one member $L_0$ at level at least $2$ (and hence at least one member $L_1$ of level $1$), then we apply the second combing algorithm combing back at an arbitrary $M_2\in\cM$ locally opposite $L$.
Then $L_0$ comes at level at least $1$, $L_1$ at level $0$, and $L_2$ at level $1$.
Hence $L''$ is legal.
Moreover, \cref{combingprep4} guarantees that the maximum level again decreases. 

\item We apply the previous two steps as long as the maximal level is at least $2$.
If the maximal level is or becomes $1$, then we apply the first combing algorithm and obtain a line $L'$ opposite each member of $T$, a contradiction that proves the assertion.
\qedhere
\end{compactenum}
\end{proof}

\subsubsection{The exceptional case $\mathsf{F_{4,4}}(q,q^2)$}\label{F44}

The above does not work for lines of metasymplectic spaces $\Delta$ isomorphic to $\mathsf{F_{4,4}}(q,q^2)$, because we cannot apply \cref{class2} since in the polar space $\mathsf{B_{3,1}}(q,q^2)$, corresponding to the residue of a point $p$, there are sets of $q^2+1$ planes admitting no common opposite plane, and yet not isomorphic to a geometric line (pencil of planes).
The examples arise when $q$ is a power of $2$ and are mentioned in \cref{typeB}(3): they are sets of planes through a common point $b$ forming a spread in a subquadrangle of order $(q,q)$ of the residue at $b$.
We will call such an example an \emph{OBS (ovoidal blocking set)}.
So we have to provide a different proof. 

Note that, viewed in $\Delta$, the point $b$ is a symp, and the elements of an OBS are lines through a common point forming an ovoid in a subquadrangle of order $(q,q)$ of the residue at that common point.
Also, such a set will be called an OBS. 

Also, note that, viewed in $\mathsf{F_{4,1}}(q,q^2)$, the point $p$ is a symp, and the elements of an OBS are planes through a common point $b$ of the symp $p$ forming a spread in a subquadrangle of order $(q,q)$ of the residue at $b$.

We first observe that it is really an example of a blocking set.

\begin{lem}\label{exceptionalexample}
Let $T$ be a set of $q^2+1$ lines of $\Delta=\mathsf{F_{4,4}}(q,q^2)$ incident with a common point $b$ and forming an ovoid in a subquadrangle of the residue $\Res_\xi(b)$ at $b$ of some symp $\xi$ through $b$.
Then no line of $\Delta$ is opposite each member of $T$.
\end{lem}

\begin{proof}
This follows directly from \cref{corolTits}.
\end{proof}

We now show a converse to \cref{exceptionalexample}, that is, any set $T$ of $q^2+1$ lines of $\Delta=\mathsf{F_{4,4}}(q,q^2)$ with the property that no line of $\Delta$ is opposite each member of $T$ is either a planar line pencil, or an OBS. 

Let $T=\{L_0,\ldots,L_t\}$, $t=q^2$, be a set of lines of $\Delta\cong \mathsf{F_{4,4}}(q,q^2)$ admitting no common opposite line. 

Note that each point of a singular subspace $S$ is 
opposite some point of a given symp $\xi$ if, and only if, $S$ and $\xi$ are far.
Indeed, 
if there is a symp through $S$ opposite $\xi$, then clearly, each point of $S$ is opposite some point of $\xi$.
Now suppose each point of $S$ is opposite some point of $\xi$.
Pick $x\in S$ and let $\zeta$ be the unique symp through $x$ intersecting $\xi$.
Our assumption implies that $S$ and $\zeta$ intersect just in $x$.
Hence we can find a symp $\zeta'$ through $S$ locally opposite $\zeta$ in $x$.
Then $\zeta'$ is opposite $\xi$ by \cref{Tits}. 

\begin{lem}\label{projectionsbxi}
There exist a point $b$ and a symp $\xi$ in $\Delta$, with $b\in\xi$, such that both $b$ and $\xi$ are far from each member of $T$.
For each such $b$ and $\xi$ we have that the projections of the members of $T$ onto $b$ and $\xi$, respectively, form either both a planar line pencil, or both an OBS.
\end{lem}

\begin{proof}
We can choose points $b_i$ contained in $L_i$ such that the $b_i$ do not form a geometric line in $\Delta$.
Then \cref{longrootlemma}, \cref{kasmal} and \cref{geomlinesexc} yield a point $b$ opposite all $b_i$, $i\in\{0,1,\ldots,s\}$.
So, $b$ is far from each member of $T$.
Now set $T'=\{L_i'=\proj^{b_i}_b(L_i)\mid i\in\{0,1,\ldots,s\}\}$.
If $T'$ is not an OBS and not a planar line pencil, then we can find a line $L$ through $b$ locally opposite each member of $T'$, and so, by \cref{Tits}, $L$ is opposite each member of $T$, a contradiction.
We conclude that $T'$ is contained in a symp $\zeta$ through $b$.
Now let $\xi$ be a symp locally opposite $\zeta$ at $b$.
Then, again by \cref{Tits}, the projection $\xi_i$ of $\zeta$ onto $b_i$ is opposite $\xi$.
However, $\xi_i$ contains $L_i$ as $\zeta$ contains $L_i'$, $i=0,1,\ldots,s$.
Hence $\xi$ is far from each member of $T$. 

Now let $L_i''$ be the projection of $L_i$ onto $\xi$ and let $L_i'''$ be the unique line of $\xi$ collinear to $L_i'$.
We claim that $L_i''$ 
intersects $L_i'''$, which then shows that the projection of $T'$ onto $\xi$ coincides with the projection onto $b$ of the projection of $T$ onto $\xi$, and hence $T'$ is isomorphic to the projection of $T$ onto $\xi$ and the lemma follows.

Let $M'''_i$ be the unique line through $b$ intersecting $L_i''$, say in the point $x_i''$.
Since $x_i''\in L_i''$, there is a unique point $x_i\in L_i$ symplectic to $x_i''$.
Then $b$ is collinear to a unique line $K_i$ of $\xi(x_i,x_i'')$ through $x_i''$, and $x_i$ is collinear to a unique point $x_i'$ of $K_i$.
Now $x_i\perp x_i'\perp b$ defines a path of length $2$ from $x_i\in L_i$ to $b$, hence $bx_i'=L_i'$ and $M_i''=L_i'''$ and the claim follows.
\end{proof}

\begin{lem}\label{pointorsymp}
Each pair of members of $T$ is either contained in a symp, or has a point in common. 
\end{lem}

\begin{proof}
It is convenient to consider the dual situation, that is, $T$ corresponds to a set $T^*$ of planes $\{\alpha_0,\ldots,\alpha_s\}$, $s=q^2$, of $\mathsf{F_{4,1}}(q,q^2)$.
By \cref{projectionsbxi} we can find a symp $\xi$ far from each member of $T^*$.
Hence we can project all planes $\alpha_i$ onto $\xi$ and obtain planes $\alpha_i'$. 
By \cref{corolTits} and \cref{typeB}, the $\alpha_i'$ form a full plane pencil or an OBS.
In particular, all planes $\alpha_i'$ contain a common point $t$, and for each line $L'_0$ of $\alpha_0'$ through $t$, except for the possible intersection line with $\alpha_1'$, there exist $q^2$ lines $L'_1$ of $\alpha_1'$ through $t$ not coplanar with $L'_0$.
Let $z_0$ and $z_1$ be two arbitrary points on $L'_0$ and $L'_1$, respectively.
Select a point $p$ in $\xi$ not collinear to $t$, but collinear to both $z_0$ and $z_1$.
Let $p_i$ and $x_i$ be the unique points in $\alpha_i$ symplectic to $z_i$ and $t$, respectively, $i=0,1$.
Since $p$ is collinear to a unique line of $\xi(p_i,z_i)$, there is a unique point $y_i$ in $\xi(p_i,z_i)$ collinear to $p_i, z_i$ and $p$, $i=0,1$.
By the ``dual'' of \cref{projectionsbxi}, the points $y_0$ and $y_1$ are either collinear or symplectic.
We claim that the symps $\xi(p_0,z_0)$ and $\xi(p_1,z_1)$ are opposite. Indeed, we observe that (1) the point $z_1$ is the unique point of $\xi(p_0,z_0)$ symplectic to $z_0$, and (2) $z_0$ is the unique point of $\xi(p_1,z_1)$ symplectic to $z_1$. It already follows that the symps cannot be adjacent. If they meet in a unique point, then, by Observation~(1), $z_0$ is symplectic to the intersection point, which must then be $z_1$, contradicting Observation~(2). If the symps are special, say $\xi(p_0,z_0)$ is adjacent to $\zeta$, and the latter to $\xi(p_1,z_1)$, then a similar argument shows that $z_1$ is contained in the plane $\zeta\cap\xi(p_1,z_1)$, contradicting the second observation again. The claim follows. 

It now easily follows that $y_0$ and $y_1$ are symplectic.
Let $t_i$ be the unique point of $\xi(p_i,z_i)$ collinear to $t, p_i$ and $z_i$, $i=0,1$.
Then, varying $p$ over all points of $\xi$ not collinear to $t$, but collinear to both $z_0$ and $z_1$, we deduce that  $p_0^\perp\cap z_0^\perp\setminus t_0^\perp$ corresponds to $p_1^\perp\cap z_1^\perp\setminus t_1^\perp$ under the projection map from $\xi(p_0,z_0)$ to $\xi(p_1,z_1)$ given on the points by ``being symplectic''.
It easily follows that $p_0^\perp\cap z_0^\perp$ corresponds to $p_1^\perp\cap z_1^\perp$.
Hence $(p_0^\perp\cap z_0^\perp)^\perp$ corresponds to $(p_1^\perp\cap z_1^\perp)^\perp$.
Since symps are isomorphic to quadrics $Q^-(7,q)$, which are embedded in non-degenerate (symplectic) polarities, we have $(p_i^\perp\cap z_i^\perp)^\perp=\{p_i,z_i\}$, 
$i=0,1$.
Since $z_0$ corresponds to $z_1$, we conclude that $p_0$ and $p_1$ are symplectic. 

We have shown that $p_0$ is symplectic to all points of $\alpha_1$, except possibly the points of a unique line.
It then easily follows that $p_0$ is collinear or symplectic to any given point of $\alpha_1$.
By the arbitrariness of $p_0$ in $\alpha_0\setminus\{x_0\}$, 
we deduce that any pair of points in $\alpha_0\cup\alpha_1$ is symplectic, collinear or identical.
Consider any symp $\xi_1$ through $\alpha_1$.
If $p_0\in\xi_1$, then $p_0$ is collinear to at least a line of $\alpha_1$.
If $p_0\notin\xi_1$, then it must be close to it and the line $p_0^\perp\cap\xi_1$ must be contained in $\alpha_1$.
Hence in any case, there is a line of $\alpha_1$ collinear to $p_0$, and so we can assume that $\xi_1$ contains $p_0$.
Suppose some point $r_0\in\alpha_0$ does not belong to $\xi_1$.
Then $r_0^\perp\cap\xi_1\subseteq\alpha_1$, as before, showing $p_0\in \alpha_1$.

Hence we have shown that either $\alpha_0$ and $\alpha_1$ are contained in a symp, or they have a point in common.
This means that, if $\alpha_i$ corresponds to $L_i$, then $L_0$ and $L_1$ either intersect in a point, or are contained in a common symp.

The assertion follows by the arbitrariness of $L_0$ and $L_1$ in $T$.
\end{proof}

We can now classify the blocking sets of lines of size $q^2+1$ in $\mathsf{F_{4,4}}(q,q^2)$.

\begin{theorem}\label{F44exception}
Let $T$ be a set of $q^2+1$ lines of $\Delta=\mathsf{F_{4,4}}(q,q^2)$.
Then all members of $T$ are incident with a common point $b$ and form either a planar line pencil, or an ovoid in a subquadrangle of the residue $\Res_\xi(b)$ at $b$ of some symp $\xi$ through $b$ if, and only if, no line of $\Delta$ is opposite each member of $T$.
\end{theorem}

\begin{proof}
The ``only if'' part is \cref{exceptionalexample}.
We now show the ``if'' part.
We first claim that each pair of members of $T$ intersect nontrivially.
Indeed, we may assume for a contradiction that $L_0$ and $L_1$ do not intersect.
Then by \cref{pointorsymp} they are contained in a common symp $\zeta$.
\cref{projectionsbxi} yields a symp $\xi$ far from
each member of $T$.
Also, the same \cref{projectionsbxi} implies that the projection of $L_0,L_1$ from $\zeta$ onto $\xi$ is a pair of intersecting lines.
Since the projection from $\zeta$ to $\xi$ is an isomorphism of polar spaces, this implies that $L_0$ and $L_1$ also intersect.
The claim follows. 

We next claim that all members of $T$ are either contained in a plane, or contain a common point.
Suppose the latter does not hold.
Then there are three lines $L_0,L_1,L_2$ forming a triangle in a plane.
Clearly, all other members of $T$ have to be contained in that plane.
The claim is proved.

Since we now have that $T$ belongs to a point residual, or the residue of a plane, the theorem follows from \cref{corolTits}, \cref{typeAD}(2) and \cref{typeB}. 
\end{proof}
\color{black}

\subsubsection{Geometric lines}\label{geomlineslines}

We now classify geometric lines in the line-Grassmannian of hexagonic Lie incidence geometries.
This will follow from the classification of round-up triples of lines.
 
\begin{lem}\label{RuT0} 
Let $\{L_1,L_2,L_3\}$ be a round-up triple of lines in an exceptional hexagonic Lie incidence geometry $\Delta$ of rank at least $3$, such that $L_1$ and $L_2$ intersect.
Then exactly one of the following holds.
\begin{compactenum}[$(i)$]
\item $L_1=L_2=L_3$;
\item $L_1,L_2,L_3$ are three lines in a common planar line pencil;
\item $L_1,L_2,L_3$ are three lines in a common symp $\xi$ containing a common point $p$ and contained in a common hyperbolic line of $\Res_\xi(p)$.
This only happens if $\Delta$ corresponds to a building of type $\mathsf{F_4}$.
\end{compactenum}
\end{lem}

\begin{proof}
Clearly, if $L_1=L_2$, then also $L_3=L_1$ since otherwise there exists a line opposite $L_3$ and not opposite $L_1$.
So we may assume $L_1\cap L_2=\{x\}$.
By \cref{apartment}, also $x\in L_3$.
By \cref{corolTits}, 
$\{L_1,L_2,L_3\}$ is a round-up triple in $\Res(x)$.
The result now follows from \cref{typeAD} for types $\mathsf{E_6}$ and $\mathsf{E_7}$, from \cref{typeB} for type $\mathsf{F_4}$, and from \cite[Corollary~5.5]{Kas-Mal:13} for type $\mathsf{E_8}$. 
\end{proof}

\begin{lem}\label{RuT1} 
Let $\{L_1,L_2,L_3\}$ be a round-up triple of disjoint lines in an exceptional hexagonic Lie incidence geometry of rank at least $3$.
Then no point of $L_2$ is collinear to any point of $L_1$. 
\end{lem}

\begin{proof}
Let, for a contradiction, $M$ be a line joining a point $x_1\in L_1$ to a point $x_2\in L_2$.
Note that $L_1\neq M\neq L_2$. Since $L_1\cap L_2=\varnothing$,
\cref{apartment2} shows that $M$ intersects $L_3$, say in the point $x_3$.
Assume first that $M$ and $L_i$ are locally opposite at $x_i$, for every $i\in\{1,2,3\}$.
Let $\pi$ be any plane containing $M$.
Let $K_i$ be the line in $\pi$ through $x_i$ not locally opposite $L_i$ at $x_i$, guaranteed to exist by \cref{pointspecialline}.
Suppose first that $z:=K_1\cap K_2$ does not belong to $K_3$.
Let $N$ be a line locally opposite $zx_3$ at $z$.
Then any point $u\in N\setminus\{z\}$ is opposite some point of $L_3$, but is not opposite any point of $L_1\cup L_2$.
It follows that there exists a line through $u$ opposite $L_3$, but not opposite either $L_1$ or $L_2$, a contradiction.
Hence we may assume that there exists some line $K_3'\subseteq\pi$ through $x_3$ intersecting $K_2$ in some point $y_2\notin K_1$, with $y_2$ special to every point of $L_3\setminus\{x_3\}$.
Then we pick a line $N'$ through $y_2$ locally opposite $K_3'$ at $y_2$, but not locally opposite $x_1y_2$ at $y_2$.
Then no point $w$ on $N'$ is opposite some point of $L_1\cup L_2$ since the pair $\{w,x_1\}$ is collinear or symplectic, and the pair $\{y_2,y_2'\}$, with $y_2'\in L_2\setminus\{x_2\}$, is symplectic.
As above, there exists a line through $w$ opposite $L_3$, but not opposite either $L_1$ or $L_2$. 

So, we may assume without loss of generality that $M$ and $L_3$ are contained in a symp $\xi$.
Then, with a similar argument as in the proof of \cref{E7RuT1}, now also using \cref{typeB} (for metasymplectic spaces), we may apply \cref{apartment3} and obtain $x_1\in L_2\subseteq\xi$, contrary to our assumptions.
The lemma is proved.
\end{proof}

\begin{lem}
\label{RuT2}
Let $\{L_1,L_2,L_3\}$ be a round-up triple of lines in an exceptional hexagonic Lie incidence geometry $\Delta$ of rank at least $3$, such that no point of $L_i$ is collinear to any point of $L_j$, for all $i,j\in\{1,2,3\}$, $i\neq j$.
Then no point of $L_1$ is symplectic to any point of $L_2$. 
\end{lem}

\begin{proof}
Suppose for a contradiction that some point $x_1\in L_1$ is symplectic to some point $x_2\in L_2$.
Let $\xi$ be the corresponding symp. 
Noting that the conditions imply that $L_1$ and $L_2$ are not contained in any common symp, \cref{apartment2} implies that $L_3$ shares a point $x_3$ with $\xi$. 

We claim that $L_3$ is collinear to a maximal singular subspace of $\xi$.
Indeed, suppose not.
Then $L_3^\perp\cap\xi$ is a line $L_3^*$.
There are two cases.
\begin{compactenum}[$(1)$] 
\item \emph{Suppose every point of $x_1^\perp\cap x_2^\perp$ is collinear to $x_3$.} Then $x_3\in \{x_1,x_2\}^{\perp\perp}$ and $\Delta$ corresponds to type $\mathsf{F_4}$.
Let $\xi_3$ be an arbitrary symp not containing $L_3$ and locally opposite $\xi$ at $x_3$.
Select $z_3\in\xi_3\setminus x_3^\perp$.
Since $x_1\equiv z_3\equiv x_2$, we can define $M_i:=\proj_{z_3}^{x_i}(L_i)$, $i=1,2$.
If $M_1\neq M_2$, we can take a line $K$ through $z_3$ locally opposite $M_1$ at $z_3$, but not locally opposite $M_2$ at $z_3$, and then $K$ is opposite $L_1$, and not opposite either $L_2$ or $L_3$ (the latter because $z_3$ is symplectic to every point of $L_3$), a contradiction.
Hence, we may assume that $M_1=M_2$.
Set $u_i:=M_i\cap x_i^{\Join}$, $i=1,2$.
If $u_1\neq u_2$, then we may replace $z_3$ with any point in $(M_1^\perp\cap\xi_3)\setminus(\{z_3\}\cup x_3^\perp)$ and apply the previous argument.
So, we may assume $u_1=u_2$.
Let $N_i$ be the line through $u_i$ intersecting $L_i$, say in the point $w_i$, $i=1,2$.
Then, since by \cref{joinjoin}, $x_1$ and $w_2$ are not opposite, the same \cref{joinjoin} implies that $N_1$ and $N_2$ are not locally opposite at $u_1$.
Hence $w_1$ and $w_2$ are symplectic (as we may assume that they are not collinear by \cref{RuT1}). 

Hence, as above, \cref{apartment2} implies that the line $L_3$ intersects $\xi(w_1,w_2)$ in a point $w_3$, which we may assume to belong to $\{w_1,w_2\}^{\perp\perp}$ (as otherwise we are in case (2) below).
Hence $u_1\perp w_3$. If $x_3=w_3$, then $M_1\subseteq\xi_3$, which implies $L_1\subseteq \xi$ (as $\proj^{z_3}_{x_1}(M_1)\subseteq\proj^{z_3}_{x_1}(\xi_3)=\xi$), a contradiction. It follows from  the fact that we are in type $\mathsf{F}_4$, \cref{generalspecial}$(ii)$ and $w_3\notin\xi_3$ that $w_3\Join z_3$. We conclude $u_1=[w_3,z_3]$ which, however, is contained in $\xi_3$ and coincides with $L_3^\perp\cap\xi_3\cap u_1^\perp$. This implies $M_1\subseteq\xi_3$, which is, as above, a contradiction.

\item \emph{Suppose some point $y\in x_1^\perp\cap x_2^\perp$ is not collinear to $x_3$.} Select $y_3\in L_3\setminus\{x_3\}$.
Then, $y_3\Join y$ and $u=[y,y_3]\in L_3^*\setminus\{x_3\}$.
Let $M$ be some line through $y$ locally opposite $yu$ at $y$.
Let $m$ be some arbitrary point on $M$ not equal to $y$.
Let $L_3'$ be the projection of $L_3$ onto $m$.
Note that $L_3'\neq M$ as $L_3\neq uy_3$.
Hence there exists a line $K$ through $m$ locally opposite $L_3'$ but not locally opposite $M$.
Then $K$ is opposite $L_3$ but not opposite $L_1$ and $L_2$, because, by \cref{joinjoin}, no point of $K$ is opposite $x_1$ or $x_2$.
\end{compactenum}
Since both cases lead to contradictions, we conclude that $L_3$ is collinear to a maximal singular subspace $U_3$ of $\xi$.
Likewise, $L_1$ and $L_2$ are also collinear to respective maximal singular subspaces $U_1$ and $U_2$ of $\xi$.
Note that this implies that $\xi$ is top-thin (or hyperbolic).

It follows that, since $x_1$ is not collinear to $x_2\in U_2$, the set $x_3^\perp\cap U_2$ contains some point $z_2$ that is not collinear to $x_1$.
Let $y_3$ be an arbitrary point of $L_3\setminus\{x_3\}$.
If $y_3\pperp z_2$, let $\xi_3$ be the symp through $y_3$ and $z_2$.
If $y_3\perp z_2$, then let $\xi_3$ be a symp containing $L_3$ and $z_2$.
Let $\xi_2$ be a symp through $z_2$ locally opposite $\xi$ but not locally opposite $\xi_3$.
Let $w_2\in\xi_2$ be symplectic to $z_2$.
Then, since $z_2$ is collinear to each point of $L_2$, \cref{joinjoin} implies that $w_2$ is not opposite any point of $L_2$.
Also, since $\xi_2$ is not locally opposite $\xi_3$, the point $w_2$ is not opposite any point of $L_3$.
But $w_2$ is opposite $x_1$ and so there is a line $K$ through $w_2$ opposite $L_1$, and $K$ is not opposite either $L_2$ or $L_3$, a contradiction.

This completes the proof of the lemma. 
\end{proof}

\begin{lem}\label{special}
Let $\{L_1,L_2,L_3\}$ be a round-up triple of lines in an exceptional hexagonic Lie incidence geometry of rank at least $3$.
Let $x_1\in L_1$ and $x_2\in L_2$ be collinear to a common point $y$.
Then $y$ is collinear to a point of $L_3$.
\end{lem}

\begin{proof}
Suppose $y$ is not collinear to any point of $L_3$.
Let $\Sigma$ be an apartment containing $L_3$ and $y$.
Since $y$ is not collinear to any point of $L_3$, it is not special to and not opposite at least two points of the line $L_3^*$ that is opposite $L_3$ in $\Sigma$.
But then $y$ is equal, collinear or symplectic with each point of $L_3^*$, implying that no point of $L_3^*$ is opposite either $x_1$ or $x_2$.
Hence $L_3^*$ is opposite $L_3$, but not opposite either $L_1$ or $L_2$, a contradiction.
\end{proof}

\begin{lem}\label{2'2'2'2'}
Let $\{L_1,L_2,L_3\}$ be a round-up triple of lines in an exceptional hexagonic Lie incidence geometry $\Delta$ of rank at least $3$ such that $L_1\cap L_2=\varnothing$. Then some point of $L_1$ is opposite some point of $L_2$.
\end{lem}

\begin{proof}
Suppose, for a contradiction, that no point of $L_1$ is opposite any point of $L_2$. Then, by \cref{RuT1} and \cref{RuT2}, each point of $L_1$ is special to each point of $L_2$.
Using \cref{pointspecialline}, we see that the set of points collinear to a point of $L_1$ and to a point of $L_2$ is a hyperbolic quadric $Q_{12}$ (of rank~2).
Each line of $Q_{12}$ is collinear to a unique point of $L_1\cup L_2$.
We claim that non-collinear points on $Q_{12}$ are also non-collinear in $\Delta$.
Indeed, one checks that in that case $Q_{12}$ generates a $3$-space $U$.
Picking non-collinear respective points in $Q_{12}$ and $L_1$, we see that they are symplectic and the corresponding symps contain $L_1$ and at least a plane of $U$.
The intersection of two such symps (with different planes in $U$) contains a $3$-space.
Hence the symps coincide.
Now the symp through $L_1$ and $U$ has a $3$-space in common with the symp through $L_2$ and $U$ and hence $L_1$ and $L_2$ are contained in a common symp, a contradiction.

By \cref{special}, each point of $Q_{12}$ is also collinear to a unique point of $L_3$ (unique indeed since, if not, then \cref{pointspecialline} would yield a point of $L_3$ collinear or symplectic to some point of $L_1$, contradicting \cref{RuT1} and \cref{RuT2}).
It follows that, for any line $L$ in $Q_{12}$, collinearity defines either a bijection between $L$ and $L_3$, or a constant transformation from $L$ to $L_3$.
In the latter case, the unique points of $L_3$ and $L_1\cup L_2$ collinear with all points of $L$ are symplectic, contradicting \cref{RuT2}.
In the former case, pick $p\in L$ and let $L'$ be the unique line of $Q_{12}$ through $p$ distinct from $L$.
Then again, collinearity defines a bijection between $L'$ and $L_3$.
Hence there is a point of $L_3$ collinear to two non-collinear points of $Q_{12}$, and since these points are also non-collinear in $\Delta$, this contradicts \cref{pentagon}. 

The lemma is proved.
\end{proof}

\begin{lem}\label{RuT3}
Let $L_1,L_2,L_3$ be three lines of an exceptional hexagonic Lie incidence geometry of rank at least $3$, such that each point of $L_1$ is special to or opposite each point of $L_2$.
Then $\{L_1,L_2,L_3\}$ is not a round-up triple. 
\end{lem}

\begin{proof}
By Lemmas~\ref{RuT0},~\ref{RuT1} and~\ref{RuT2}, we may assume that no point of $L_3$ is equal, collinear or symplectic to any point of $L_1\cup L_2$.
Moreover, by \cref{2'2'2'2'}, we may also assume that some point of $L_3$ is opposite some point of $L_1$ and some point of $L_3$ is opposite some point of $L_2$.
This implies that the mutual positions of $L_i$ and $L_j$, $i,j\in\{1,2,3\}$, $i\neq j$, are given by either $(2223)$ or $(2332)$. 

We may assume that $\{L_1,L_2,L_3\}$ is a round-up triple.
By the nature of $(2223)$ and $(2332)$, there exists at most one point of $L_3$ that is not opposite all points of $L_1$, and at most one point of $L_3$ that is not opposite all points of $L_2$.
Hence we find a point $x_3\in L_3$ that is opposite at least one point $x_1$ of $L_1$ and at least one point $x_2$ of $L_2$.
We can then project $L_i$, $i=1,2$, from $x_i$ onto $x_3$ and obtain lines $L_i'$ and points $y_i\in L_i'$ collinear to a point of $L_i$.
\cref{special} and the uniqueness of the projections yield $L_1'=L_2'=:M_3$ and $y_1=y_2=:y$. 

Let $M_i$, $i=1,2$, be the lines through $y$ intersecting $L_i$.
Our assumptions imply that these lines are pairwise locally opposite at $y$.
By \cref{typeAD}, \cref{typeB} and \cite[Corollary~5.6]{Kas-Mal:13}, $\{M_1,M_2,M_3\}$ is not a round-up triple in the point residual $\Res(y)$.
Hence, up to renumbering, we find a line $M\ni y$ locally opposite $M_1$, and not locally opposite either $M_2$ or $M_3$ at $y$. 
Pick $z\in M\setminus\{y\}$.
Then \cref{joinjoin} implies that $z$ is opposite $x_1$, but not opposite any point of $L_2\cup L_3$.
It follows that each line $K$ through $z$ opposite $L_1$ (which exists) is not opposite either $L_2$ or $L_3$, a contradiction.

This proves the lemma completely. 
\end{proof}

\begin{prop}\label{geomlinegrass} 
Let $T$ be a geometric line of the line-Grassmannian of an exceptional hexagonic Lie incidence geometry $\Delta$ of rank at least $3$.
Then exactly one of the following holds.
\begin{compactenum}[$(i)$]
\item $T$ is an ordinary line of the corresponding line-Grassmannian parapolar space, that is, a planar line pencil of $\Delta$;
\item $\Delta$ is $\mathsf{F_{4,4}}(\K,\K)$ and $T$ is a cone over a hyperbolic line in a symplectic symp. 
\end{compactenum}
\end{prop}

\begin{proof}
If $\Delta$ is not $\mathsf{F_{4,4}}(\K,\K)$, then \cref{RUPtoGL}, together with Lemmas~\ref{RuT0}, \ref{RuT1}, \ref{RuT2}, \ref{2'2'2'2'} and \ref{RuT3} imply $(i)$. 

Suppose now $\Delta\cong\mathsf{F_{4,4}}(\K,\K)$.
If some pair of elements of $T$ is contained in an ordinary line $K$ of the line-Grassmannian of $\Delta$, then, again by the previous lemmas and the fact that every triple of members of $T$ is a round-up triple, all elements are contained in that line, hence all triples are and \cref{RUPtoGL} implies that $T=K$.

Next suppose that two elements $L,M$ of $T$ are not coplanar.
Then, again by Lemmas~\ref{RuT0}, \ref{RuT1}, \ref{RuT2}, \ref{2'2'2'2'} and \ref{RuT3}, they are contained in a hyperbolic line $H$ of the point residual of a symp $\xi$.
The same lemmas now imply that $T\subseteq H$ and \cref{typeB}$(ii)$ yields $T=H$ and $\xi$ is a symplectic polar space.
The proposition is proved. 
\end{proof}

\section{Generalised hexagons}\label{hexagons}

\subsection{Blocking sets}

We start with a nonexistence result of a class of hexagons with certain parameters. 

\begin{lem}\label{nonex}
Let $t$ be a natural number at least $2$.
Then there does not exist a generalised hexagon of order $(s,t)$, with $s=t+t^2$. 
\end{lem}

\begin{proof}
Since by \cite{Fei-Hig:64} the number $st$ is a perfect square, we have that $t^2+t^3$ is a perfect square.
Hence $t+1$ is a perfect square, say $t=a^2-1$.
Then $s=a^2(a^2-1)$.
Now, by \cite{Fei-Hig:64}, we know that the rational number
\[r:=\frac{st(1+s+t+st)(1+\sqrt{st}+st)}{2(s+t+\sqrt{st})}\]
is an integer.
The denominator of that expression is equal to $2(a-1)(a+1)(a^2+a+1)=2t(a^2+a+1)$.
Hence $a^2+a+1$ divides the numerator divided by $t$.
We now observe, taking into account that $a^2+a+1$ is odd, the following facts.

\begin{itemize}
\item Clearly $\gcd(a^2,a^2+a+1)=1$.
\item Since $a^2-1=(a^2+a+1)-(a+2)$ and $a^2+a+1=(a+2)^2-3(a+2)+3$, we find $\gcd(a^2-1,a^2+a+1)\in\{1,3\}$.
\item We have $1+s+t+st=a^2(a^4-a^2+1)$. 
Since
\[a^4-a^2+1=a^2(a^2+a+1)-a(a^2+a+1)-(a^2+a+1)+2(a+1),\]
we find $\gcd(a^4-a^2+1,a^2+a+1)=\gcd(a+1,a^2+a+1)=1$.
\item We have $1+\sqrt{st}+st=a^6-2a^4+a^3+a^2-a+1$, and consequently
\[a^6-2a^4+a^3+a^2-a+1=(a^4-a^3-2a^2+4a-1)(a^2+a+1)-4a+2.\]
\end{itemize}
We conclude that $a^2+a+1$ divides $6a-3$. This implies, since $a\neq 1$, that $a=4$. Nu further information can now lead to a contradiction since, if $a=4$, then  $r=80.667.520\in\mathbb{N}$. 
But in this case, one calculates that the number
\[\frac{st(1+s+t+st)(1-\sqrt{st}+st)}{2(s+t-\sqrt{st})}\]
is not an integer (the denominator is divisible by $13$, whereas this is not the case for the numerator), as is required by \cite{Fei-Hig:64}. 
\end{proof}

Every point $x$ of a generalised hexagon has a projection onto a given line $L$, which is $x$ itself if $x\in L$, which is the unique point of $L$ collinear to $x$ if $x$ is close to $L$, and which is special to $x$ if $x$ is far from $L$.
We call this projection occasionally the \emph{nearest point to $x$ on $L$}.
Also, recall that, if the nearest point to $x$ on $L$ is collinear to but distinct from $x$, then we called $x$ and $L$ close (as we also did above). 

\begin{prop}\label{hex}
If a set of $s+1$ points $S=\{p_0,p_1,\ldots,p_s\}$ of a generalised hexagon $\Delta$ of finite order $(s,t)$, $s,t>1$, admits no opposite point, then it is either
\begin{compactenum}[$(i)$]
\item a line, or
\item a hyperbolic line (and then $s=t$), or
\item a regular distance-$3$ trace (and then $s\geq t$).
\end{compactenum} 
\end{prop}

\begin{proof}
Suppose $S=\{p_0,p_1,\ldots,p_s\}$ is a set of $s+1$ points in $\Delta$ such that no point of $\Delta$ is opposite every point of $S$.
We proceed with proving some claims.

\textbf{Claim 1.} \emph{If $p_0\perp x\perp p_1$, with $p_0$ not collinear to $p_1$, and $x\notin S$, then every line through $x$ contains at least one member of $S$.}\\
Indeed, suppose the line $L$ through $x$ is disjoint from $S$.
Since $p_0$ and $p_1$ project onto the same point $x$ of $L$, there exists some point $y\in L$ not collinear to any member of $S$.
Consider a line $M\neq L$ through $y$.
No point of $S$ is contained in $M$ or is close to $M$ (since this would lead to a $4$-gon or $5$-gon containing $y$).
Hence they are all far from $M$.
Since $p_0$ and $p_1$ project onto the same point $y$ of $M$, there is a point $z\in M$ opposite each member of $S$, a contradiction. 

Claim~1 is proved. 

\textbf{Claim 2.} \emph{If $p_0$ and $p_1$ are collinear, then $S$ is a line of $\Delta$.}\\
Indeed, suppose first that there exists a point $x\in L:=p_0p_1$ that does not coincide with a projection of some member of $S$ onto $L$.
Consider a line $M\neq L$ containing $x$.
Then the only points of $S$ close to $M$ are on $L$.
Since $p_0$ and $p_1$ project onto the same point $x$ of $M$, there exists some point $y\in M$ which is not the projection of any member of $S$ onto $M$.
We deduce that every line $K\neq M$ through $y$ is far from every member of $S$.
Since $p_0$ and $p_1$ project onto the same point $y$ of such a line $K$, there exists a point opposite every member of $S$ on each such line $K$, a contradiction. 

Hence every point $q_i$ on $L$ is the projection of a unique point $p_i$ of $S$.
Suppose $S\neq L$.
Then we may assume that $p_2\notin L$ and so $p_2$ is either special to or opposite $p_0$.
If $p_2$ is special to $p_0$, then by Claim~1 each line through $q_2$ contains a point, say $p_3$, of $S$, implying that $q_2=q_3$, contradicting the uniqueness of $p_2$.
So $p_2$ is far from $L$.
Select a line $M'$ through $q_2$ distinct from $L$ and far from $p_2$ ($M'$ exists since $t\geq 2$).
Then, since $q_2\neq q_i$, the point $p_i$ is either on $L$ or opposite $q_2$, for each $i\in\{3,4,\ldots, s\}$.
It follows that, with $M'$ in the role of $M$ in the previous paragraph, we again reach the same contradiction.

Claim~2 is proved. 
From now on, 
we may assume that $S$ does not contain two collinear points. 

\textbf{Claim 3.} \emph{If $p_0$ and $p_1$ are opposite, and some line $L$ close to both contains no point of $S$, then $L$ is close to each point of $S$ and each point of $L$ is collinear to a unique point of $S$.}\\

Let $x_i$ be the nearest point to $p_i$ on $L$, $i\in\{0,1,\ldots,s\}$, and note that $x_i\neq p_i$ by assumption.
Suppose there exists a point $x\in L\setminus\{x_0,x_1,\ldots,x_s\}$.
Let $M\neq L$ be any line through $x$.
Then $M$ is far from each point of $S$.
Since $x$ is special to at least two points $p_0,p_1$ of $S$, there is some point of $M$ opposite each point of $S$, a contradiction.
Hence $L=\{x_0,x_1,\ldots,x_s\}$.
Suppose some point, say $p_2$, of $S$ is not collinear to its projection $x_2$ onto $L$.
Let $M_2$ be a line through $x_2$ not close to $p_2$ and distinct from $L$ (which exists as $t>1$). 
 
Then $M_2$ is far from each point of $S$, but $p_0$, $p_1$ and $p_2$ have the same projection $x_2$, yielding a point $y_2\in M_2$ opposite each point of $S$, a contradiction.
Hence $p_i\perp x_i$, for all $i\in\{0,1,\ldots,s\}$. 

\textbf{Claim 4.} \emph{If $S$ only contains pairwise opposite points, then $s\geq t$ and $S$ is a regular distance-$3$ trace.}\
The assumptions of Step~3 are satisfied for each line close to both $p_0$ and $p_1$.
Hence every point of $S$ is collinear to some point of each line that is close to both $p_0$ and $p_1$.
We conclude that $S$ is a regular distance-$3$ trace.
We now show that $s\geq t$.

Let $z$ be any point special to both $p_0$ and $p_1$, and not on a line close to $p_0$ and $p_1$.
Note that $z$ is not collinear to any point of $S$, as $z\perp p_2$ would imply, by interchanging the roles of $p_1$ and $p_2$, that the line through $z$ close to $p_0$ is also close to $p_1$, which is not the case by the assumptions on $z$.
Then, similarly as before, every line through $z$ is close to some point of $S$, but not to two such points, as this would mean that $z$ is already collinear to some point of $S$ (by the definition of distance-$3$ trace), contradicting our note above.
We conclude $t\leq s$.

\textbf{Claim 5.} \emph{If $S$ only contains pairwise special points, then $s=t$ and $S$ is a hyperbolic line.}\
Indeed, set $x=[p_0,p_1]$.
By Claim~1, every line through $x$ contains a point of $S$.
Hence $t\leq s$.
If $t=s$, then let $y$ be a point opposite $x$, but not opposite either $p_0$ or $p_1$.
We claim that $p_i$ is not opposite $y$, for every $i\in{2,3,\ldots,s}$.
Indeed, suppose $p_2$ is opposite $y$ and let $L_y$ be the unique line through $y$ not opposite $xp_2$.
No point $p_i$, $i\in\{0,1,\ldots,s\}$, is collinear to some point $q_i$ of $L_y$, as this would induce a $5$-gon containing $x,p_i,q_i$ and the lines $L_y$ and $xp_i$.
Hence all points $p_i$ have a unique point on $L_y$ to which they are not opposite.
But $p_0$ and $p_1$ are not opposite the same point, yielding a point on $L_y$ opposite all members of $S$, a contradiction.
The claim is proved.
This now implies that $S$ is a hyperbolic line.

Suppose now that $t<s$.
Claim~1 implies that $S$ is a $t$-cloud, in the terminology of \cite{Bro-Tha-Mal:02}.
By \cite[Lemma~1]{Bro-Tha-Mal:02} and the remark following Lemma 1 of \cite{Bro-Tha-Mal:02}, it follows that $S\cup S^*$, where $S^*$ is the set of points collinear to at least two points of $S$, is the point set of a subhexagon of order $(1,t)$, and as such $S$ and $S^*$ are the point and line set, respectively, of a projective plane of order $t$.
Hence $s=t^2+t$.
This contradicts \cref{nonex}.

There remains one case to take care of.

\textbf{Claim 6.} \emph{If $S$ contains opposite pairs, then it does not contain special pairs.
}\
Indeed, let $p_0\equiv p_1$.
Suppose, for a contradiction, that $S$ contains a special pair, too.
We first show that every line $L$ close to $p_0$ and $p_1$, respectively, contains a (unique) point of $S$.
Indeed, suppose not.
Then Claim~3 implies that each point of $L$ is collinear to a unique point of $S$, implying that each pair of points of $S$ is opposite, contradicting our assumption.
Hence $L$ contains some point $p_2$ of $S$, unique by Claim~2.
It also follows from Step~1 that each line through $[p_j,p_2]$, $j=0,1$, contains a unique point of $S$.

Now let $T$ be the set of points of $\Delta$ with the property that each line through them contains a point of $S$.
Let $\cL$ be the set of lines of $\Delta$ through such points and note that each member of $\cL$ contains at least one point of $T$ and exactly one point of $S$.
Then we prove that $\Gamma=(S\cup T,\cL)$ is a subhexagon.
Indeed, if $L,L'$ are two distinct lines containing points collinear to $p_0$ and $p_1$, respectively, then $p_0,p_1,L,L'$ are contained in an ordinary hexagon $H$, implying that the girth of the incidence graph of $\Gamma$ is equal to $12$.

In order to show that the diameter of the said graph is $6$, it suffices to prove that for every point $x\in S\cup T$ and every line $L\in\cL$, the unique minimal path joining $x$ and $L$ in $\Delta$ belongs to $\Gamma$.
If $x\in L$, then this is trivial.
Suppose now $x\in M\ni y\in L$, with $x\notin L$.
If $x\in S$, then $y\notin S$ and so $L$ contains a point of $S$ distinct from $y$.
It follows that $y\in T$ and consequently $M\in \cL$.
Suppose now $x\in T$.
Then there exists some point $x'\in S\cap M$.
It again follows that $y\in T$ and $M\in\cL$.
At last suppose $x\in M\ni y\in K\ni z\in L$, with $x\neq y\neq z$ and $M\neq K\neq L$.
If $x\in T$, then there exists $x'\in M\cap S$.
If $x'=y$, then the previous case proves the assertion; if $x'\neq y$, then we replace $x$ with $x'$ and hence we may assume $x\in S$.
If $z\in S$, then $y\in T$ and the assertion follows easily.
So suppose $z\notin S$.
Then some point $q\in L$ different from $z$ belongs to $S$.
By the first paragraph, the line $K$ contains a point of $S$.
It follows that $y,z\in T$ and the assertion follows.
Since it is easy to see that every point of $\Delta$ (and hence of $\Gamma$) is opposite at least one point of $H$, we see that all lines through any point of $S\cup T$ belong to $\cL$.
Hence, by \cite[Lemma~1.3.6]{Mal:98} in combination with \cite[Theorem~1.6.2]{Mal:98}, $\Gamma$ is a subhexagon of order $(s',t)$, $1\leq s'\leq s$.
Since, with the above notation, the line $L$ contains at least three points of $S\cup T$, we have $s'>1$.
Now, by the definition of $\cL$, every member of $\cL$ contains a unique member of $S$.
A standard count reveals that $|S|=1+s't+(s't)^2$ points.
It follows that $s=s't+(s't)^2$.
Now, both $st$ and $s't$ are perfect squares by \cite{Fei-Hig:64}.
It follows that $s/t=s'(1+s't)$ is a perfect square.
Since $s'$ and $1+s't$ are relatively prime, both $s'$ and $1+s't$ are perfect squares, which contradicts the fact that $s't$ is a perfect square.
This completes the proof of Claim~5.

This also completes the proof of the proposition.
\end{proof}

The following result classifies very explicitly all sets of size $s+1$ admitting no global opposite point in finite Moufang hexagons of order $(s,t)$.

\begin{coro}\label{BShex}
A set $S$ of $s+1$ points of a Moufang generalised hexagon of finite order $(s,t)$, $s,t>1$, admits no opposite point if, and only if, it is either
\begin{compactenum}[$(i)$] \item a line, or \item a hyperbolic line in the split Cayley hexagon $\mathsf{G_{2,2}}(s,s)$, or \item a distance-$3$ trace in the split Cayley hexagon $\mathsf{G_{2,2}}(s,s)$ with $s$ even, or in the twisted triality hexagon $\mathsf{G_{2,2}}(t,s)$, with $t^3=s$ even.
\end{compactenum}
\end{coro}

\begin{proof}
In view of \cref{hex}, $(ii)$ follows from \cite[Remark~6.3.5]{Mal:98} and $(iii)$ follows from \cite[Theorem~1]{Gov:97}.
\end{proof}

Note that the previous corollary implies that not every regular distance-$3$ trace is a set of $s+1$ points such that no point is opposite each point of that set.
Indeed, the split Cayley hexagons and the twisted triality hexagons in odd characteristic are counterexamples.

\begin{remark}
In Claim~5 of the proof of \cref{hex}, hexagons of order $(t^2+t,t)$ appear as possible counterexamples, but, as we assume thickness, they are killed by \cref{nonex}.
If we drop the thickness assumption, it is curious to note that the arguments of that step give rise to a rather exceptional example of a set $F$ of $s+1$ point-line flags in a projective plane of order $s$ such that no point-line flag is opposite all members of $S$.
Indeed, putting $t=1$, we obtain a hexagon of order $(2,1)$, which arises from $\PG(2,2)$, and $S$ consists of three flags $\{p_0,p_0p_1\},\{p_1,p_1p_2\},\{p_2,p_0p_2\}$ from a triangle $\{p_0,p_1,p_2\}$.
The points of these flags do not form a line, and the lines of these flags do not form a line pencil.
We conjecture that this is the only example of size $s+1$ in any projective plane of order $s$ with that property and such that no flag of the plane is opposite all of its members.
\end{remark}

\subsection{Geometric lines}

We now classify geometric lines in Moufang hexagons.
We first consider the general case and then specify further.
As usual, we deal with round-up triples. 

\begin{lem}\label{RUP1hex}
Let $\Gamma=(X,\cL)$ be a generalised hexagon and $\{x_1,x_2,x_3\}$ a round-up triple of points.
Suppose $x_1\perp x_2$.
Then $x_1,x_2,x_3$ are contained in a common line.
\end{lem}

\begin{proof}
This immediately follows from \cref{apartment}.
\end{proof}

\begin{lem}\label{RUP2hex}
Let $\Gamma=(X,\cL)$ be a generalised hexagon and $\{x_1,x_2,x_3\}$ a round-up triple of points.
Suppose $x_1\Join x_2$.
Then $x_1,x_2,x_3$ are contained in $[x_1,x_2]^\perp\cap y^{\Join}$, for every point $y\in[x_1,x_2]^\equiv\cap x_1^{\Join}\cap x_2^{\Join}$.
\end{lem}

\begin{proof}
By \cref{RUP1hex} we have $x_3\neq[x_1,x_2]$.
Now \cref{apartment2} yields $x_3\perp[x_1,x_2]$ and the assertion follows directly from the definition of a round-up triple.
\end{proof}

\begin{lem}\label{RUP3hex}
Let $\Gamma=(X,\cL)$ be a generalised hexagon and $\{x_1,x_2,x_3\}$ a round-up triple of points.
Suppose $x_1\equiv x_2$.
Then $x_1,x_2,x_3$ are contained in every distance-$3$ trace containing at least two of them.

\end{lem}

\begin{proof}
Let $L$ be an arbitrary line of $\Gamma$ containing a point $x_1'$ collinear to $x_1$ and also a point $x_2'$ collinear to $x_2$ (then $x_1'\neq x_2'$).
By \cref{RUP2hex}, $x_3\notin L$.
Assume $x_3$ is special to some point $y\in L$ with $[x_3,y]\notin L$.
Then any point of $L\setminus\{y\}$ is opposite $x_3$ and not opposite both $x_1,x_2$.
So $x_3$ is collinear to some point of $L$, distinct from both $x_1'$ and $x_2'$ (use \cref{RUP2hex} again).
Now it is clear that the assertion follows.
\end{proof}

\begin{prop}\label{geomlinesghex}
Let $\Gamma=(X,\cL)$ be a generalised hexagon and let $T$ be a geometric line of $\Gamma$.
Then $T$ is either a line, a hyperbolic line, or a regular distance-$3$ trace.

\end{prop}

\begin{proof}
Since every triple of points of a geometric line is a round-up triple, the previous three lemmas imply that $T$ is contained in either a line, or a hyperbolic line, or a distance-$3$ trace.
But if $T$ were not equal to one of these objects, then, in each case, it is easy to find a point opposite every member of $T$, a contradiction.
\end{proof}

We can now prove \cref{B} for type $\mathsf{G_2}$.

\begin{prop}\label{geomlineshex}
Let $\Gamma=(X,\cL)$ be a Moufang generalised hexagon and let $T$ be a geometric line.
Then $T$ is either
\begin{compactenum}[$(1)$]
\item an ordinary line, or
\item a hyperbolic line in a split Cayley hexagon, or
\item a distance-$3$ trace in a split Cayley hexagon over a perfect field in characteristic~$2$.
\end{compactenum}
\end{prop}
\begin{proof}
By \cref{geomlinesghex} there are three possibilities for $T$.
The first one is a line, which leads to $(1)$.
The second one is a hyperbolic line.
Let $T$ be collinear to the unique point $c$.
Since $T$ is a geometric line, every hyperbolic line in $c^\perp$ intersects $T$ in exactly one point.
By transitivity of the automorphism group on paths $x_1\perp x_2\perp x_3$, with $x_1\Join x_3$, which follows readily from the Moufang condition, we see that every pair of hyperbolic lines in $c^\perp$ intersects.
Then \cite[Corollary~5.14]{Mal:98} implies that $\Gamma$ is a split Cayley hexagon.

The third possibility is that $T$ is a regular distance-$3$ trace.
Let $x,y\in T$.
Then we obviously can write $T=\{x,y\}^{\not\equiv\not\equiv}$.
Let $L\in\cL$ be arbitrary but such that it contains unique points $x'$ and $y'$ special to $x$ and $y$, respectively, with $x'\neq y'$.
That at least one such line exists is easily seen.
Now every point of $L$ is special to precisely one point of $T$, since $T$ is a geometric line.
This means that, in the terminology of \cite[Definitions~6.5.5]{Mal:98}, the set $T$ is a long imaginary line, and \cite[Theorem~6.5.6]{Mal:98} now implies that $\Gamma$ is a split Cayley hexagon over a perfect field in characteristic~$2$.
\end{proof}

\begin{remark}\label{rem1}
For every natural number $n\geq 5$, there exists an (obvious) analogue of \cref{geomlinesghex} for the class of (thick) generalised $n$-gons. This requires defining a ``regular'' distance-$i$ trace, $2\leq i\leq \frac{n}{2}$, similarly to a hyperbolic line (which would be a regular distance-$2$ trace) and a regular distance $3$-trace for a generalised hexagon. Proofs are straightforward generalisations of the above proofs for hexagons. Restricting to Moufang octagons (the only class of exceptional Moufang buildings not yet considered in this paper), one obtains that, using the results in \cite{Bon-Cuy-Mal:96} (see also \cite[Section~6.5]{Mal:98}), the only geometric lines in Moufang octagons are the ordinary lines.  However, the classification of minimal blocking sets in finite Moufang octagons is still open; however, see also \cref{rem2}. 
\end{remark}

\begin{remark}\label{rem2}
S. Petit and G. Van de Voorde \cite[Theorem~6]{Pet-Voo:26} prove that, if $s\leq t$, then every blocking set of $s+1$ points in a finite generalised polygon of order $(s,t)$ is either a line or a regular distance-$i$ trace. Together with \cref{rem1}, this leads to a classification of blocking sets of size $s+1$ in the Moufang octagons of order $(s,s^2)$: only lines occur. The case of Moufang octagons of order $(s,\sqrt{s})$ is hence the only open case for finite Moufang polygons. Note that \cref{hex} extends  \cite[Theorem~6]{Pet-Voo:26} for generalised hexagons to arbitrary order. 
\end{remark}

\section{An application}\label{anappl}
\subsection{Surjective oppositon preserving maps.}
\begin{prop}\label{appl}
Let $\Delta$ and $\Delta'$ be two buildings of the same exceptional type $\mathsf{F_4, E_6, E_7}$ or $\mathsf{E_8}$.
Let, with Bourbaki labelling, $T_i$ and $T_i'$ be the set of vertices of type $i$ of $\Delta$ and $\Delta'$, respectively, where \begin{compactenum}[$\bullet$]
\item $i\in\{1,2,3,4\}$ is arbitrary if $\Delta$ has type $\mathsf{F_4}$;
\item $i\in\{1,2,3,4,5,6\}$ is arbitrary if $\Delta$ has type $\mathsf{E_6}$;
\item $i\in\{1,3,6,7\}$ if $\Delta$ has type $\mathsf{E_7}$;
\item $i\in\{7,8\}$ if $\Delta$ has type $\mathsf{E_8}$.
\end{compactenum}
Then any surjective map $\varphi \colon T_i\to T_i'$ preserving opposition and non-opposition is induced by an isomorphism of buildings. 
\end{prop}

\begin{proof}
It suffices to show that $\varphi$ is a bijective collineation between the corresponding $i$-Grassmannian geometries.
We first show that $\varphi$ is bijective.
Suppose, for a contradiction, that two vertices $v,u\in T_i$ are mapped onto the same vertex.
\cref{3.30b} yields a vertex $w\in T_i$ opposite $v$ but not opposite $u$.
Then our assumptions imply $\varphi(v)\equiv\varphi(w)\not\equiv\varphi(u)=\varphi(v)$, a contradiction.
Hence $\varphi$ is a bijection.
Since opposition and non-opposition are preserved, one deduces that geometric lines are mapped onto geometric lines.
If the only geometric lines are the ordinary lines, then this concludes the proof of the proposition.

By \cite[Corollary 6.6]{Kas-Mal:13} and \cref{geomlinegrass}, we may assume $i=3$ and $\Delta$ has type $\mathsf{F_4}$.
It suffices to recognise the planar line pencils of $\mathsf{F_{4,4}}(\K,\K)$ among all geometric lines of $\mathsf{F_{4,3}}(\K,\K)$.
Let $\Gamma$ be the point-line geometry with point set the points of $\mathsf{F_{4,3}}(\K,\K)$ and line set the set of ordinary and geometric lines of $\mathsf{F_{4,3}}(\K,\K)$.
We claim that no geometric line different from an ordinary line of $\mathsf{F_{4,3}}(\K,\K)$ is contained in a maximal subspace of $\Gamma$  isomorphic to a projective plane.
Suppose, for a contradiction, that the geometric line $Z$ is contained in a maximal singular subspace $\alpha$ of $\Gamma$ isomorphic to a projective plane, and that $Z$ is not an ordinary line of $\mathsf{F_{4,3}}(\K,\K)$.
We argue in $\mathsf{F_{4,4}}(\K,\K)$, where $Z$ is a cone in a symp with vertex $p$ over a hyperbolic line $h$.
Suppose first that $\alpha$ does not contain any ordinary line of  $\mathsf{F_{4,3}}(\K,\K)$.
So, the point set of $\alpha$ corresponds to a set $\Pi$ of lines through $p$, and the cones over hyperbolic lines correspond to the lines of $\alpha$.
Consequently, any two points on distinct lines of $\Pi$ are symplectic, and the unique hyperbolic line through them is contained in the union of all lines of $\Pi$.
We select a point $q$ opposite $p$.
Then every line $L\in\Pi$ contains a unique point $p_L$ special to $q$.

We claim that the set $\beta=\{p_L\mid L\in\Pi\}$, endowed with the hyperbolic lines contained in it, is a projective plane.
Indeed, in view of the fact that $\Pi$ is the point set of a projective plane whose lines are geometric lines of $\mathsf{F_{4,3}}(\K,\K)$, it suffices to prove that $\beta$ is closed under taking hyperbolic lines through two arbitrary distinct points $y_1$ and $y_2$ of $\beta$.
Since  $p\in\xi(y_1,y_2)$, $q$ is far from $\xi(y_1,y_2)$.
Since $y_1$ and $y_2$ are collinear to the unique point $q'$ of $\xi(y_1,y_2)$ symplectic to $q$, all points of the hyperbolic line $h(y_1,y_2)$ defined by $y_1,y_2$ are collinear to $q'$, by the very definition of hyperbolic line.
The claim is proved.

Now \cite[Lemma~5.21]{Sch-Sas-Mal:18} implies that $\beta$ is contained in an extended equator geometry $\widehat{E}$.
Then \cite[Proposition 5.24]{Sch-Sas-Mal:18} implies that $p$ is collinear to a set $\gamma$ of points of $\widehat{E}$ that forms a $3$-dimensional projective space when endowed with the hyperbolic lines it contains.
Hence the line set  $\{px\mid x\in\gamma\}$ forms a projective $3$-space in $\Gamma$.
So, $\alpha$ is not maximal, a contradiction.

Consequently, we may suppose that $\alpha$ contains at least one ordinary line of $\mathsf{F_{4,3}}(\K,\K)$.
Then we have a plane $\pi$ of $\mathsf{F_{4,4}}(\K,\K)$ through $p$ intersecting $h$ in some point $x$.
Select $y\in h\setminus\{x\}$ and $z\in \pi\setminus px$.
Let $\xi$ be the symp containing $h$; we have $p\in \xi$.
Suppose, for a contradiction, that $z$ is not contained in $\xi$.
Since $z\perp px$ and $y\notin px$, we deduce $z\pperp y$.
Hence $y\perp px$, a contradiction.
Consequently $z\in\xi$ and so $\pi\subseteq\xi$.
Now the set of lines of $\xi$ through $p$ forms a $3$-dimensional projective space of $\Gamma$, contradicting the maximality of $\alpha$. 

Hence $Z$ is not contained in a maximal singular subspace of $\Gamma$ isomorphic to a projective plane.
Evidently, any line of $\mathsf{F_{4,4}}(\K,\K)$ is contained in an  ordinary projective plane of $\mathsf{F_{4,4}}(\K,\K)$, which gives rise to a maximal singular subspace of $\Gamma$ of dimension~$2$.
Hence we can recognise the ordinary lines of $\mathsf{F_{4,3}}(\K,\K)$, and the proof of the proposition is complete.
\end{proof}

The following is an immediate consequence.

\begin{coro}
Let $\Delta$ be a finite building of exceptional type $\mathsf{F_4, E_6, E_7}$ or $\mathsf{E_8}$.
Let, with Bourbaki labelling, $T_i$ be the set of vertices of type $i$ of $\Delta$, where \begin{compactenum}[$\bullet$]
\item $i\in\{1,2,3,4\}$ is arbitrary if $\Delta$ has type $\mathsf{F_4}$;
\item $i\in\{1,2,3,4,5,6\}$ is arbitrary if $\Delta$ has type $\mathsf{E_6}$;
\item $i\in\{1,3,6,7\}$ if $\Delta$ has type $\mathsf{E_7}$;
\item $i\in\{7,8\}$ if $\Delta$ has type $\mathsf{E_8}$.
\end{compactenum}
Then any map $\varphi:T_i\to T_i'$ preserving opposition and non-opposition is induced by an automorphism of $\Delta$. 
\end{coro}

\begin{proof}
We only have to establish the surjectivity of $\varphi$ in order to be able to apply \cref{appl}.
Therefore, we note that the injectivity of $\varphi$ is proved in a completely similar way as in the first paragraph of the proof of \cref{appl}.
 Now the assertion follows from the trivial fact that an injective transformation of a finite set is always surjective. 
\end{proof}

\subsection{Further outlook} We mention a possible further application of the fact that opposition of points in long root subgroup geometries determines collinearity. The notion of \emph{imaginary lines} mentioned in \cref{secgenhex} can be extended to all long root subgroup geometries. In group-theoretic terms, imaginary lines are the orbits entirely consisting of mutually opposite points of so-called \emph{fundamental $\SL_2$s}. Taking the imaginary lines of some long root subgroup geometry $\Delta$ as set of vertices of a graph $\Gamma$, where two vertices are adjacent if the corresponding $\SL_2$s commute, one obtains what we could call the \emph{imaginary graph of $\Delta$}. These graphs were considered in \cite{Alt-Gra:10,Gra:04} (and we thank a referee to point us to these papers) for some finite buildings of classical type $\mathsf{C}_n$ (associated to symplectic and unitary groups). In particular, the imaginary graphs for these groups were characterised by their local graphs (which are also imaginary graphs), and this lead to a group theoretic characterisation of the corresponding groups. A large part of the proof in \cite{Alt-Gra:10} consisted in building the long root subgroup geometry, that is, defining the points and the lines of it. If one liked to do a similar job for the exceptional cases, then, once the points are defined in terms of the vertices of the graph, the lines would follow since we know opposition: two points are opposite if they are contained in a common imaginary line, hence somehow related to a common vertex of the graph. Perhaps this also yields a simplification of the proofs in \cite{Gra:04,Alt-Gra:10}, but this should be first looked at in some more detail. We defer this to the future.  

\textbf{Acknowledgment.} The authors are grateful to two referees whose remarks and suggestions made the paper and some arguments more clear.

\textbf{Statement}: On behalf of all authors, the corresponding author states that there is no conflict of interest.

\textbf{Data}: No data were used or generated for this article.

\end{document}